\documentclass[article]{amsart}
\usepackage{enumerate}
\usepackage{enumitem}
\usepackage{amsfonts,amssymb,amsmath,amsthm}
\usepackage{epsfig}
\usepackage{graphics}
\usepackage[normalem]{ulem}
\usepackage{color}
\usepackage{comment}
\usepackage{stmaryrd} 
\usepackage{url}
\usepackage{overpic}
\usepackage{mathrsfs}

\usepackage{tikz-cd}
\usetikzlibrary{arrows.meta,bending, positioning, shapes} 
\usetikzlibrary{automata,decorations.pathmorphing,arrows}
\usetikzlibrary{positioning}
\usetikzlibrary{calc}

\usepackage[pdftex]{hyperref}

\input xy 
\xyoption{all}
\numberwithin{equation}{section}

\definecolor{OrangeRed}{cmyk}{0,0.6,1,0}            
\definecolor{DarkBlue}{cmyk}{1,1,0,0.20}
\definecolor{DarkGreen}{cmyk}{1,0,0.6,0.2}
\definecolor{myblue}{rgb}{0.66,0.78,1.00}
\definecolor{Violet}{cmyk}{0.79,0.88,0,0}
\definecolor{Lavender}{cmyk}{0,0.48,0,0}

\newcounter{main}

\theoremstyle{plain}
        \newtheorem{theorem}{Theorem}[section]
        \newtheorem*{theorem*}{Theorem}
        \newtheorem*{conj*}{Conjecture}
        \newtheorem{lemma}[theorem]{Lemma}
        \newtheorem{corollary}[theorem]{Corollary}
        \newtheorem{prop}[theorem]{Proposition}

\theoremstyle{definition}
        \newtheorem{definition}[theorem]{Definition}
        \newtheorem*{definition*}{Definition}

\theoremstyle{remark}
        \newtheorem*{remark}{Remark}
        
        \newtheorem{rem}[theorem]{Remark}
        \newtheorem{example}[theorem]{Example}

        \newtheorem{question}{Question}
        
        \newtheorem*{example*}{Example}
        \newtheorem*{examples*}{Examples}        
        \newtheorem*{claim}{Claim}
        \newtheorem*{fact}{Fact}

        \newtheorem*{convention}{Convention}

\DeclareMathAlphabet{\mathbbmsl}{U}{bbm}{m}{sl}

\def\C{\mathbb{C}}
\def\P{\mathbb{P}}

\newcommand{\abs}[1]{| #1 |}

\def\bcases{\begin{cases}}

\def\ecases{\end{cases}}

\newcommand{\N}{\mathbb N}

\newcommand{\R}{\mathbb R}

\newcommand{\Z}{\mathbb Z}

\newcommand{\bea}{\begin{eqnarray*}}
\newcommand{\eea}{\end{eqnarray*}}

\newcommand{\be}{\begin{equation}}
\newcommand{\ee}{\end{equation}}

\newcommand{\MM}{\mathcal{M}}
\newcommand{\RR}{\mathcal{R}}
\newcommand{\GG}{\mathcal{G}}
\newcommand{\dist}{\operatorname{dist}}
\renewcommand{\epsilon}{\varepsilon}
\renewcommand{\phi}{\varphi}

\newcommand{\xbf}{{\bf x}}
\newcommand{\Xbf}{{\bf X}}
\newcommand{\ebf}{{\bf e}}
\newcommand{\obf}{{\bf 0}}
\newcommand{\zbf}{{\bf z}}
\newcommand{\ybf}{{\bf y}}

\newcommand{\wbf}{{\bf w}}
\newcommand{\ibf}{{\bf 1}}
\newcommand{\F}{\mathbbmsl{F}}
\newcommand{\IS}{\mathrm{I}}

\begin{document}

\title{Zeros of the independence polynomial on recursive sequences of graphs}

\author{Mikhail Hlushchanka}
\address{Korteweg-de Vries Instituut voor Wiskunde, Universiteit van Amsterdam,  1090 GE \newline Amsterdam, The Netherlands}
\email{mikhail.hlushchanka@gmail.com}

\author{Han Peters}
\address{Korteweg-de Vries Instituut voor Wiskunde, Universiteit van Amsterdam,  1090 GE \newline Amsterdam, The Netherlands}
\email{h.peters@uva.nl}

\begin{abstract}
    We study the hard-core model on recursively defined sequences $(G_n)_{n\geq0}$ of graphs with a fixed number $k\geq 1$ of labeled vertices in each graph. The next graph in the sequence is constructed by taking a fixed number $m\geq 2$ of copies of the previous graph, connecting these copies by identifying some labeled vertices according to a fixed rule, and afterward choosing $k$ labeled vertices in the resulting graph, again in accordance with a fixed rule.  Examples of such sequences include the Sierpi\'nski gasket graphs, hierarchical lattices, and many more. We prove that, when the vertex degrees of the graphs $G_n$ are uniformly bounded and the distances between the labeled vertices in $G_n$ diverge, the complex zeros of the univariate independence polynomials $Z_{G_n}(\lambda)$ avoid a neighborhood of the non-negative real axis. By the Lee--Yang theory this implies that no phase transitions occur for the hard-core model on these recursive sequences of graphs, independently of the starting graph $G_0$.

    The proof relies on the study of the dynamical properties of a one-parameter family of rational maps $F_\lambda$ on the $(2^k-1)$-dimensional complex projective space induced by the graph recursion operator. The dynamical framework developed in this paper can be naturally extended to other classical models in statistical mechanics (such as the Ising or Potts models) and to more general notions of graph recursions.
\end{abstract}

\maketitle
\tableofcontents

\section{Introduction}

\subsection{Main results.} 

The study of phase transitions in statistical mechanics is intimately connected with the study of the zeros of the appropriate partition functions. In this paper, we focus on the independence polynomial, which is the partition function of the \emph{hard-core model} (see Section~\ref{subsec: statistical physics}). Formally, the \emph{independence polynomial}
$Z_G(\lambda)$ of a finite graph $G$ is the generating function for \emph{independent sets}, i.e., subsets of pairwise non-adjacent vertices, in the graph $G$: the coefficient of $\lambda^n$ in $Z_G(\lambda)$ represents the number of independent sets of size $n$. This polynomial is a classical graph invariant in enumerative combinatorics, and its zeros play an important role not only in statistical physics, but also in problems from probabilistic combinatorics and theoretical computer science; see, for example, \cite{ScottSokal05, PatelRegts, BezakovaEtAl, BoerEtAlApprox}. It is known, in particular, that the roots of the independence polynomial are dense outside a neighborhood of the origin for the family of bounded-degree graphs \cite{BoerEtAlApprox,BezakovaEtAl}. However,
the overall structure of this zero locus, as well as the corresponding zero locus for concrete sequences of
graphs approximating some infinite graph or converging to some limit space, remains far less understood.

The goal of this paper is to develop a unified dynamical framework for the study of zeros of the independence polynomials $Z_{G_n}$ for sequences $(G_n)_{n\geq 0}$ of finite graphs that are defined recursively, i.e., $G_{n+1} = \RR(G_n)$. The recursion operator $\RR$ acts on finite graphs $G$ with a fixed number $k \in \mathbb N$ of distinct labeled vertices, by first taking a fixed number $m \ge 2$ of copies of $G$, then identifying some of the identically labeled vertices in these copies, and finally assigning distinct labels to $k$ previously labeled vertices in the resulting graph, all according to a fixed rule. An example of such a recursive procedure inspired by the Sierpi\'nski gasket fractal is illustrated in Figure~\ref{fig: sierpinski tripod graphs}; see also Figure~\ref{fig: sierpinski graphs} for an illustration of the same recursive construction with a different starting graph.

We now state our main result.

\begin{theorem}\label{main result zeros}
    Let $G_0$ be a starting graph with $k$ labeled vertices, and let $\RR$ be a recursion operator defining a sequence of graphs $(G_n)_{n\geq 0}$. Suppose that the graph sequence satisfies the following two conditions:
    \begin{enumerate}[label={\normalfont (\roman*)}]
        \item\label{item: non-deg intro} the maximal vertex degrees of the graphs $G_n$ are uniformly bounded, and 
        \item\label{item: exp intro} the minimal distance between pairs of distinct labeled vertices in $G_n$ diverges as $n \rightarrow \infty$. 
    \end{enumerate}
    Then the zeros of the independence polynomials $Z_{G_n}$ avoid a uniform neighborhood of $\R_{\geq 0}$. 
\end{theorem}

\begin{figure}
\begin{tikzpicture}[every node/.style={font=\small}]

\def\h{0.866}
\def\x{1}

\begin{scope}
\draw[black, thick] (0,0) -- (\x/2,\x*\h/3) ;
\draw[black, thick] (\x,0) -- (\x/2,\x*\h/3) ;
\draw[black, thick] (\x/2,\x*\h) -- (\x/2,\x*\h/3);

\filldraw[black] (0,0) circle (2.4pt) node[below left] {$\mathbf{1}$};
\filldraw[black] (\x,0) circle (2.4pt) node[below right] {$\mathbf{2}$};
\filldraw[black] (\x/2,\x*\h) circle (2.4pt) node[above=2pt] {$\mathbf{3}$};
\filldraw[fill = gray] (\x/2,\x*\h/3) circle (1.6pt);

\node[font=\normalsize] at (\x/2,0) [below=15pt] {$G_0$};
\end{scope}

\begin{scope}[xshift=4cm]
\draw[DarkGreen, thick] (0,0) -- (\x/2,\x*\h/3) ;
\draw[DarkGreen, thick] (\x,0) -- (\x/2,\x*\h/3) ;
\draw[DarkGreen, thick] (\x/2,\x*\h) -- (\x/2,\x*\h/3);

\draw[blue, thick] (\x,0) -- (3/2*\x,\x*\h/3) ;
\draw[blue, thick] (2*\x,0) -- (3/2*\x,\x*\h/3) ;
\draw[blue, thick] (3/2*\x,\x*\h) -- (3/2*\x,\x*\h/3);

\draw[red, thick] (\x/2,\h) -- (\x,4/3*\h) ;
\draw[red, thick] (3/2*\x,\h) -- (\x,4/3*\h) ;
\draw[red, thick] (\x,2*\h) -- (\x,4/3*\h);

\filldraw[black] (0,0) circle (2.4pt) node[below left] {$\mathbf{1}/1$};
\filldraw[black] (2*\x,0) circle (2.4pt) node[below right] {$\mathbf{2}/2$};
\filldraw[black] (\x,2*\x*\h) circle (2.4pt) node[above] {$\mathbf{3}/3$};

\filldraw[black] (\x/2,\h) circle (1.6pt) node[left=1.5pt] {$2$};
\filldraw[black] (\x,0) circle (1.6pt) node[below=1.5pt] {$3$};;
\filldraw[black] (3/2*\x,\h) circle (1.6pt) node[right=1.5pt] {$1$};
\filldraw[fill =  gray] (\x/2,\h/3) circle (1.6pt);
\filldraw[fill = gray] (\x,4/3*\h) circle (1.6pt);
\filldraw[fill = gray] (3/2*\x,\h/3) circle (1.6pt);

\node[font=\normalsize] at (\x,0) [below=15pt] {$G_1$};
\end{scope}

\begin{scope}[xshift=10cm]
\draw[DarkGreen, thick] (0,0) -- (\x/2,\x*\h/3) ;
\draw[DarkGreen, thick] (\x,0) -- (\x/2,\x*\h/3) ;
\draw[DarkGreen, thick] (\x/2,\x*\h) -- (\x/2,\x*\h/3);

\draw[DarkGreen, thick] (\x,0) -- (3/2*\x,\x*\h/3) ;
\draw[DarkGreen, thick] (2*\x,0) -- (3/2*\x,\x*\h/3) ;
\draw[DarkGreen, thick] (3/2*\x,\x*\h) -- (3/2*\x,\x*\h/3);

\draw[DarkGreen, thick] (\x/2,\h) -- (\x,4/3*\h) ;
\draw[DarkGreen, thick] (3/2*\x,\h) -- (\x,4/3*\h) ;
\draw[DarkGreen, thick] (\x,2*\h) -- (\x,4/3*\h);

\begin{scope}[xshift=1cm,yshift=2*\h cm]
\draw[red, thick] (0,0) -- (\x/2,\x*\h/3) ;
\draw[red, thick] (\x,0) -- (\x/2,\x*\h/3) ;
\draw[red, thick] (\x/2,\x*\h) -- (\x/2,\x*\h/3);

\draw[red, thick] (\x,0) -- (3/2*\x,\x*\h/3) ;
\draw[red, thick] (2*\x,0) -- (3/2*\x,\x*\h/3) ;
\draw[red, thick] (3/2*\x,\x*\h) -- (3/2*\x,\x*\h/3);

\draw[red, thick] (\x/2,\h) -- (\x,4/3*\h) ;
\draw[red, thick] (3/2*\x,\h) -- (\x,4/3*\h) ;
\draw[red, thick] (\x,2*\h) -- (\x,4/3*\h);
\end{scope}

\draw[blue, thick] (2*\x,0) -- (5/2*\x,\x*\h/3) ;
\draw[blue, thick] (3*\x,0) -- (5/2*\x,\x*\h/3) ;
\draw[blue, thick] (5/2*\x,\x*\h) -- (5/2*\x,\x*\h/3);

\draw[blue, thick] (3*\x,0) -- (7/2*\x,\x*\h/3) ;
\draw[blue, thick] (4*\x,0) -- (7/2*\x,\x*\h/3) ;
\draw[blue, thick] (7/2*\x,\x*\h) -- (7/2*\x,\x*\h/3);

\draw[blue, thick] (5/2*\x,\h) -- (3*\x,4/3*\h) ;
\draw[blue, thick] (7/2*\x,\h) -- (3*\x,4/3*\h) ;
\draw[blue, thick] (3*\x,2*\h) -- (3*\x,4/3*\h);

\begin{scope}[xshift=0cm]
\filldraw[fill = gray] (\x/2,\h) circle (1.6pt);
\filldraw[fill = gray] (\x,0) circle (1.6pt);
\filldraw[fill = gray] (3/2*\x,\h) circle (1.6pt);
\end{scope}

\begin{scope}[xshift=2*\x cm]
\filldraw[fill = gray] (\x/2,\h) circle (1.6pt);
\filldraw[fill = gray] (\x,0) circle (1.6pt);
\filldraw[fill = gray] (3/2*\x,\h) circle (1.6pt);
\end{scope}

\begin{scope}[xshift=\x cm, yshift=2*\h cm]
\filldraw[fill = gray] (\x/2,\h) circle (1.6pt);
\filldraw[fill = gray] (\x,0) circle (1.6pt);
\filldraw[fill = gray] (3/2*\x,\h) circle (1.6pt);
\end{scope}

\filldraw[black] (2*\x,0) circle (1.6pt) node[below=1.5pt] {$3$};
\filldraw[black] (\x,2*\x*\h) circle (1.6pt) node[left=1.5pt] {$2$};
\filldraw[black] (3*\x,2*\x*\h) circle (1.6pt) node[right=1.5pt] {$1$};

\filldraw[black] (0,0) circle (2.4pt) node[below left] {$\mathbf{1}/1$};
\filldraw[black] (4*\x,0) circle (2.4pt) node[below right] {$\mathbf{2}/2$};
\filldraw[black] (2*\x,4*\x*\h) circle (2.4pt) node[above] {$\mathbf{3}/3$};

\filldraw[fill = gray] (\x/2,\h/3) circle (1.6pt);
\filldraw[fill = gray] (\x,4/3*\h) circle (1.6pt);
\filldraw[fill = gray] (3/2*\x,\h/3) circle (1.6pt);

\begin{scope}[xshift=1cm,yshift=2*\h cm]
\filldraw[fill = gray] (\x/2,\h/3) circle (1.6pt);
\filldraw[fill = gray] (\x,4/3*\h) circle (1.6pt);
\filldraw[fill = gray] (3/2*\x,\h/3) circle (1.6pt);
\end{scope}

\begin{scope}[xshift=2cm]
\filldraw[fill = gray] (\x/2,\h/3) circle (1.6pt);
\filldraw[fill = gray] (\x,4/3*\h) circle (1.6pt);
\filldraw[fill = gray] (3/2*\x,\h/3) circle (1.6pt);
\end{scope}

\node[font=\normalsize] at (2*\x,0) [below=15pt] {$G_2$};
\end{scope}

\end{tikzpicture}
    \caption{Sierpi\'{n}ski tripod graphs $(G_{n})_{n=0,1,2}$. The thicker black vertices with their bold labels correspond to the marked vertices in each graph, labeled $1,2,3$. The black vertices all together with their normal font labels correspond to the marked vertices of the copies of the graph from the previous step. For $n=1,2$, the edges of $G_{n+1}$ induced by different copies of the graph $G_n$ are shown in different colors.}
    \label{fig: sierpinski tripod graphs}
\end{figure}
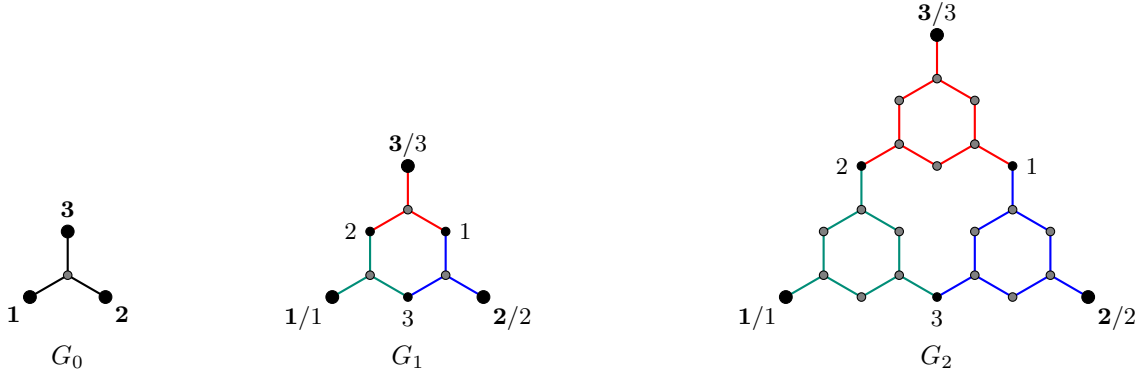

In the theorem above and in the rest of the paper, ``uniform'' means ``independent of $n$''. The following is a quick consequence of our main result, combined with Lemma~\ref{lem: limit_energy}.

\begin{corollary}
    Let $(G_n)_{n\geq 0}$ be a sequence of graphs as in Theorem~\ref{main result zeros}. Then the limiting free energy per site 
    $$
\rho(\lambda) = \lim_{n\to\infty}\frac{\log Z_{G_n}(\lambda)}{\#V(G_n)}
$$
    is well defined and real-analytic on all of $\mathbb R_+$, that is, there are no phase transitions for the hard-core model on $(G_n)_{n\geq 0}$.
\end{corollary}

We emphasize that the uniform zero-free neighborhood of the positive real axis in Theorem~\ref{main result zeros} depends on the starting graph $G_0$. For example, the independence polynomials of the Sierpi\'nski tripod graphs from Figure~\ref{fig: sierpinski tripod graphs} have complex zeros with positive real part, while for the Sierpi\'nski gasket graphs from Figure~\ref{fig: sierpinski graphs}, it is known that all zeros are real and negative due to a result of Chudnovsky and Seymour \cite{CS_clawfree}.

When the starting graph $G_0$ is connected, the assumptions~\ref{item: non-deg intro} and~\ref{item: exp intro} in Theorem~\ref{main result zeros} depend only on the recursion operator $\RR$, although the actual bound on the vertex degrees in~\ref{item: non-deg intro} also depends on the graph $G_0$. We will say that the operator $\RR$ is \emph{non-degenerate} if for some connected starting graph $G_0$ (and thus for any starting graph) the vertex degrees in the graphs $G_n=\RR^n(G_0)$ remain uniformly bounded. If for some connected starting graph $G_0$ (and thus for any starting graph) the minimal distance between the labeled vertices in $G_n$ diverges, we will say that $\RR$ is \emph{expanding}. 

The recursion operator $\RR$ induces a renormalization map on a multi-dimensional complex projective space. The proof of Theorem~\ref{main result zeros} relies on the analysis of the dynamical properties of this renormalization map (see Theorem~\ref{thm: dynamics}). In particular, the non-degeneracy and expansion properties of $\RR$ will play a key role in our considerations. Our dynamical framework can be naturally extended to other classical models in statistical mechanics, such as the Ising or Potts models. Moreover, some of our results hold for more general notions of graph recursion; see our earlier paper~\cite{HP2024}.

Finally, we remark that the graph recursion operators we consider in this paper naturally arise from fractal spaces originating in various settings, such as attractors of iterated function systems, Julia sets of rational maps, and limit spaces of self-similar groups. For instance, each Misiurewicz polynomial induces a non-degenerate and expanding recursion operator; see Example~\ref{ex: dendrite}. Conversely, given a non-degenerate and expanding recursion operator $\RR$ and an induced recursive sequence $(G_n)_{n\geq 0}$ of finite graphs, one may naturally define a limiting self-similar infinite graph (in the sense of local convergence for pointed graphs), as well as a limiting fractal space (in the sense of  Gromov--Hausdorff convergence of metric spaces or as the boundary at infinity of a Gromov hyperbolic space). These connections place our setting within a broader geometric framework and suggest further directions for investigation. We do not pursue these constructions in the present paper; for related perspectives, we refer the interested reader to the survey paper \cite{BGN_Fractal_groups} in the context of self-similar groups, and to more recent articles \cite{IGS,VERS} in the context of vertex and edge replacement systems for graphs.

\subsection{Dynamical setting and results.}

The main reason for studying partition functions on recursive sequences of graphs, rather than on sequences of graphs that are potentially more natural from a physical perspective, is that the graph recursion operator typically induces a renormalization map directly related to the partition function of interest. This allows one to use techniques from dynamical systems in order to prove results for the zeros of partition functions that may be impossible to obtain outside of the recursive setting; see Section~\ref{sss: results for rec graphs} for an overview of such studies in the literature. On the other hand, recursive graphs provide a large class of non-trivial rational dynamical systems, the study of which stimulates new challenges in complex dynamics. We now briefly introduce our dynamical framework for the study of the independence polynomial roots for recursive graph sequences.

Let $\GG_k$ denote the set of finite graphs with a fixed number $k\in \N$ of  distinct labeled vertices (with labels $1,\dots,k$). We refer the reader to Section~\ref{sec: graph recursion} for a formal description of graph recursion operators $\RR$ on $\GG_k$ that we consider in this paper. As above, $\RR(G)$ is obtained by gluing together a fixed number $m\geq 2$ of copies of $G\in \GG_k$ according to a fixed rule; see Figures~\ref{fig: sierpinski tripod graphs}--\ref{fig: dendrite} for an illustration. 

For each fixed parameter $\lambda\in \C^*:=\C\setminus \{0\}$, the recursion operator $\RR$ induces a homogeneous polynomial self-map $\widehat{F}_\lambda: \C^{2^k}\to \C^{2^k}$ of degree $m$ such that $\widehat{F}_\lambda$ is semi-conjugate to the operator $\RR$: there exists a map $\widehat{\phi}_\lambda: \GG_k\to \C^{2^k}$ with $\widehat{F}_\lambda \circ \widehat{\phi}_\lambda =\widehat{\phi}_\lambda \circ \RR$; see Section~\ref{sec: renormalization map} and Corollary~\ref{cor: renormalization map} in particular. Furthermore, $\widehat{F}_\lambda$ is related to the independence polynomial via the following equation: 
\begin{equation}\label{eq: relation_phi_Z}
\Sigma \circ \widehat{\phi}_\lambda(G) = Z_G(\lambda),    
\end{equation}
where $\Sigma: \C^{2^k}\to \C$ denotes the coordinate-sum map.

The homogeneous polynomial self-map $\widehat{F}_\lambda: \C^{2^k}\to \C^{2^k}$ descends to a rational map 
\[
F_\lambda: \P^{2^k-1} \dashrightarrow \P^{2^k-1}
\]
of degree $m$ on the $(2^k-1)$-dimensional complex projective space, which we call the \emph{renormalization map} (for the independence
polynomial) associated with the graph recursion $\RR$. (Here and below, the dashed arrow $\dashrightarrow$ indicates that the respective map is \emph{partial}, i.e., it is well-defined only outside a certain \emph{indeterminacy set}.) At the same time, the semi-conjugacy $\widehat{\phi}_\lambda: \GG_k\to \C^{2^k}$ for $\widehat{F}_\lambda$ induces a (partial) semi-conjugacy $\phi_\lambda: \GG_k\dashrightarrow \P^{2^k-1}$ for $F_\lambda$ satisfying 
\begin{equation}\label{eq: semi-conj_for_F}
    F_\lambda(\phi_\lambda(G)) = \phi_\lambda(\RR(G)),
\end{equation}
whenever the images $\phi_\lambda(G)$ and $\phi_\lambda(\RR(G))$ of a graph $G\in \GG_k$ are defined. In particular, for all parameters $\lambda\in \R_+$ and all starting graphs $G_0\in \GG_k$, we have 
\[
    F^n_\lambda(\phi_\lambda(G_0)) = \phi_\lambda(G_n),
\]
where $G_n:=\RR^n(G_0)$ for all $n\geq 0$.

Even though the coordinate-sum map $\Sigma: \C^{2^k}\to \C$ does not descend to $\P^{2^k-1}$, its zero set does. Equations~\eqref{eq: relation_phi_Z} and~\eqref{eq: semi-conj_for_F} therefore imply that the values $F^n_\lambda(\xi_0)$ of the orbit of $\xi_0=\phi_\lambda(G_0)\in \P^{2^k-1}$ under $F_\lambda$ determine whether the corresponding independence polynomials $Z_{G_n}$ vanish at $\lambda$. This relates the study of the roots of the polynomials $Z_{G_n}$ to the study of the dynamical properties of the renormalization map $F_\lambda$.

The following theorem summarizes the dynamical properties of $F_\lambda$, which are used to establish Theorem~\ref{main result zeros}.

\begin{theorem}\label{thm: dynamics}
Let $k\in \N$ and $\lambda\in \C^*$, and suppose $F_\lambda$ is the renormalization
map associated with a graph recursion operator $\RR$ on $\GG_k$. There exists an explicit $k$-dimensional algebraic variety $\MM \subset \P^{2^k-1}$ that is invariant under $F_\lambda$ and such that the following statements are true: 
\begin{enumerate}[label={\normalfont (\roman*)}]
    \item\label{item: dyn1} If $\RR$ is non-degenerate, then, after passing to an iterate of $F_\lambda$ if necessary, $F_\lambda$ retracts $\MM$ to a subvariety $\MM_0$ of dimension $k_0 \le k$, which consists entirely of fixed points.
    \item\label{item: dyn2} If $\RR$ is non-degenerate and expanding, then the variety $\MM_0$ is transversally superattracting. 
    \item\label{item: dyn3} If $\RR$ is non-degenerate and expanding, then for every starting graph $G_0\in \GG_k$ and every parameter $\lambda\in \R_+$, the orbit of the point $\xi_0 = \phi_\lambda(G_0)$ converges to $\MM$.
\end{enumerate}
\end{theorem}

We remark that the statements in Theorem~\ref{thm: dynamics} should be understood to hold outside the indeterminacy set of $F_\lambda$ or of its appropriate iterate. Formal definitions of non-degenerate and expanding graph recursions are provided in Definition~\ref{def: non-deg_and_exp}; see also Lemma~\ref{lem: degeneration and expansion} for various characterizations. We refer the reader to Section~\ref{subsec: invariant variety} for an explicit description of the invariant variety $\MM$ and to (the proof of) Proposition~\ref{prop: pre fixed dynamics} for a description of the subvariety $\MM_0$. 

We note that the variety $\MM$ depends only on the number $k$ of labels, and is independent of a particular graph recursion $\RR$ on $\GG_k$ and parameter $\lambda\in \C^*$; in contrast, $\MM_0$ depends on both $\RR$ and $\lambda$. Furthermore, the variety $\MM$ turns out to be a standard construction in projective geometry, namely, $\MM$ is a \emph{$k$-fold Segre
variety} (see Section~\ref{sss: segre_embeddings}). This allows us to deduce its geometric properties: $\MM$ is a smooth and irreducible projective toric variety (see Lemma~\ref{lem: Segre}), and the same also holds for its subvariety $\MM_0$ when $\RR$ is non-degenerate (see  Proposition~\ref{prop: pre fixed dynamics}). Finally, we point out that the variety $\MM$ has a probabilistic interpretation: for a graph $G\in \GG_k$, the containment $\phi_\lambda(G)\in \MM$ is determined by \emph{absence of correlation} between the labeled vertices in $G$; see Lemma~\ref{lem: inv variety} for a precise statement. This interpretation plays a key role in the proof of part \ref{item: dyn3} of Theorem~\ref{thm: dynamics}; see Section~\ref{sec: decay of correlation}.

Another application of our dynamical framework is the following result.

\begin{theorem}\label{main bounded zeros}
    Let $k\geq 2$ and $\RR$ be a non-degenerate and expanding graph recursion operator on $\GG_k$. Consider a recursive graph sequence $(G_n)_{n\geq 0}$ generated by $\RR$, that is, $G_{n+1}=\RR(G_n)$. If the starting graph $G_0\in \GG_k$ is maximally independent, then the zeros of the independence polynomials $Z_{G_n}$ are uniformly bounded and avoid a uniform cone around $\R_{\geq0}$.
\end{theorem}

Maximally independent graphs in $\GG_k$ are introduced in Definition~\ref{def: maximally}; an example of such a graph is a $(k+2)$-star with $k$ labeled leaves. Example~\ref{ex: chebyshev} shows that the statement of Theorem~\ref{main bounded zeros} is false if the restriction on the starting graph $G_0$ is dropped. This contrasts with Theorem~\ref{main result zeros}, which holds for every starting graph.

\subsection{Motivation from statistical physics}\label{subsec: statistical physics}

Partition functions in statistical
physics encode the equilibrium properties of physical systems. In the discrete setting, the probability that a given state $\sigma$ of the system occurs is proportional to its \emph{weight} given by the \emph{Boltzmann factor} $e^{-\beta H(\sigma)}$, where $H(\sigma)$ is the energy of the system in the respective state and $\beta=\frac{1}{k_{\text{B}}T}$ is the inverse thermodynamic temperature (here, $T$ is the temperature and $k_\text{B}$ is the Boltzmann constant). The \emph{partition function} $Z$ is then defined as the total sum of the weights of all states $\sigma$ of the system:
\[Z=\sum_{\sigma}e^{-\beta H(\sigma)}.\]

The independence polynomial is one of the many possible partition functions, and it arises in the \emph{hard-core gas model}. In this model, each gas particle is assumed to occupy some definite region of space in which no other particle can occur, i.e., each particle has a ``hard-core''. In addition, it is assumed that the interaction energy between different gas particles is negligible. A discrete mathematical model of such a physical system is given by a (possibly infinite) graph $G$, often an induced subgraph of a regular lattice, in which each vertex can be either occupied or empty. Due to the hard-core constraint, two adjacent vertices are not allowed to be occupied simultaneously (equivalently, the interaction energy between two neighboring occupied vertices is infinite), hence the occupied vertices form an independent subset of the vertex set $V(G)$. We thus refer to an independent set $I\subset V(G)$ as a \emph{state} (or a \emph{configuration}) of the \emph{hard-core model on $G$}.

For the hard-core model on a finite graph $G$, the energy $H(I)$ of a state $I$ is determined only by the number of occupied vertices and equals $H(I)=-\mu\cdot \#I$, where $\mu$ is the chemical potential. It is customary to introduce the \emph{fugacity} (or \emph{activity}) \emph {parameter} $\lambda=e^{\beta\mu}$, so that the weight of a state $I$ is given by $e^{-\beta H(I)}=\lambda^{\#I}$. The corresponding partition function is then 
\[Z=\sum_I\lambda^{\#I},\] where 
the sum is taken over all independent sets $I\subset V(G)$; in other words, it coincides with the independence polynomial $Z_G(\lambda)$. In particular, the probability that an independent set $I$ occurs is
\[\P[I]= \lambda^{\#I}/Z_G(\lambda).\]
Note that, when the fugacity parameter $\lambda$ equals $1$, we obtain the ``uniform'' hard-core model, in which an independent set is chosen uniformly at random.

In statistical physics, one typically considers not a single finite graph, but rather a sequence $(G_n)_{n\geq 0}$ of larger and larger graphs converging in some sense to an infinite graph, such as a regular lattice. In this setting, one expects that the \emph{free energy per site}
\[\rho_ n(\lambda)=\frac{\log{Z_{G_n}(\lambda)}}{\#V(G_n)}\] 
has a limit as $n\to \infty$ for all $\lambda \geq 0$. The corresponding limit 
\begin{equation}\label{eq: lim_free_energy}
    \rho(\lambda) = \lim_{n\to \infty} \rho_n(\lambda) = 
\lim_{n\to\infty}\frac{\log{Z_{G_n}(\lambda)}}{\#V(G_n)}
\end{equation}
(if it exists) is called the \emph{limiting free energy per site}. By standard thermodynamic conventions, various physical quantities, such as pressure or density (i.e., fraction of occupied sites), can be derived from this limiting free energy $\rho$. In particular, the density is given by
\[d(\lambda)=\lambda \cdot 
\frac{d}{d\lambda} \rho(\lambda).\]

Of particular interest are the \emph{phase transitions} of the physical system: real parameters $\lambda_0>0$ where the limiting free energy per site $\rho$ fails to be real-analytic at $\lambda_0$. A relationship between zero sets of partition functions and the existence of phase transitions was first described by Lee and Yang in the context of induced finite subgraphs $G_n$ of the cubic lattice $\Z^d$ \cite{YangLee}. Under a mild assumption on the relative size of the boundaries of these subgraphs, Lee and Yang showed that the limiting free energy per site $\rho(\lambda)$ is well-defined and continuous, and moreover, that if the zeros of the partition functions $Z_{G_n}$ all avoid a fixed complex neighborhood of a parameter $\lambda_0 \ge 0$, then $\rho$ is real-analytic near $\lambda_0$. For the convenience of the reader, we give a short proof of this latter statement adapted to our setting.

\begin{lemma}
    Let $(G_n)_{n\geq 0}$ be a sequence of finite graphs, and assume that the free energy per site $\rho(\lambda)$ is well-defined for real parameters $\lambda$ in a neighborhood of a fixed parameter $\lambda_0 \ge 0$. Assume moreover that there exists a complex neighborhood $U$ of $\lambda_0$ where all of the partition functions $Z_{G_n}$ are non-zero. Then $\rho$ is real-analytic near $\lambda_0$.
\end{lemma}
\begin{proof}
    Since each partition function $Z_{G_n}$ is assumed to be non-zero on $U\subset \C$, the functions
    $$
    \rho_n(\lambda) = \frac{\log Z_{G_n}(\lambda)}{\#V(G_n)}
    $$
    are all well-defined holomorphic functions on $U$. Note that
    $$
    \mathrm{Re} (\rho_n(\lambda)) = \frac{\log |Z_{G_n}(\lambda)|}{\#V(G_n)} \le \frac{\log Z_{G_n}(|\lambda|)}{\#V(G_n)},
    $$
    where the latter holds since the polynomials $Z_{G_n}$ all have non-negative real coefficients. 
    
    Observe that removing edges from a finite graph $G$ can only increase the number of independent sets of $G$ of any fixed size. Hence, for a fixed $\#V(G)$, the maximal value of $Z_{G}(|\lambda|)$ is obtained when the graph $G$ has no edges at all, that is, $Z_G(|\lambda|)\leq (1+|\lambda|)^{\#V(G)}$. It therefore follows that 
    $$
    \mathrm{Re} (\rho_n(\lambda)) \le \frac{\log \left( (1+ |\lambda|)^{\#V(G_n)}\right)}{\#V(G_n)} = \log (1+ |\lambda|).
    $$
    By Montel's fundamental normality test, this implies that the holomorphic functions $\rho_n$ form a normal family on $U$, i.e., every sequence of these functions has a subsequence that converges uniformly on compact subsets of $U$. Since the free energies per site $\rho_n$ are assumed to converge to $\rho$ on the real axis near $\lambda_0$, they must converge on all of $U$ by the Identity Principle. The corresponding limiting map is holomorphic and thus real-analytic on $U\cap \R$. This implies the desired statement. 
\end{proof}

\subsection{Related results}
We now briefly discuss some related results in various settings. 

\subsubsection{Independence polynomial for bounded-degree graphs} The study of zeros of the independence polynomial for graphs with restricted local structure goes back to Chudnovsky and Seymour \cite{CS_clawfree}, who showed that the zeros of $Z_G$ are negative reals for all \emph{claw-free graphs} $G$, i.e., graphs that do not contain an induced $3$-star. In a recent work \cite{JP_zeroes}, Jerrum and Patel determined to what extent this result can be generalized to the setting of $H$-free graphs for different choices of a graph $H$; see also a prior work by Bencs \cite{Bencs_Thesis}. Specifically, they consider \emph{subdivided claws}---trees with a single vertex of degree $3$, and all other vertices of degree $1$ or $2$---which are obtained from a $3$-star by subdividing its edges. In particular, Jerrum and Patel prove that for every fixed subdivided claw $H$ and
any $\Delta\geq 3$, there exists a neighborhood $U=U_{H,\Delta}$ of $\R_{\geq0}$ in $\C$ such that the independence polynomial of every $H$-free graph of maximum vertex degree $\Delta$ has all of its zeros outside of $U$. Moreover, they show that no such result can hold when $H$ is not a subdivided claw, or when the maximum degree assumption is dropped.

In our recursive setting, whether or not the sequence $G_n=\RR^n(G_0)$ consists of $H$-free graphs for some subdivided claw $H$ depends both on the recursion $\RR$ and on the starting graph $G_0$. In particular, for the Sierpi\'nski gasket recursion (see Example~\ref{ex: sierpinksi}), if $G_0$ is a $3$-cycle, then the graphs $G_n$ are all claw-free; see Figure~\ref{fig: sierpinski graphs}. However, when the starting graph $G_0$ is a $3$-star with labeled leaves, then for any subdivided claw $H$ the graphs $G_n$ contain an induced $H$ for all sufficiently large $n$; see Figure~\ref{fig: sierpinski tripod graphs}. In particular, such a graph sequence $(G_n)_{n\geq 0}$ does not lie in the scope of the result from \cite{JP_zeroes}. Similar observations also apply to the $(z^2+i)$-graphs from Example~\ref{ex: dendrite} or, more generally, to the graph recursions induced by a non-Chebyshev Misiurewicz polynomial.

For general bounded-degree graphs, it follows from the results in \cite{BezakovaEtAl,BoerEtAlApprox} that zeros of the independence polynomial are dense outside of an explicit open set $U_\Delta\ni 0$ for graphs with vertex degrees bounded by $\Delta\geq 3$, with $U_\Delta$ shrinking to the origin as $\Delta \rightarrow \infty$. Moreover, it is known that considering only trees does not reduce the set of possible zeros: for every graph $G$, there exists a tree $T$ (with the same maximal vertex degree), the so-called \emph{tree of self-avoiding walks}~\cite{ScottSokal05,Weitz2006}, for which
$$
\{\lambda \; : \; Z_G(\lambda) = 0\} \subset \{\lambda \; : \; Z_T(\lambda) = 0\};
$$
see \cite[Proposition~2.7]{Bencs2018} for a more precise statement. We also point the interested reader to \cite{BoerEtAlTorus} for results concerning the boundedness of zeros of the independence polynomial for sequences of finite graphs converging to a cubic lattice.

\subsubsection{Partition functions on recursive sequences of graphs}
\label{sss: results for rec graphs}

Results regarding partition functions on recursive graphs go back to Ising, who showed in his PhD thesis that the Ising model on one-dimensional paths does not exhibit phase transitions; see \cite{IsingThesis} and \cite{IsingPaper}. We note that one-dimensional paths fall within the scope of the recursion operators we consider in this paper; see Example~\ref{ex: chebyshev}.

There are many examples in more recent literature where the authors exploit induced rational dynamical systems to study zeros of graph polynomials for more intricate graph recursions. One particular setting that has attracted considerable interest is that of orbital and finite \emph{Schreier graphs} of self-similar groups. The finite Schreier graphs frequently have a recursive structure and converge to infinite self-similar orbital Schreier graphs in the sense of local convergence for pointed graphs. The most notable studies employing dynamical systems machinery concern spectral properties of these graphs; see \cite{Bac2023} and references therein. We also refer the reader to \cite{DDN2011, DDN2012} for studies of partition functions of the Ising and dimer models and the corresponding thermodynamic limits on some specific Schreier graph sequences.

Another class of recursive graphs where the dynamical approach has proven particularly successful is that of \emph{hierarchical lattices}. These are sequences $(\Gamma_n)_{n\geq0}\subset \GG_2$ of finite graphs with two labeled vertices obtained iteratively from a single-edge graph $\Gamma_0$ by replacing each edge of $\Gamma_{n}$ with a fixed finite graph $\Gamma=\Gamma_1\in \GG_2$, called the \emph{generating graph}; see Example~\ref{ex: hier_lattices}. For this construction to be well-defined, the generating graph $\Gamma$ needs to be symmetric with respect to its labeled vertices. Figure~\ref{fig: diamond} illustrates the first few elements of such a sequence $(D_n)_{n\geq 0}$ of graphs when the generating graph $D_1$ is a $4$-cycle with two non-adjacent vertices labeled; these are the \emph{diamond hierarchical graphs}. The Ising model on such graphs was studied by Bleher, Lyubich and Roeder \cite{BLRI, BLRII} using dynamics of the Migdal--Kadanoff renormalization map. Subsequently, the Potts model on these graphs was studied in \cite{CRS20}, and on general hierarchical lattices by Chio and Roeder \cite{chio2021chromatic}; see also references therein.

The construction of the diamond hierarchical graphs $(D_n)_{n\geq 0}$ admits a natural duality: instead of replacing each edge of $D_n$ by the $4$-cycle $D_1$ to obtain $D_{n+1}$, one may equivalently replace each edge of the $4$-cycle by a copy of $D_n$. This duality, which holds for general hierarchical lattices, shows that the sequence $(D_n)_{n\geq 0}$ falls within the framework of recursive graphs treated in this paper. In fact, all hierarchical lattices with a bipartite generating graph are encompassed in our framework. However, the corresponding graph recursion operator generating $(D_n)_{n\geq 0}$ is degenerate: the maximal vertex degrees of the graphs $D_n$ grow exponentially and are in particular unbounded. Hence, our main results on zeros of the associated independence polynomials, as well as the results from \cite{JP_zeroes}, do not apply to this sequence.

We discuss one more sequence of recursive graphs extensively studied in statistical physics since the 1970s. For fixed down-degree $d$, let $(T_n)_{n\geq 0}\subset \GG_1$ be the sequence of $d$-ary rooted trees of depth $n$ with labeled roots. We note that such a sequence does not literally fall within the framework discussed here. Nevertheless, these trees can be constructed recursively in the following natural way: starting with the graph $T_0$ consisting of a single labeled vertex, the trees $T_{n+1}$ are defined inductively by taking $d$ copies of the previous tree $T_{n}$ and connecting their $d$ labeled roots to a new vertex, which becomes the labeled root of $T_{n+1}$. In other words, the labeled vertices of the $d$ copies of $T_{n}$ are connected not by identifying them but by inserting a $d$-star, whose central vertex becomes the new labeled root of $T_{n+1}$. 

In unpublished work by Rivera-Letelier and Sombra \cite{Rivera}, the dynamical approach was used to establish that the hard-core model on $(T_n)_{n\geq 0}$ exhibits a unique phase transition, which is of infinite order: the limiting free energy per site $\rho$ is real analytic at all parameters except one critical parameter $\lambda_{\text{cr}}=\lambda_{\text{cr}}(d)\in \R_+$, where it is still $C^\infty$. We point the reader to the master's thesis of van Willigen \cite{Willigen}, where the proof of analyticity of $\rho$ on $\R_+\setminus \{\lambda_{\text{cr}}\}$ as well as of its $C^\infty$-differentiability at $\lambda_{\text{cr}}$ is discussed, based on lectures given by Rivera-Letelier. Finally, we also refer the reader to \cite{Chio2019,Roeder_Cayley} and references therein for results concerning the Ising and Potts models on $(T_n)_{n\geq 0}$.

\subsubsection{Finite order ramification.} Our results are related to claims of Gefen, Aharony, and Mandelbrot~\cite{Mandelbrot1, Mandelbrot2,Mandelbrot3}, who studied the relationship between phase transitions and \emph{order of ramification} for fractal lattices, that is, infinite graphs with a self-similar structure.

\begin{definition}\label{def: finite order of ramification}
    Let $G$ be a countably infinite connected graph. We say that $G$ has \emph{finite order of ramification} if there exists $s\in \N$ such that for every finite connected subgraph $H$ of $G$ there exists a subset $S \subset V(G)$ of cardinality at most $s$ such that every infinite (self-avoiding) path in $G$ that starts in a vertex of $H$ must pass through a vertex of $S$.
\end{definition}

The authors of~\cite{Mandelbrot1, Mandelbrot2,Mandelbrot3} suggest that for infinite graphs $G$ with finite order of ramification, there should be no phase transitions in physical models on $G$ with short-range interactions, like the hard-core, Ising, or Potts models, in the sense that there should be a unique Gibbs measure on $G$ for all physical parameters. A precise result connecting the uniqueness of Gibbs measures for the multivariate hard-core model on an infinite graph and the critical probability for site percolation on that graph was proved by van den Berg and Steif \cite[Theorem~2.3]{vdBergSteif}. We provide below a more restrictive version of this statement adapted to our setting and notation, and refer the reader to \cite{vdBergSteif} for the details.  

\begin{theorem}[{Corollary of \cite[Theorem~2.3(i)]{vdBergSteif}}]\label{thm: unique_Gibbs_measure}
 Let $\lambda\in \R_{\geq 0}$ and $G$ be a countably infinite, locally finite connected graph. Consider the site percolation process on $G$ under which each vertex $v\in V(G)$ is open, independently of the others, with probability $\lambda/(\lambda+1)$ and closed with probability $1/(\lambda+1)$. We denote by $\P_{\lambda/(\lambda+1)}$ the induced probability measure on $2^{V(G)}$, and by $\{\exists\;\text{infinite open path}\}$ the event that there exists an infinite path in $G$ with all of its vertices open. 

 If $\P_{\lambda/(\lambda+1)}[\{\exists\;\text{infinite open path}\}]=0$, then the hard-core model on $G$ with the activity parameter $\lambda$ has a unique Gibbs measure.
\end{theorem}

If the graph $G$ from Theorem~\ref{thm: unique_Gibbs_measure} has finite order of ramification, then the condition \[\P_{\lambda/(\lambda+1)}[\{\exists\;\text{infinite open path}\}]=0\] is always satisfied. Indeed, one can find a sequence $(S_n)_{n\geq 0}$ of pairwise-disjoint subsets of $V(G)$ with uniformly bounded cardinality, such that every infinite path in $G$ must pass through a vertex of $S_n$ for all sufficiently large $n$. The desired condition then follows from the second Borel--Cantelli lemma. We get the following immediate corollary.

\begin{corollary}
If $G$ is a countably infinite, locally finite connected graph with finite order of ramification, then the hard-core model on $G$ has a unique Gibbs measure for all parameters $\lambda\in \R_{\geq 0}$.
\end{corollary}

In the current paper we study a different and generally non-equivalent notion of phase transitions, namely the non-analyticity of the limiting free energy per site $\rho$ for sequences $(G_n)_{n\geq0}$ of finite graphs, see \eqref{eq: lim_free_energy}. Motivated by the corollary above, one might naturally ask the following question:

\medskip
\emph{
    Let $G$ be a countably infinite, locally finite connected graph with finite order of ramification and $(G_n)_{n\geq0}$ be an increasing sequence of finite connected induced subgraphs of $G$ whose union equals $G$. Assume moreover that the free energy per site of the graphs $G_n$ converges to a limit $\rho(\lambda)$ for each $\lambda\in \R_+$; see \eqref{eq: lim_free_energy}. Does it follow that the limiting free energy $\rho$ is real-analytic at all $\lambda\in \R_+$?
}

\medskip

The answer to this question is negative:

\begin{example} Consider an arbitrary sequence $(H_j)_{j\geq 0}$ of finite connected graphs with $\#V(H_j)\to \infty$ as $j\to \infty$ and for which the limiting free energy per site exists for all $\lambda\in \R_+$ but is not real-analytic at some positive parameter $\lambda_0$. As an example we can consider the sequence of $d$-ary rooted trees of depth $j$; see also \cite[Proposition~2.7] {SlySun}. Fix a vertex $v(j)\in V(H_j)$ for each $j\geq 0$.  Given an increasing sequence $(j_n)_{n\geq 0}$ of positive integers, we construct recursively a sequence $(G_n)_{n\geq 0}$ of finite graphs as follows: we  set $G_0 = H_{j_0}$ and define $G_{n+1}$ to be the graph obtained from the disjoint union of the graphs $G_n$ and $H_{j_{n+1}}$ by connecting the vertices $v(j_n)$ of $G_n$ and $v(j_{n+1})$ of $H_{j_{n+1}}$ by an edge. The corresponding infinite graph $G$ is then obtained from the disjoint union $\bigsqcup_{n\geq 0} H_{j_{n}}$ by adding an edge between $v(j_{n})$ and $v(j_{n+1})$ for all $n\geq 0$. It is immediate that $G$ is a locally finite connected graph with finite order of ramification.

One also easily checks that, for each fixed $\lambda\in \R_+$ and $n\geq 0$, we have
\[Z_{G_n}(\lambda)\leq  \prod_{i=0}^{n}Z_{H_{j_i}}(\lambda)\leq Z_{G_n}(\lambda) \cdot (1+\lambda)^{n+1},\]
and thus 
\[\log Z_{G_n}(\lambda)= \sum_{i=0}^{n}\log Z_{H_{j_i}}(\lambda) + O(n).\]
It follows that
\[\frac{\log Z_{G_n}(\lambda)}{\#V(G_n)}= \sum_{i=0}^{n}\frac{\log Z_{H_{j_i}}(\lambda)}{\#V(H_{j_i})} \cdot \frac{\#V(H_{j_i})}{\#V(G_n)} + O\left(\frac{n}{\#V(G_n)}\right),\]
where 
\[\#V(G_n)=\sum_{i=0}^{n}\#V(H_{j_i}).\]
Using the Silverman--Toeplitz theorem, we conclude that
\[
\lim_{n\to \infty}\frac{\log Z_{G_n}(\lambda)}{\#V(G_n)}=\lim_{n\to \infty}\frac{\log Z_{H_{j_n}}(\lambda)}{\#V(H_{j_n})},
\]
provided that the sequence $(j_n)_{n\geq 0}$ increases sufficiently fast so that $n/\#V(G_n)\to 0$ as $n\to \infty$. In other words, the limiting free energy per site for the sequence $(G_n)_{n\geq 0}$ exists and coincides with the limiting free energy per site for the sequence $(H_j)_{j\geq 0}$, and in particular is not real-analytic at $\lambda_0$. 
\end{example}

It would be interesting to know under which conditions on the convergence of the sequence $(G_n)_{n\geq0}$, or perhaps on the self-similarity of the limiting graph $G$, the finite order of ramification of $G$ does imply the analyticity of the corresponding limiting free energy at all $\lambda\in \R_+$, or even stronger, the existence of a neighborhood of $\R_+$ that is free of the zeros of the independence polynomials $Z_{G_n}$. The recursivity of graph sequences that we introduce in this paper can be viewed as a stringent notion of both convergence and self-similarity that is sufficient for this purpose.

The notion of finite order of ramification plays an important role in our paper: an adaptation of Definition~\ref{def: finite order of ramification} to the setting of sequences $(G_n)_{n\geq 0}$ of finite graphs in $\GG_k$ (see Definition~\ref{def: finite ram order}) will be used to show that for all parameters $\lambda\in \R_+$ there occurs decay of correlation between the labeled vertices in $G_n$ (see Definition~\ref{def: decay of correlation} and Proposition~\ref{prop: decay of correlation}). In particular, this result applies to the recursive sequences $(G_n)_{n\geq 0}$ generated by a non-degenerate and expanding graph recursion operator $\RR$, which ultimately leads to part~\ref{item: dyn3} of Theorem~\ref{thm: dynamics}. We note that our proof actually implies that correlations between the labeled vertices of $G_n$ decay to $0$ exponentially fast in the generation $n$ (see Remark~\ref{rem: correlation_decay}), but since the distances between the labeled vertices in $G_n$ also increase exponentially fast, this does not immediately imply that the correlations decay exponentially fast in terms of the distance between the labeled vertices. However, since the invariant variety $\MM$ from Theorem~\ref{thm: dynamics} signifies absence of correlations and $\MM_0$ is transversally superattracting (see Lemma~\ref{lem: inv variety} and Theorem~\ref{thm: superattraction}),  it follows that the correlations between labeled vertices do decay exponentially fast in terms of the distance between the vertices. For more general results regarding the relation between zero-freeness and exponential decay of correlation for the hard-core model, we refer the reader to the recent papers \cite{Regts2023, PRR2026}. 

\subsection{Open questions} We conclude the introduction with a discussion of open questions and topics for future study that arise most naturally from the results and methods of this paper.

\subsubsection{Equidistribution of zeros.}

From the dynamical systems perspective, the most pressing question that is not addressed in this paper concerns the equidistribution of zeros. Let $\RR$ be a graph recursion operator, not necessarily non-degenerate or expanding, and let $G_0$ be a starting graph. As usual we write $G_n = \RR^n(G_0)$ for $n\in \N$. We formulate the question for the independence polynomials $Z_{G_n}$; it may be naturally extended to many other partition functions on $G_n$.

\begin{question}
    Under which conditions on the graph recursion $\RR$ and the starting graph $G_0$ does the sequence of probability measures $\nu_n$ defined by the normalized sums 
    \[
    \nu_n = \frac{1}{\deg(Z_{G_n})} \sum_{\lambda\in \C: \; Z_{G_n}(\lambda) = 0} \delta_\lambda
    \]
    of the Dirac masses $\delta_\lambda$ at the zeros of $Z_{G_n}$ (counted with multiplicity) 
    converge weakly to a probability measure $\mu$? If the limiting measure $\mu$ exists, how do its properties depend on the graph recursion and starting graph?
\end{question}

The equidistribution of zeros and properties of the limiting measures have been successfully studied and exploited for various partition functions on specific recursive graph sequences, such as (diamond) hierarchical lattices  \cite{BLRI, BLRII, chio2021chromatic} and $d$-ary rooted trees \cite{Chio2019,Rivera}, using dynamics of the renormalization map induced by a given graph recursion operator. We emphasize that the equidistribution of zeros is interesting not only from a dynamics point of view, but also for describing phase transitions in the physical context; see, for example, the works of Shrock and Tsai \cite{ShrockTsai} and Salas and Sokal \cite{SalasSokal} on the Potts model.

In the abstract dynamical context, the above question translates to the following.

\begin{question}
    Let $F_\lambda: \P^N \dashrightarrow \P^N$ be a one-parameter family of rational maps, let $\phi: \C \rightarrow \P^N$ be a  rational map, and let $\Sigma: \C^{N+1} \rightarrow \C$ be a homogeneous polynomial. Under which conditions on $F_\lambda$, $\phi$, and $\Sigma$, do the probability measures $\nu_n$, given by the normalized sums of the Dirac masses at the solutions of the equation
    \[
    (\Sigma \circ F_\lambda^n \circ \phi)(\lambda) = 0,
    \]
    equidistribute toward a probability measure $\mu$?
\end{question}

Here, the composition $\Sigma \circ F_\lambda^n \circ \phi$ is understood after lifting of $F_\lambda^n \circ \phi$ to $\C^{N+1}$. Since $\Sigma$ is homogeneous, whether the value $(\Sigma \circ F_\lambda^n \circ \phi)(\lambda)$ vanishes does not depend on the specific choice of the lift. In the most basic and classical setting, one simply considers a single rational map $f\colon \widehat{\C}\to \widehat{\C}$ and solutions to the equations $f^n(z)=w$ for a fixed value $w\in \widehat{\C}$; this corresponds to the case where $N=1$, the maps $F_\lambda=f$ do not depend on the parameter $\lambda$, $\phi\colon \C\to\widehat{\C}$ is the canonical inclusion map, and $\Sigma$ is a homogeneous linear map. The probability measures $\nu_n$ are then given by
\[
\nu_n = \frac{1}{\mathrm{deg}(f)^n} \sum_{z\in \widehat{\C}: \; f^n(z) = w} \delta_z.
\]
It was proved for polynomials by Brolin \cite{Brolin}, and for rational functions by Lyubich \cite{Lyubich83} and independently by Freire--Lopes--Ma\~n\'e \cite{FLM} that, except for at most two values of $w$, the measures $\nu_n$ converge to the unique probability measure of maximal entropy, which is independent of $w$. This result was generalized to holomorphic self-maps of projective space in arbitrary dimensions by Forn{\ae}ss and Sibony \cite{FS94}, with a more precise description of the exceptional set due to Briend and Duval \cite{BD2009}. 

Despite much work in this direction, the more general rational setting is far from fully understood. It is related to the theory of activity/bifurcation  currents and parameter space dynamics. We mention important contributions in this area due to Dinh--Sibony \cite{DinhSibony}, Guedj \cite{guedj2003}, DeMarco \cite{DeMarco2001, DeMarco2003}, Bassanelli--Berteloot \cite{BB2007}, and Dujardin--Favre \cite{DujardinFavre}.

\subsubsection{More general recursions.}

There exist many recursive sequences of graphs that are outside the framework discussed in this paper. Two important examples of such graphs are $d$-ary rooted trees and Hanoi graphs. Our definition of graph recursions $\RR\colon\GG_k\to \GG_k$ can be naturally generalized to include these and many other examples: when constructing $G_{n+1}=\RR(G_n)$, instead of connecting the different copies of $G_n$ by identifying their labeled vertices, we can connect them via some auxiliary \emph{connecting graphs}; see Remark~\ref{rem: general_rec} and our earlier paper \cite{HP2024} for more details. 

\begin{question}
    Do the main results in this paper hold when allowing arbitrary connecting graphs? In particular, do the zeros of the independence polynomials $Z_{G_n}$ still avoid a uniform neighborhood of $\R_{\geq 0}$ for non-degenerate and expanding graph recursion operators $\RR$ in this framework?
\end{question}

Note that even though the algebraic subvariety $\MM\subset \P^{2^k-1}$ from Theorem~\ref{thm: dynamics} will still be invariant under the corresponding renormalization map $F_\lambda\colon \P^{2^k-1} \dashrightarrow  \P^{2^k-1}$ in this more general framework, the dynamics of $F_\lambda$ on $\MM$ need not be pre-periodic even when the recursion operator $\RR$ is non-degenerate; see \cite{HP2024}. Furthermore, in this framework, degenerate recursion operators can produce sequences of connected graphs with uniformly bounded vertex degrees, such as $d$-ary rooted trees. Similarly, non-expanding recursion operators can produce sequences
$(G_n)_{n\geq0}\subset \GG_k$ of graphs where the distances between the labeled vertices in $G_n$ diverge, though the divergence rate will be linear in $n$, rather than exponential as for expanding recursion operators. These observations suggest that the scope of Theorem~\ref{main result zeros} may be extended, and motivate the following question.

\begin{question}
  Is it possible to characterize for which recursion operators $\RR\colon \GG_k\to \GG_k$ as in the more general framework from \cite{HP2024}, the hard-core model on the recursive sequences $(G_n)_{n\geq 0}$, $G_{n+1}=\RR(G_n)$, has a phase transition?
\end{question}

\subsubsection{Other partition functions.}

Recursive sequences of graphs make for natural models to study phase transitions of other partition functions arising in statistical physics, such as the Ising model, the Potts model, and many others. A natural question is to what extent the results obtained in this paper are model-independent.

Let us consider the $2$-state Ising model as an example. Given a finite graph $G$, up to a non-zero factor, we can write the corresponding partition function as
\[
Z_G(\lambda, b) = \sum_{U \subset V(G)} \lambda^{\#U} \cdot b^{\#\partial U},
\]
where $\partial U$ denotes the set of edges with one endpoint in $U$ and the other in $V(G)\setminus U$. Here, $\lambda$ is the magnetic field-like parameter, which plays an analogous role to the fugacity parameter in the hard-core model. The temperature-like parameter $b$ can be determined from the coupling constant that describes the interaction between neighboring particles. The seminal Lee--Yang theorem \cite{LY1952} asserts that in the ferromagnetic case, i.e., when $b \in (0,1)$, the zeros of $Z_G$ in the $\lambda$-plane---now known as the Lee--Yang zeros---lie on the unit circle. Moreover, it follows from \cite[Theorems 1.2 and 1.4]{BBGP22} that for every fixed $b\in \R_+\setminus\{1\}$, there exists $\Delta(b)\geq 2$ such that the union of the Lee--Yang zeros over all graphs with vertex degrees bounded by 
$\Delta(b)$ is dense in the unit circle.

It was also proved in joint work of the second author with Regts~\cite{PR2020} that for bounded-degree graphs, the Lee--Yang zeros avoid a neighborhood $U_{\Delta,b}$ of $\lambda = 1$ for each fixed $b\in (b_c,1)$, where the critical parameter $b_c\in (0,1)$ depends only on the upper bound $\Delta$ on the vertex degrees of the graphs. Moreover, $U_{\Delta,b}$ shrinks to $1$ as $b\to b_c$. We also refer the reader to the related paper~\cite{Chio2019} on the distribution of the Lee--Yang zeros for the $d$-ary rooted trees and Cayley trees of finite depths.

Given a non-degenerate and expanding graph recursion operator, which provides sequences of graphs with uniformly bounded vertex degrees, it may still be the case that the Lee--Yang zeros avoid a neighborhood of $\lambda = 1$ even for $0<b<b_c$.

\begin{question}
    Let $\RR$ be a non-degenerate and expanding graph recursion operator, and consider the sequence $(G_n)_{n \ge 0}$, $G_{n+1}=\RR(G_n)$, obtained from a starting graph $G_0$. Is it true that, in the ferromagnetic Ising model, for all fixed $b\in(0,1)$ the Lee--Yang zeros of $Z_{G_n}$ avoid a uniform neighborhood of $\lambda=1$?
\end{question}

\subsubsection{Random recursive graphs}

Another topic for future study, raised by Aleksey Kostenko, concerns the independence polynomial on randomly generated recursive sequences of graphs. Let us formulate an interpretation that stays closest to the setting of this paper.

Suppose that we consider $j\in \N$ different recursion operators $\RR_i$, $i\in\{1,\dots,j\}$, all acting on graphs $G \in \GG_k$ with the same number $k$ of distinct labeled vertices. Given a starting graph $G_0\in \GG_k$, we recursively define $G_{n+1} = \RR_i(G_n)$, where the recursion operators $\RR_i$ are chosen independently at each recursion step according to some probability vector $p = (p_1, \ldots, p_j)$ on $\{\RR_1,\dots ,\RR_j\}$. We are then interested in the behavior of the zeros of the independence polynomial for almost every sequence of graphs $(G_n)_{n\ge 0}$.

\begin{question}
Suppose that each of the recursion operators $\RR_i$ is expanding and non-degenerate. Does it follow that the zeros of the independence polynomials $Z_{G_n}$ almost surely avoid a uniform neighborhood of $\R_{\geq 0}$?
\end{question}

Our dynamical setting still provides a useful framework for studying this problem. Specifically, denoting by $F_{\lambda,i}\colon \P^{2^k-1}\dashrightarrow \P^{2^k-1}$ the renormalization map associated with the graph recursion $\RR_i$, we are interested in the behavior of the orbit of $\phi_\lambda(G_0) \in \P^{2^k-1}$ given by compositions of rational maps in $\mathcal{F}_\lambda=\{F_{\lambda,1}, \ldots, F_{\lambda,j}\}$, chosen independently at each iteration step according to the probability vector $p$.

Recall that the algebraic variety $\MM\subset \P^{2^k-1}$ from Theorem~\ref{thm: dynamics} depends only on the number $k$ of labeled vertices, and hence it is independent of the chosen collection $\{\RR_1,\dots ,\RR_j\}$ of recursion operators on $\GG_k$ and is invariant under each of the maps $F_{\lambda,i}\in \mathcal{F}_\lambda$. By the assumption that each of the recursion operators $\RR_i$ is expanding, it follows that the distances between labeled vertices in $G_n$ diverge almost surely. However, simple examples show that the maximal vertex degrees of the graphs $G_n$ do not need to remain uniformly bounded, and may, in fact, diverge with full probability, in stark contrast to the deterministic setting. It would therefore be interesting to see whether the arguments from Section~\ref{sec: decay of correlation} can be adapted to prove that for all parameters $\lambda\in \R_+$ the orbit of $\phi_\lambda(G_0)$ still converges to $\MM$ with full probability.

We also point out that the subvariety $\MM_0$ from Theorem~\ref{thm: dynamics} as well as the corresponding retraction map from $\MM$ to $\MM_0$ depend both on $\lambda$ and on the graph recursion operator, which now varies from step to step. It is therefore not clear whether the arguments in Section~\ref{section: zeros} generalize to the random setting, as these rely crucially on the fixed-point structure of $\MM_0$, which is lost when the recursion operator varies randomly.

\subsection{Organization of the paper} In Section~\ref{sec: preliminaries} we review some basic terminology and set up notation. We also show some auxiliary statements about the hard-core model on finite graphs in Section~\ref{subsec: auxiliary statements}. The formal graph recursions that we consider will be defined in Section~\ref{sec: graph recursion}. The induced action on projective space is considered in Section~\ref{sec: renormalization map}. In this section we introduce the invariant variety $\MM$, and prove that the subvariety $\MM_0$ is transversally superattracting for non-degenerate expanding recursion operators, Theorem \ref{thm: superattraction}. In Section~\ref{sec: decay of correlation} we prove decay of correlation, i.e., for physical parameters $\lambda> 0$ and non-degenerate expanding graph recursions, the corresponding orbits in projective space always converge to the invariant variety. Our main results on the lack of zeros near the non-negative real axis and the boundedness of zeros will be treated in Section~\ref{section: zeros}, and are quick consequences of the results in Sections~\ref{sec: renormalization map} and~\ref{sec: decay of correlation}.

\subsection*{Acknowledgments} The first author was partially supported by the Marie Sk\l{}odowska-Curie Postdoctoral Fellowship under the \emph{EU’s Horizon Europe} Programme (Grant No.\ 101068362). The research is partly financed by the Dutch Research Council (NWO) under the \emph{Open Competitie ENW M22-2} Programme (File No.\ OCENW.M.22.155). The authors would also like to thank Daniel
Meyer, Nguyen-Bac Dang, Tyler Helmuth, Guus Regts, and Palina Salanevich for various valuable discussions and comments.

The results presented here are related to the preprint \cite{HP2024}, which was made public several years ago by the authors but will not be published in its current form. In comparison, the arguments in this paper are more precise and accessible, but some of the statements are less general. The authors used suggestions by the LLM Claude (Anthropic) to improve the grammatical and mathematical precision of the final version of this paper.

\section{Notation, conventions, and preliminaries}\label{sec: preliminaries}

We denote the positive integers by $\N$, the integers by $\Z$, the real numbers by $\R$, the complex numbers by $\C$, the Riemann sphere by $\widehat{\C}:=\C\cup\{\infty\}$, and the $n$-dimensional complex projective space by $\P^n$. We also use the notations $\R_+$ and $\R_{\ge 0}$ for the positive and non-negative real numbers, respectively. The multiplicative group of nonzero complex numbers is denoted by $\C^*=\C\setminus\{0\}$. 

As standard, the symbol $\sqcup$ stands for the disjoint union of sets. The cardinality of a set $X$ is denoted by $\#X$, the power set of $X$  by $2^X$, and the identity map on $X$ by $\operatorname{id}_X$. If $f: X\to Y$ is a map between two sets and $A\subset X$, then $f|A$ stands for the restriction of $f$ to~$A$. When $X=Y$, that is, when $f$ is a self-map, we denote by $f^n$ the $n$-th iterate of $f$ for each $n\in \N$. We then also set $f^0=\operatorname{id}_X$ for convenience.

In this paper, unless otherwise stated, by a \emph{graph} $G$ we always mean an (undirected) multi-graph, that is, $G$ is given by a pair $(V, E)$, where $V$ is a finite \emph{vertex set} and $E$ is a finite \emph{edge multi-set} of unordered pairs of distinct elements in $V$. We emphasize that we allow \emph{multiple edges}, that is, distinct edges that connect the same pair of distinct vertices, but we do not allow \emph{loops}, that is, edges that connect a vertex to itself. 

Similarly, by a \emph{hypergraph} we mean a pair $(V,E)$, where $V$ is a finite \emph{vertex set} and $E$ is a finite (\emph{hyper})\emph{edge multi-set} of nonempty subsets of $V$. Again, a single subset of $V$ may occur multiple times in $E$. An edge $e\in E$ is called a \emph{$d$-edge} if $d=\#e$. Given a (hyper)graph $G$, we write $V(G)$ and $E(G)$ for the vertex set and edge multi-set of $G$, respectively. 

Let $G$ be a graph. As usual, a graph $G'$ is a \emph{subgraph} of $G$ if $V(G')\subset V(G)$ and $E(G')\subset E(G)$. We say that the subgraph $G'$ is \emph{induced} if $E(G')$ contains all edges of $E(G)$ that connect vertices in $V(G')$. Given a subset $X\subset V(G)$, we denote by $G[X]$ the induced subgraph of $G$ with the vertex set $X$. The graph $G$ is called \emph{$H$-free} for a graph $H$, if $G$ does not contain any induced subgraph that is isomorphic to $H$.

The \emph{vertex degree} of $v\in V(G)$ is the number of edges \emph{incident} to $v$, that is, the number of pairs $\{v,w\}\in E(G)$. A vertex $v$ is called a \emph{leaf} if it has degree $1$. We denote by $N_G[v]$ the \emph{closed neighborhood} of $v$ in $G$, that is, the subset of $V(G)$ consisting of the vertex $v$ and all vertices of $G$ adjacent to $v$.

Suppose $v,w\in V(G)$ are two distinct vertices. A \emph{path in $G$ from $v$ to $w$} is a sequence $v_0=v,v_1, \dots, v_n=w$ of distinct vertices such that each consecutive pair $v_i, v_{i+1}$ is connected by an edge. The number $n\in \N$ is called the \emph{length} of the path;  we say that the path \emph{connects} $v$ and $w$ and \emph{passes through} the vertices $v_0,v_1, \dots, v_n$. The \emph{graph distance} between $v$ and $w$ (in $G$) is the length of the shortest path connecting them; when no such path exists the distance is set to be $\infty$. The graph $G$ is \emph{connected} if for every pair of distinct vertices there exists a path in $G$ connecting them. An \emph{$n$-cycle} ($n\geq 2$) is a connected graph with $n$ vertices, each of degree $2$. 

The graph $G$ is called a \emph{tree} if it is connected and has no cycles. A tree with a single non-leaf vertex of degree $d\geq 2$ is called a $d$-star or a \emph{$d$-pod}.

For $k \in \mathbb N$, we write $\GG_k$ for the set of graphs with $k$ distinct marked vertices labeled $1, \ldots, k$. Furthermore, given $G\in \GG_k$, we denote by $P(G) \subset V(G)$ the corresponding subset of marked vertices in~$G$. 

\subsection{Independence polynomial}\label{subsec: indep-poly}
Let $G=(V,E)$ be a graph. We call any map $\sigma: V\to \{0,1\}$ a (\emph{vertex}) \emph{assignment on $G$}. We say that such $\sigma$ is \emph{independent} if $\sigma^{-1}(1):=\sigma^{-1}(\{1\}) \subset V$ forms an \emph{independent set} in $G$, that is, if no two vertices in $\sigma^{-1}(1)$ form an edge of $G$. We denote by $Z_G(\lambda)$ the (\emph{univariate}) \emph{independence polynomial} of $G$, that is, 
$$
Z_G(\lambda)=\sum_{\sigma} \lambda^{\#\sigma^{-1}(1)},
$$
where the sum is taken over all independent vertex assignments $\sigma$ on $G$.

Let $X\subset V$ be a vertex subset of $G$. We refer to any map $\tau: X\to \{0,1\}$ as a (\emph{vertex}) \emph{assignment on~$X$}, and write
$Z^\tau_G(\lambda)$ for the associated \emph{$\tau$-conditioned} independence polynomial of $G$, that is, 
\begin{equation}\label{eq: cond_indep_poly}
Z_G^\tau(\lambda)=\sum_{\sigma\sim\tau} \lambda^{\#\sigma^{-1}(1)},
\end{equation}
where the sum is taken over all independent vertex assignments $\sigma$ on $G$ with $\sigma|X=\tau$. Note that then
\begin{equation*}
Z_G(\lambda)=\sum_{\tau: X\to \{0,1\}} Z^\tau_G(\lambda).
\end{equation*}
We allow the possibility $X=\emptyset$, in which case $\tau$ is called the \emph{empty assignment} and $Z^\tau_G(\lambda)=Z_G(\lambda)$. We denote by $\obf_X$ and $\ibf_X$ the constant vertex assignments on $X$ that assign the values $0$ and $1$ to each vertex, respectively. By a \emph{partial} (\emph{vertex}) \emph{assignment} on $X$ we will mean a vertex assignment on a subset of $X$.

Given two vertex assignments $\tau: X\to \{0,1\}$ and $\sigma: Y\to \{0,1\}$ on two disjoint subsets $X,Y\subset V$, we write $\tau\wedge\sigma$ for the induced vertex assignment on $X\sqcup Y$ with $\tau\wedge\sigma|X=\tau$ and $\tau\wedge\sigma|Y=\sigma$. Furthermore, we denote by $\overline{\tau}$ the vertex assignment on $X$ given by $\overline{\tau}(x)= 1 - \tau(x)$ for all $x\in X$; we call such $\overline{\tau}$ the \emph{opposite} vertex assignment for $\tau$. 

We record below some well-known properties of the independence polynomial. 
\begin{enumerate}[label={(P\arabic*)},font=\normalfont]
    \item\label{prop:ind-poly-1}  The independence polynomial of a disjoint union $\bigsqcup_{i=1}^{m} G_i$ of graphs $G_1, \dots, G_m$ is simply the product of the individual independence polynomials:
    $$
        Z_{\bigsqcup_{i=1}^{m}  G_i}(\lambda) = \prod^m_{i=1} Z_{G_i}(\lambda).
    $$
    \item\label{prop:ind-poly-2}  Given a vertex $v$ of a graph $G=(V,E)$, suppose $G-v:=G[V\setminus\{v\}]$ and $G-N[v]:=G[V\setminus N_G[v]]$ are the induced subgraphs of $G$ obtained by removing the vertex $v$ and the closed neighborhood of $v$ from $G$, respectively. Then we have:
    $$
        Z_G^{0_v}(\lambda) = Z_{G-v}(\lambda) \quad \text{and} \quad Z_G^{1_v}(\lambda) = \lambda \cdot Z_{G-N[v]}(\lambda),
    $$
    where $0_v=\obf_{\{v\}}$ and $1_v=\ibf_{\{v\}}$ are the vertex assignments on $\{v\}$ mapping $v$ to $0$ and $1$, respectively. 
\end{enumerate}

\subsection{Probability interpretation}\label{subsec: probability interpretation}

Given a graph $G$ and a positive real parameter $\lambda$, we define a probability space $(\Omega_G, \mathcal{F}_G, \P_{G,\lambda})$ as follows:
\begin{itemize}
    \item the sample space $\Omega_G:=2^{V(G)}$ is the set of all subsets of the vertex set of $G$;
    \item the event space $\mathcal{F}_G:=2^{\Omega_G}$ is the set of all subsets of $\Omega_G$;
    \item the probability measure $\P_{G,\lambda}: \mathcal{F}_G \to  [0,1]$ is specified by assigning values to each elementary event $\{I\}\in \mathcal{F}_G$, $I\subset V(G)$: 
    \begin{equation}\label{eq: prob_of_state}
        \P_{G,\lambda}[\{I\}] = \begin{cases}
        \frac{\lambda^{\#I}}{Z_G(\lambda)}, & \text{if $I$ is an independent set in $G$};\\
        0, & \text{otherwise}. 
        \end{cases}
    \end{equation}
\end{itemize}
Recall from Section~\ref{subsec: statistical physics} that $\P_{G,\lambda}[\{I\}]$ is exactly the probability that the state $I\subset V(G)$ occurs in the hard-core model on $G$ in statistical physics. As long as the parameter $\lambda$ is fixed, we will simply write $\P_{G}$ in place of $\P_{G,\lambda}$ to simplify the notation.  

Let $X\subset V(G)$ be a vertex subset and $\tau: X\to \{0,1\}$ be a vertex assignment. Given $I\subset V(G)$, we say that $I$ \emph{agrees with $\tau$}, denoted $I \sim \tau$, whenever $I\cap X = \tau^{-1}(1)$. In other words, $I$ agrees with $\tau$ if, for each $x\in X$, we have $x\in I$ if and only if $\tau(x)=1$. With a slight abuse of notation, we will treat the vertex assignment $\tau$ as the following event in $\mathcal{F}_G$:
\[\{I\subset V(G): \,  I \sim \tau\}.\]
Then 
\[
\P_G[\tau] = \sum_{I\subset V(G):\, I \sim \tau} \P_G[\{I\}] = \frac{Z_G^\tau(\lambda)}{Z_G(\lambda)}.
\]
We will say that the assignment $\tau$ is \emph{admissible} (\emph{for $G$}) if $\P_G[\tau]>0$, that is, if no two vertices in $\tau^{-1}(1)\subset X$ form an edge of $G$.

Given two disjoint subsets $X,W \subset V(G)$ and vertex assignments $\tau: X \rightarrow \{0,1\}$ and $\sigma: W \rightarrow \{0,1\}$, the conditional probability $\P_G[\tau : \sigma]$ is defined by
\[
\P_G[\tau : \sigma] := \frac{\P_G[\tau \wedge \sigma]}{\P_G[\sigma]} = \frac{Z_G^{\tau\wedge\sigma}(\lambda)}{Z_G^{\sigma}(\lambda)},
\]
provided that $\sigma$ is admissible for $G$. Suppose now that a subset $Y\subset V(G)$ is disjoint from both $X$ and $W$. Then the \emph{law of total probability} (for conditional probabilities) implies that
\begin{equation}
\P_G[\tau : \sigma] = \sum_{\eta} \P_G[\tau : \eta \wedge \sigma] \, \P_G[\eta : \sigma], \tag{P3}\label{prop:ind-poly-3}
\end{equation}
where the sum ranges over all assignments $\eta$ on $Y$ such that $\P_G[\eta\wedge \sigma]\neq 0$.

\subsection{Auxiliary statements}\label{subsec: auxiliary statements} 
Let us recall first that a matrix $A=(A_1\dots A_s)\in \R^{r\times s}$ is called (\emph{column}) \emph{stochastic} if each of its columns $A_i$ is a probability vector, that is, $A_i$ has non-negative real entries that sum to $1$. The \emph{Dobrushin contraction coefficient $\delta(A)$} for the total variation distance, sometimes also called the \emph{Dobrushin ergodicity coefficient}, of the matrix $A$ is defined by
\[
\delta(A) := \frac{1}{2}\max_{i,j}{\|A_i-A_j\|_1},
\]
where $\|\cdot \|_1$ denotes the $\ell_1$-norm. In Section~\ref{sec: decay of correlation}, we will use the fact that the Dobrushin contraction coefficient is submultiplicative. Since we were able to locate this statement in the literature only for square matrices (see, for example, \cite[Section~4.3]{Seneta}), we include a proof of the general case here for completeness. 

\begin{lemma} \label{lem: Dobrushin contraction}
    Suppose $A$ and $B$ are two stochastic matrices whose dimensions are such that the product $AB$ is well-defined. Then 
    $\delta(AB)\leq \delta(A)\delta(B)$.
\end{lemma}

We remark that the proof below only requires $B$ to be stochastic.

\begin{proof}
     It is sufficient to consider the case when the matrix $B$, and hence also $C:=AB$, has only two columns. We denote the number of rows and columns of $A$ by $r$ and $s$, respectively. Let us also denote the columns of $A$ by $A_1,\dots, A_s$, the columns of $C$ by $C_1$ and $C_2$, and the entries of $B$ by $b_{i,j}$. We then aim to obtain an upper estimate for
    \[
    \|C_1-C_2\|_1 = \|\sum_{j=1}^s (b_{j,1} - b_{j,2})A_j  \|_1.
    \]    
    If $\delta(B)=0$, then $\delta(C)=0$ and the statement trivially holds. So we assume in the following that $\delta(B)>0$. 
     
     We may renumber the row indices of $B$ and correspondingly the column indices of $A$ so that
     \[\text{$b_{j,1} \ge b_{j,2}$\;
     for $j = 1 ,\ldots ,j_0$ \quad and \quad 
     $b_{j,1} < b_{j,2}$ \; for $j = j_0 + 1 ,\ldots, s$.}\] 
     Since 
    \[
    1=\sum_{j=1}^s b_{j,1}=\sum_{j=1}^{s}b_{j,2}
    \]
    and
    \[
    2\delta(B)=\sum_{j=1}^s|b_{j,1}-b_{j,2}|=\sum_{j=1}^{j_0}(b_{j,1}-b_{j,2})+\sum_{j=j_0+1}^{s}(b_{j,2}-b_{j,1}),
    \]
    it follows that 
    \[
     \delta(B)=\sum_{j=1}^{j_0}(b_{j,1}-b_{j,2})=\sum_{j=j_0+1}^{s}(b_{j,2}-b_{j,1}).
    \] 

    Let us set 
    \[\text{$x_j:=\frac{1}{\delta(B)}(b_{j,1}-b_{j,2})$\;
     for $j = 1 ,\ldots ,j_0$ \quad and \quad 
    $y_i:=\frac{1}{\delta(B)}(b_{i,2}-b_{i,1})$ \; for $i = j_0 + 1 ,\ldots, s$},\]
    so that
    \[1=\sum_{j=1}^{j_0}x_j=\sum_{i=j_0+1}^{s}y_i.\]
    Then 
    \begin{align*}
     \frac{1}{\delta(B)}\|C_1-C_2\|_1 = &\|\sum_{j=1}^{j_0} x_jA_j- \sum_{i=j_0+1}^s y_iA_i \|_1=\|\sum_{\substack{{j=1,\dots,j_0,}\\{i=j_0+1,\dots, s}}}x_jy_i(A_j-A_i) \|_1\leq\\
     &\sum_{\substack{{j=1,\dots,j_0,}\\{i=j_0+1,\dots, s}}}x_jy_i\|A_j-A_i\|_1 \leq \sum_{\substack{{j=1,\dots,j_0,}\\{i=j_0+1,\dots, s}}}x_jy_i 2\delta(A) = 2\delta(A), 
    \end{align*}
    which establishes the desired inequality $\delta(C)\leq \delta(A)\delta(B)$.
\end{proof}

We now collect several auxiliary statements that describe how conditioning behaves in the hard-core model under certain structural assumptions on the underlying graph. 

\begin{lemma}\label{lem: probability lower bound} 
    Given $\lambda \in \R_+$ and $n, \Delta \in \N$, there exists a constant $c(\lambda, \Delta, n)>0$ such that the following holds: 
    
    For every graph $G$ and every non-empty subset $X\subset V(G)$ whose vertices have degree at most $\Delta$, we have
    \[
    \P_{G,\lambda}[\tau] \geq c(\lambda, \Delta, \#X)
    \]
    for all admissible assignments $\tau$ on $X$.
\end{lemma}
\begin{proof}
    Consider the subset $Y:=\bigcup_{v\in X} N_G[v]$. We can then write
    \[
    \P_{G,\lambda}[\tau] = \sum_{\sigma} \P_{G,\lambda}[\tau : \sigma] \P_{G,\lambda}[\sigma],
    \]
    where we sum over all possible admissible assignments $\sigma$ on $V(G) \setminus Y$. Observe now that for the independence polynomial, considering independent sets in $G$ that agree with an assignment $\sigma$ on $V(G) \setminus Y$ is equivalent to considering independent sets in an induced subgraph $G[S]$ for a certain subset $S=S(\sigma)$ of $Y$. More precisely, let $S \subset Y$ consist of those vertices in $Y$ that are not adjacent to any vertex in $\sigma^{-1}(1)$. Then $X \subset S$ and 
    \[
    \P_{G,\lambda}[\tau : \sigma] = \P_{G[S],\lambda}[\tau].
    \]
   
    Since 
    \[Z^\tau_{G[S]}(\lambda) \geq \lambda^{\#\tau^{-1}(1)}\geq \min\{1,\lambda^{\#X}\} \quad \text{and} \quad Z_{G[S]}(\lambda) \leq (1+\lambda)^{\#S}\leq (1+\lambda)^{\#X(\Delta+1)},\]
    we get that 
    \[\P_{G[S],\lambda}[\tau]\geq c := \frac{ \min\{1,\lambda^{\#X}\}}{(1+\lambda)^{\#X(\Delta+1)}}>0,\]
    and thus
    \[
    \P_{G,\lambda}[\tau] = \sum_{\sigma} \P_{G,\lambda}[\tau : \sigma] \P_{G,\lambda}[\sigma] \geq   \sum_{\sigma} c  \, \P_{G,\lambda}[\sigma] = c,
    \]
    with $c=c(\lambda,\Delta,\#X)$ as desired. This finishes the proof.   
\end{proof}  

The following two lemmas will play a key role in the proof of Theorem~\ref{thm: decay of correlation}.

\begin{lemma}\label{lem: conditioning wrt Y}
    Let $G$ be a graph, and let $\tau$, $\eta$, and $\sigma$ be vertex assignments on pairwise disjoint subsets $X\neq \emptyset$, $Y$, and $W\neq \emptyset$ of $V(G)$, respectively.  Suppose that every path from a vertex in $X$ to a vertex in $W$ passes through a vertex in $Y$. Then, for all $\lambda\in \R_+$, we have
    \[\P_{G,\lambda}[\tau: \eta\wedge \sigma] = \P_{G,\lambda}[\tau: \eta],\]
    provided that $\P_{G,\lambda}[\eta\wedge \sigma]\neq 0$.
\end{lemma}

\begin{proof}
    Let $H$ be the induced subgraph of $G$ that contains $X\sqcup Y$ as well as all vertices of $G$ that can be reached by a path starting in $X$ and not passing through any vertex of $Y$. Note that then $W\subset V(G) \setminus V(H)$. The desired statement will follow directly from the following claim using the law of total probability.

    \begin{claim}
    Suppose $\widetilde{\sigma}$ is an arbitrary assignment on $V(G) \setminus V(H)$ such that $\eta\wedge \widetilde{\sigma}$ is admissible for $G$. Then 
    \[\P_{G,\lambda}[\tau: \eta\wedge \widetilde{\sigma}] = \P_{H,\lambda}[\tau: \eta] = \P_{G,\lambda}[\tau: \eta].\]    
    \end{claim}
    
    Indeed, we have 
    \[
    Z_{G}^{\tau\wedge \eta\wedge \widetilde{\sigma}}(\lambda)=Z_{H}^{\tau\wedge \eta}(\lambda)\cdot \lambda^{\#\widetilde{\sigma}^{-1}(1)} \quad \text{and} \quad 
    Z_{G}^{\eta\wedge \widetilde{\sigma}}(\lambda)=Z_{H}^{\eta}(\lambda)\cdot \lambda^{\#\widetilde{\sigma}^{-1}(1)}, 
    \]
    which imply the first equality. The second equality follows immediately from the following chain of identities (based on the law of total probability):
    \begin{align*}
    \P_{G,\lambda}[\tau: \eta]&=\sum_{\widetilde{\sigma}}\P_{G,\lambda}[\tau\wedge \widetilde{\sigma}: \eta] = \sum_{\widetilde{\sigma}}\P_{G,\lambda}[\tau: \widetilde{\sigma} \wedge \eta] \, \P_{G,\lambda}[\widetilde{\sigma}: \eta]  \\
    &=\sum_{\widetilde{\sigma}}\,\P_{H,\lambda}[\tau:\eta] \P_{G,\lambda}[\widetilde{\sigma}: \eta] = \P_{H,\lambda}[\tau:\eta] \cdot \sum_{\widetilde{\sigma}}\P_{G,\lambda}[\widetilde{\sigma}: \eta] \\
    &= \P_{H,\lambda}[\tau:\eta],
    \end{align*}
    where all the sums are taken over all possible vertex assignments $\widetilde{\sigma}$ on $V(G) \setminus V(H)$ such that $\widetilde{\sigma}\wedge \eta$ is admissible for $G$. 
\end{proof}

Let $G$ be a graph and $\lambda$ a fixed positive parameter. Given two disjoint non-empty subsets $X,W\subset V(G)$, we may consider a matrix $A=A(X,W)$ whose entries are given by the conditional probabilities $\P_{G,\lambda}[\tau:\sigma]$, with rows corresponding to admissible assignments $\tau$ on $X$ and columns corresponding to admissible assignments $\sigma$ on $W$. Clearly, the matrix $A$ is stochastic. The following lemma implies that the Dobrushin contraction coefficient $\delta(A)$ is bounded above by a constant $\varepsilon(\lambda, \#X)\in (0,1)$.

\begin{lemma}\label{lem: total var bound} 
    Given $\lambda \in \R_+$ and $n \in \N$, there exists a constant $\varepsilon(\lambda, n)\in (0,1)$ such that the following holds: 
    
    For every graph $G$, for all disjoint non-empty subsets $X,W\subset V(G)$, and for all admissible assignments $\sigma,\widetilde{\sigma}$ on $W$, we have
    \[\frac{1}{2}\; \sum_{\tau: X\to \{0,1\}} |\P_{G,\lambda}[\tau : \sigma]-\P_{G,\lambda}[\tau : \widetilde\sigma]| \leq \varepsilon(\lambda, \#X),\]
    where the sum is taken over all assignments $\tau$ on $X$. 
\end{lemma}

\begin{proof}
Recall that $\obf_X$ denotes the constant assignment on $X$ that assigns the value $0$ to each vertex. 
Then, for every assignment $\tau$ on $X$ and every admissible assignment $\sigma$ on $W$, we have 
    \[
    Z^{\tau\wedge\sigma}_{G}(\lambda)\leq \lambda^{\#\tau^{-1}(1)}\cdot Z^{\obf_X\wedge\sigma}_{G}(\lambda)\leq \max\{1,\lambda^{\#X}\}\cdot Z^{\obf_X\wedge\sigma}_{G}(\lambda),
    \]
   and thus
   \[
   \P_{G,\lambda}[\obf_X:\sigma]=   \frac{Z^{\obf_X\wedge\sigma}_{G}(\lambda)}{Z^{\sigma}_{G}(\lambda)} = \frac{Z^{\obf_X\wedge\sigma}_{G}(\lambda)}{\sum_{\tau}Z^{\tau\wedge\sigma}_{G}(\lambda)} \geq \frac{1}{2^{\#X}\max\{1,\lambda^{\#X}\}},
   \]
   where the sum is taken over all possible assignments $\tau$ on $X$. 

   Suppose now that  $\sigma,\widetilde{\sigma}$ are arbitrary admissible assignments on $W$. Since $|a-b|=a+b-2\min\{a,b\}$, we obtain
   \begin{align*} 
   &\sum_{\tau: X\to \{0,1\}} |\P_{G,\lambda}[\tau : \sigma]-\P_{G,\lambda}[\tau : \widetilde{\sigma}]|=\\
   &\sum_{\tau: X\to \{0,1\}} \P_{G,\lambda}[\tau : \sigma]+\sum_{\tau: X\to \{0,1\}} \P_{G,\lambda}[\tau : \widetilde{\sigma}]-2\sum_{\tau: X\to \{0,1\}}\min\{\P_{G,\lambda}[\tau :  \sigma],\, \P_{G,\lambda}[\tau : \widetilde{\sigma}]\}\leq\\
   &1+1-2\min\{\P_{G,\lambda}[\obf_X :  \sigma],\, \P_{G,\lambda}[\obf_X : \widetilde{\sigma}]\}\leq 2\left(1-\frac{1}{2^{\#X}\max\{1,\lambda^{\#X}\}}\right),
   \end{align*}
    which implies the desired statement with
    \[
    \varepsilon(\lambda,n):=1-\frac{1}{2^{n}\max\{1,\lambda^{n}\}}\in (0,1).
    \] 
\end{proof}

\section{Graph recursion operator}\label{sec: graph recursion}

Fix positive integers $k$ and $m\geq 2$. We start by formally introducing a \emph{graph recursion operator} $\RR=\RR_{(H,\Phi)}$
on the set $\GG_k$ of graphs with $k$ (distinct) vertices labeled $1,\dots,k$. The operator $\RR$ depends on a pair $(H,\Phi)$ that is specified as follows:
\begin{itemize}
    \item $H$ is a hypergraph with the vertex set $V(H) = \{1,\dots,m\}$ and edge multi-set $E(H)$, whose edges $e\in E(H)$ each have a label $\ell(e)\in \{1,\dots, k\}$ such that the following holds: for each $j\in\{1,\dots,k\}$, the edges of $H$ labeled $j$ form a partition of $V(H)$.
    \item  $\Phi: \{1,\dots,k\}\to E(H)$ is an injective map.
\end{itemize} 
We call the edge-labeled hypergraph $H$ a \emph{gluing scheme}, the map $\Phi$ a \emph{labeling map}, and the pair $(H,\Phi)$ a (\emph{reduced}) \emph{gluing data} (with parameters $k,m$).

\begin{definition}[Graph recursion $\RR$]\label{def:graph recursion} Let $(H, \Phi)$ be a gluing data with parameters $k,m\in \N$ as above, and suppose that we are given a graph $G\in\GG_k$. We define a new marked graph $\RR(G)=\RR_{(H,\Phi)}(G)\in \GG_k$ as follows:
\begin{enumerate}[label=(\Roman*)]
    \item First, we take $m$ disjoint copies $G(1),\dots, G(m)$ of $G$, where each copy $G(i)$ is associated with the vertex $i$ of the gluing scheme $H$.

    \item Then for each edge $e = \{i_1, \ldots, i_s\} \in E(H)$ with label $\ell(e)\in\{1,\dots,k\}$, we identify the vertices labeled $\ell(e)$ in the copies $G(i_1),\dots,G(i_s)$.

    \item Finally, we assign $k$ marked vertices in the resulting graph according to the labeling map $\Phi$: the label $j \in \{1, \ldots, k\}$ is assigned to the vertex that corresponds to the edge $\Phi(j) \in E(H)$.
\end{enumerate}
We call the induced operator $$
\RR = \RR_{(H, \Phi)}: \GG_k \rightarrow \GG_k, \; G\mapsto \RR(G),
$$
the \emph{graph recursion} with parameters $k,m$ associated with the gluing data $(H,\Phi)$.
\end{definition}

\begin{rem}\label{rem: general_rec}
    In \cite{HP2024}, we considered more general gluing data and graph recursions. There, roughly speaking, the $m$ copies of the given graph $G\in \GG_k$ are joined by some auxiliary \emph{connecting graphs} $\Sigma_e$ for every edge $e\in E(H)$. Additionally, each label $j\in\{1,\dots,k\}$ is assigned to a distinguished marked vertex within the graph $\Sigma_{\Phi(j)}$. The graph recursion considered in this paper corresponds to the special case where each $\Sigma_e$ consists of a single vertex.
\end{rem}

Given a gluing data $(H, \Phi)$ with parameters $k,m$ and an arbitrary \emph{starting graph} $G_0\in \GG_k$, we define a sequence $(G_n)_{n\geq 0}$ of marked graphs, each having $k$ distinct marked vertices labeled $1,\dots, k$, by $G_{n+1} := \RR_{(H, \Phi)}(G_n)$. For every $n\in \N$, the $n$-th iterate of the operator $\RR_{(H,\Phi)}$ determines a gluing data $(H_n,\Phi_n)$ with parameters $k$ and $m^n$ such that $G_n=\RR_{(H_n,\Phi_n)}(G_0)$. We refer to $(H_n,\Phi_n)$ as the \emph{$n$-th iterate} of $(H,\Phi)$.

\begin{example}[Sierpi\'nski gluing data]\label{ex: sierpinksi}
    We now describe the \emph{Sierpi\'nski gluing data} $(H,\Phi)$ with parameters $k=3$ and $m=3$ governing the recursive construction of the Sierpi\'nski tripod graphs from Figure~\ref{fig: sierpinski tripod graphs} and the Sierpi\'nski gasket graphs from Figure~\ref{fig: sierpinski graphs}. The gluing scheme is given by the hypergraph $H$ with 
    \[V(H)=\{1,2,3\} \quad \text{and} \quad E(H)=\lbrace \{1\},\{2\},\{3\},\{1,2\},\{1,3\},\{2,3\}\rbrace.\]
    For each $j \in \{1,2,3\}$, we label the $1$-edge $\{j\}$ and the $2$-edge $\{1,2,3\}\setminus \{j\}$ with the label $j$. We note that these two edges form a partition of $V(H)=\{1,2,3\}$. The labeling map $\Phi$ is defined by $\Phi(j) = \{j\}$ for each $j\in \{1,2,3\}$. We remark that, since the elements of $E(H)$ are pairwise distinct, an edge of $H$ is already determined by the respective subset of $V(H)$, and thus it is not necessary to specify the label of the edge $\{j\}$.
    When we take the starting graph to be a $3$-star (i.e., a \emph{tripod}) and iterate $\RR_{(H,\Phi)}$, we obtain the Sierpi\'nski tripod graphs from Figure~\ref{fig: sierpinski tripod graphs}; when we take the starting graph to be a $3$-cycle, we obtain the Sierpi\'nski gasket graphs from Figure~\ref{fig: sierpinski graphs}.   
\end{example}

\begin{figure}
\begin{tikzpicture}[every node/.style={font=\small}]

\def\h{0.866}
\def\x{1}

\begin{scope}
\draw[black, thick] (0,0) -- (\x,0);
\draw[black, thick] (0,0) -- (\x/2,\x*\h);
\draw[black, thick] (\x/2,\x*\h) -- (\x,0);

\filldraw[black] (0,0) circle (2.4pt) node[below left] {$\mathbf{1}$};
\filldraw[black] (\x,0) circle (2.4pt) node[below right] {$\mathbf{2}$};
\filldraw[black] (\x/2,\x*\h) circle (2.4pt) node[above=2pt] {$\mathbf{3}$};

\node[font=\normalsize] at (\x/2,0) [below=15pt] {$S_0$};
\end{scope}

\begin{scope}[xshift=4cm]
\draw[black, thick] (0,0) -- (2*\x,0);
\draw[black, thick] (0,0) -- (\x,2*\x*\h);
\draw[black, thick] (\x,2*\x*\h) -- (2*\x,0);

\draw[black, thick] (\x/2,\h) -- (\x,0);
\draw[black, thick] (\x/2,\h) -- (3/2*\x,\h);
\draw[black, thick] (\x,0) -- (3/2*\x,\h);

\filldraw[black] (0,0) circle (2.4pt) node[below left] {$\mathbf{1}/1$};
\filldraw[black] (2*\x,0) circle (2.4pt) node[below right] {$\mathbf{2}/2$};
\filldraw[black] (\x,2*\x*\h) circle (2.4pt) node[above] {$\mathbf{3}/3$};

\filldraw[black] (\x/2,\h) circle (1.6pt) node[left=1.5pt] {$2$};
\filldraw[black] (\x,0) circle (1.6pt) node[below=1.5pt] {$3$};;
\filldraw[black] (3/2*\x,\h) circle (1.6pt) node[right=1.5pt] {$1$};;

\node[font=\normalsize] at (\x,0) [below=15pt] {$S_1$};
\end{scope}

\begin{scope}[xshift=10cm]
\draw[black, thick] (0,0) -- (4*\x,0);
\draw[black, thick] (0,0) -- (2*\x,4*\x*\h);
\draw[black, thick] (2*\x,4*\x*\h) -- (4*\x,0);

\draw[black, thick] (2*\x,0) -- (\x,2*\x*\h);
\draw[black, thick] (2*\x,0) -- (3*\x,2*\x*\h);
\draw[black, thick] (\x,2*\x*\h) -- (3*\x,2*\x*\h);

\begin{scope}[xshift=0cm]
\draw[black, thick] (\x/2,\h) -- (\x,0);
\draw[black, thick] (\x/2,\h) -- (3/2*\x,\h);
\draw[black, thick] (\x,0) -- (3/2*\x,\h);

\filldraw[fill = gray] (\x/2,\h) circle (1.6pt);
\filldraw[fill = gray] (\x,0) circle (1.6pt);
\filldraw[fill = gray] (3/2*\x,\h) circle (1.6pt);
\end{scope}

\begin{scope}[xshift=2*\x cm]
\draw[black, thick] (\x/2,\h) -- (\x,0);
\draw[black, thick] (\x/2,\h) -- (3/2*\x,\h);
\draw[black, thick] (\x,0) -- (3/2*\x,\h);

\filldraw[fill = gray] (\x/2,\h) circle (1.6pt);
\filldraw[fill = gray] (\x,0) circle (1.6pt);
\filldraw[fill = gray] (3/2*\x,\h) circle (1.6pt);
\end{scope}

\begin{scope}[xshift=\x cm, yshift=2*\h cm]
\draw[black, thick] (\x/2,\h) -- (\x,0);
\draw[black, thick] (\x/2,\h) -- (3/2*\x,\h);
\draw[black, thick] (\x,0) -- (3/2*\x,\h);

\filldraw[fill = gray] (\x/2,\h) circle (1.6pt);
\filldraw[fill = gray] (\x,0) circle (1.6pt);
\filldraw[fill = gray] (3/2*\x,\h) circle (1.6pt);
\end{scope}

\filldraw[black] (2*\x,0) circle (1.6pt) node[below=1.5pt] {$3$};
\filldraw[black] (\x,2*\x*\h) circle (1.6pt) node[left=1.5pt] {$2$};
\filldraw[black] (3*\x,2*\x*\h) circle (1.6pt) node[right=1.5pt] {$1$};

\filldraw[black] (0,0) circle (2.4pt) node[below left] {$\mathbf{1}/1$};
\filldraw[black] (4*\x,0) circle (2.4pt) node[below right] {$\mathbf{2}/2$};
\filldraw[black] (2*\x,4*\x*\h) circle (2.4pt) node[above] {$\mathbf{3}/3$};

\node[font=\normalsize] at (2*\x,0) [below=15pt] {$S_2$};
\end{scope}

\end{tikzpicture}
    \caption{Illustration of the recursion for the Sierpi\'{n}ski gasket graphs $(S_{n})_{n=0,1,2}$. The thicker black vertices with their bold labels correspond to the marked vertices in each graph, labeled $1,2,3$. The black vertices all together with their normal font labels correspond to the marked vertices of the copies of the graph from the previous step.}
    \label{fig: sierpinski graphs}
\end{figure}
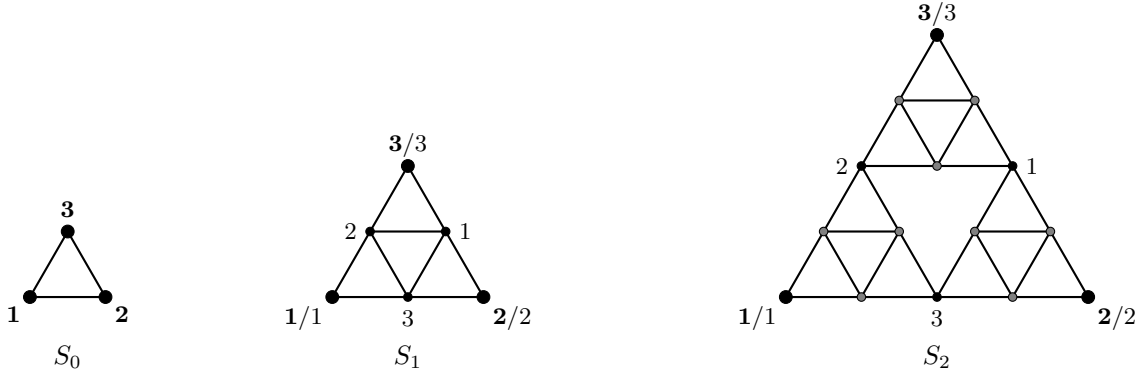

\begin{definition}[Label dynamics]\label{def: label dynamics}
    A gluing data $(H,\Phi)$ with parameters $k,m$ induces a dynamical system $\Lambda=\Lambda_{(H,\Phi)}: \{1,\ldots, k\}\to \{1,\ldots, k\}$ on the labels: namely, we map each label $j\in \{1,\dots, k\}$ to the label $\Lambda(j):=\ell(\Phi(j))$ of the edge $\Phi(j)\in E(H)$. A label $j$ is called \emph{periodic} if $\Lambda^{p}(j)=j$ for some $p\geq 1$. Evidently, every label is eventually mapped to a periodic one under iteration of $\Lambda$.
    
    The number $\#\Phi(j)$ is called the (\emph{local}) \emph{degree} of a label $j$.   
    We say that $j$ is \emph{critical} if $\#\Phi(j)\geq 2$. 
\end{definition}

\begin{definition}[Non-degenerate and expanding gluing data]\label{def: non-deg_and_exp} We say that a gluing data $(H,\Phi)$ with parameters $k,m$ and the associated  graph recursion $\RR=\RR_{(H,\Phi)}$ are:
\begin{itemize}
    \item \emph{non-degenerate} if the induced dynamical system $\Lambda=\Lambda_{(H,\Phi)}$ has no periodic critical labels, and \emph{degenerate} in the opposite case;

    \item \emph{expanding} if there exists $n\in \N$ such that the edges $\Phi_n(1), \ldots, \Phi_n(k)$ form pairwise disjoint subsets of $V(H_n)$, where $(H_n, \Phi_n)$ is the $n$-th iterate of $(H, \Phi)$. 
\end{itemize}
\end{definition}

\begin{example}
    For the Sierpi\'nski gluing data $(H,\Phi)$, we have $\Lambda(j) = j$ for all $j \in \{1,2,3\}$, and each label has degree $1$. In particular, this data is non-degenerate. At the same time, since $\Phi(j)=\{j\}$ for each label $j$, we have that the data is also expanding. 
\end{example}

The proof of the following lemma, which provides different equivalent characterizations for the two classes of gluing data above, is straightforward and is left to the reader.

\begin{lemma}\label{lem: degeneration and expansion}
Let $\RR=\RR_{(H,\Phi)}$ be a graph recursion with parameters $k$ and $m$, and $\Lambda=\Lambda_{(H,\Phi)}$ be the corresponding dynamical system on the labels. Suppose $N\in \N$ is an iterate such that $\Lambda^N(j)=j$ for each periodic label $j\in\{1,\dots,k\}$ and $\Lambda^{2N}(j)=\Lambda^N(j)$ for each non-periodic label $j$, and let $(H_N, \Phi_N)$ be the gluing data of $\RR^N$. Then the following conditions are equivalent:
\begin{itemize}
    \item the recursion $\RR$ is non-degenerate;
    \item for each periodic label $j\in\{1,\dots,k\}$, we have $\#\Phi_N(j)=1$;
    \item for every starting graph $G_0 \in \GG_k$, the vertex degrees in the graphs $G_n=\RR^n(G_0)$ are uniformly bounded (in $n$);
    \item for some starting graph $G_0 \in \GG_k$ where each vertex with a periodic label is non-isolated, the vertex degrees in the graphs $G_n=\RR^n(G_0)$ are uniformly bounded (in $n$).
\end{itemize}
Similarly, the following conditions are equivalent: 
\begin{itemize}
    \item the recursion $\RR$ is expanding;
    \item for each pair of distinct periodic labels $j,j'\in\{1,\dots,k\}$ the corresponding edges $\Phi_N(j)$ and $\Phi_N(j')$ form disjoint subsets of $V(H_N)$;
    \item for every starting graph $G_0 \in \GG_k$ and for all distinct labels $j,j'\in\{1,\dots,k\}$ the graph distance between the vertices labeled $j$ and $j'$ in $G_n=\RR^n(G_0)$ goes to $\infty$ as $n\to \infty$;
    \item for some starting graph $G_0 \in \GG_k$ where all vertices with periodic labels lie in the same connected component, we have that for all distinct periodic labels $j,j'\in\{1,\dots,k\}$ the graph distance between the vertices labeled $j$ and $j'$ in $G_n=\RR^n(G_0)$ goes to $\infty$ as $n\to \infty$.
\end{itemize}
\end{lemma}

\begin{rem}\label{rem: expansion}
    The definitions imply that, if a gluing data $(H,\Phi)$ is expanding, then the edges $\Phi_n(1), \ldots, \Phi_n(k)$ form pairwise disjoint subsets of $V(H_n)$ for  \emph{all} sufficiently large $n\in \N$. In fact, this holds for all $n \ge N+1$, where $N$ is as in the statement of Lemma~\ref{lem: degeneration and expansion}. 
\end{rem}

We can easily see that the limiting free energy per site is well defined for sequences of recursive graphs.

\begin{lemma}\label{lem: limit_energy}
    Let $\RR=\RR_{(H,\Phi)}$ be an arbitrary graph recursion with parameters $k, m$, and let $G_0 \in \GG_k$ be an arbitrary starting graph. Suppose $(G_n)_{n \ge 0}\subset \GG_k$ is the sequence of graphs 
    defined recursively by $G_{n+1} = \RR(G_n)$. Then the free energy per site $\rho_n(\lambda)=\frac{\log{Z_{G_n}(\lambda)}}{\# V(G_n)}$ for the sequence $(G_n)_{n\geq 0}$ converges for all $\lambda\in\mathbb R_+$, i.e., the limiting free energy per site $\rho(\lambda) = \lim_{n\to\infty } \rho_n(\lambda)$ is well defined. 
\end{lemma}
\begin{proof}
    Since $G_{n+1}$ is constructed by taking $m$ copies of $G_n$ and identifying a fixed number of vertices, we have that
    $$
    \# V(G_{n+1}) = m \cdot \# V(G_n) - C,
    $$
    where the constant integer $C\geq 0$ depends only on $\RR$ and not on $G_n$. In fact, \[C=\sum_{e\in E(H)} (\#e-1)\leq k(m-1).\] It follows that
    \[
    \alpha_n:=\frac{\# V(G_n)}{m^n}=\#V(G_0)-C\cdot \sum_{i=1}^{n}\frac{1}{m^i}
    \]
    converges to 
    \[\alpha:=\#V(G_0)-\frac{C}{m-1}\geq 0  \quad \mathrm{as} \quad n \rightarrow \infty.
    \] 
    We note that $\alpha>0$, unless $k=\#V(G_0)$ and the gluing scheme $H$ has a single edge $\{1,\dots, m\}$ labeled $j$ for each $j\in \{1,\dots, k\}$. In the latter case, the sequence  $(G_n)_{n \ge 0}$ is constant (up to edge multiplicities) and the statement is trivial. Hence, we may assume that $\alpha>0$ in the following.

    Now observe that by identifying some of the vertices in a given graph, the number of independent sets of a fixed size can only decrease. It follows that
    $$
    Z_{G_{n+1}}(\lambda) \le \big(Z_{G_n}(\lambda)\big)^m  \quad \text{for all} \quad \lambda\in \R_+. 
    $$
    Hence, for each fixed $\lambda\in\R_+$, the sequence $$
    \beta_n:=\frac{\log Z_{G_n}(\lambda)}{m^n}
    $$
    is decreasing in $n$, and thus converges to some $\beta(\lambda)\geq0$. It follows that the free energy per site $\rho_n(\lambda)=\frac{\beta_n}{\alpha_n}$ converges to $\frac{\beta(\lambda)}{\alpha}$ as $n\to\infty$, which finishes the proof. 
\end{proof}

We close this section by providing several examples of gluing data on graphs and discussing their properties.

\begin{figure}[t]
\begin{tikzpicture}[scale=0.8]

\begin{scope}
\draw[black, thick] (0,-4) -- (0,-2.37);
\filldraw[black] (0,-2.37) circle (3pt) node[above=2pt] {$\mathbf{1}$};
\filldraw[black] (0,-4) circle (3pt) node[below=17pt] {$D_0$} node[below=1pt] {$\mathbf{2}$};
\end{scope}

\begin{scope}[xshift=3cm]
\draw[black, thick] (0,-4) -- (0.4,-2.42);
\draw[black, thick] (0,-4) -- (-0.4,-2.42);
\draw[black, thick] (0,-0.83) -- (0.4,-2.42);
\draw[black, thick] (0,-0.83) -- (-0.4,-2.42);

\filldraw[black] (0,-4) circle (3pt) node[below=17pt] {$D_1$} node[below] {$\mathbf{2}/1$};
\filldraw[black] (0,-0.83) circle (3pt) node[above] {$\mathbf{1}/1$};
\filldraw[black] (-0.4,-2.42) circle (2pt) node[left=1.5pt] {$2$};
\filldraw[black] (0.4,-2.42) circle (2pt) node[right=1.5pt] {$2$};
\end{scope}

\begin{scope}[xshift=7cm]
\draw[black, thick] (0,-4) -- (0.4,-2.42);
\draw[black, thick] (0,-4) -- (-0.4,-2.42);
\draw[black, thick]  (0,-4) -- (1.01,-2.75);
\draw[black, thick]  (0,-4) -- (-1.01,0.-2.75);

\draw[black, thick] (0,1.65) -- (0.4,0.06);
\draw[black, thick] (0,1.65) -- (-0.4,0.06);
\draw[black, thick] (0,1.65) -- (1.01,0.4);
\draw[black, thick] (0,1.65) -- (-1.01,0.4);

\draw[black, thick] (1.41,-1.17) -- (0.4,-2.42);
\draw[black, thick] (-1.41,-1.17) -- (-0.4,-2.42);
\draw[black, thick] (1.41,-1.17) -- (1.01,-2.75);
\draw[black, thick] (-1.41,-1.17) -- (-1.01,0.-2.75);
\draw[black, thick] (1.41,-1.17) -- (0.4,0.06);
\draw[black, thick] (-1.41,-1.17) -- (-0.4,0.06);
\draw[black, thick] (1.41,-1.17) -- (1.01,0.4);
\draw[black, thick] (-1.41,-1.17) -- (-1.01,0.4);

\filldraw[black] (0,-4) circle (3pt) node[below=17pt] {$D_2$} node[below] {$\mathbf{2}/1$};
\filldraw[black] (0,1.65) circle (3pt) node[above] {$\mathbf{1}/1$};
\filldraw[black] (-1.41,-1.17) circle (2pt) node[left=1.5pt] {$2$};
\filldraw[black] (1.41,-1.17) circle (2pt) node[right=1.5pt] {$2$};
\filldraw[fill=gray] (-0.4,-2.42) circle (2pt);
\filldraw[fill=gray] (0.4,-2.42) circle (2pt);
\filldraw[fill=gray]  (0.4,0.06) circle (2pt);
\filldraw[fill=gray]  (-0.4,0.06) circle (2pt);
\filldraw[fill=gray]  (-1.01,0.-2.75) circle (2pt);
\filldraw[fill=gray]  (1.01,-2.75) circle (2pt);
\filldraw[fill=gray]  (-1.01,0.-2.75) circle (2pt);
\filldraw[fill=gray]  (1.01,-2.75) circle (2pt);
\filldraw[fill=gray]  (-1.01,0.4) circle (2pt);
\filldraw[fill=gray]  (1.01,0.4) circle (2pt);
\end{scope}

\begin{scope}[xshift=5cm]
\draw[black, thick] (6,0) -- (6.125, 1.625);
\draw[black, thick] (6,0) -- (6.775, 1.375);
\draw[black, thick] (6,0) -- (7.375,0.875);
\draw[black, thick] (6,0) -- (7.625,0.125);

\draw[black, thick] (6,0) -- (6.125, -1.625);
\draw[black, thick] (6,0) -- (6.775, -1.375);
\draw[black, thick] (6,0) -- (7.375,-0.875);
\draw[black, thick] (6,0) -- (7.625,-0.125);

\draw[black, thick] (7,3) -- (6.125, 1.625);
\draw[black, thick] (7,3) -- (6.775, 1.375);
\draw[black, thick] (9,1) -- (7.375,0.875);
\draw[black, thick] (9,1) -- (7.625,0.125);

\draw[black, thick] (7,-3) -- (6.125, -1.625);
\draw[black, thick] (7,-3) -- (6.775, -1.375);
\draw[black, thick] (9,-1) -- (7.375,-0.875);
\draw[black, thick] (9,-1) -- (7.625,-0.125);

\draw[black,thick] (9,-1) -- (9.875, -2.375);
\draw[black,thick] (9,-1) -- (9.225, -2.625);
\draw[black,thick] (7,-3) -- (8.625, -3.125);
\draw[black,thick] (7,-3) -- (8.375, -3.875);

\draw[black,thick] (10,-4) -- (9.875, -2.375);
\draw[black,thick] (10,-4) -- (9.225, -2.625);
\draw[black,thick] (10,-4) -- (8.625, -3.125);
\draw[black,thick] (10,-4) -- (8.375, -3.875);

\draw[black,thick] (9,1) -- (9.875, 2.375);
\draw[black,thick] (9,1) -- (9.225, 2.625);
\draw[black,thick] (7,3) -- (8.635, 3.125);
\draw[black,thick] (7,3) -- (8.375, 3.875);

\draw[black,thick] (10,4) -- (9.875, 2.375);
\draw[black,thick] (10,4) -- (9.225, 2.625);
\draw[black,thick] (10,4) -- (8.625, 3.125);
\draw[black,thick] (10,4) -- (8.375, 3.875);

\draw[black, thick] (6,0) -- (6.125, 1.625);
\draw[black, thick] (6,0) -- (6.775, 1.375);
\draw[black, thick] (6,0) -- (7.375,0.875);
\draw[black, thick] (6,0) -- (7.625,0.125);

\draw[black, thick] (14,0) -- (13.875, 1.625);
\draw[black, thick] (14,0) -- (13.225, 1.375);
\draw[black, thick] (14,0) -- (12.625,0.875);
\draw[black, thick] (14,0) -- (12.375,0.125);

\draw[black, thick] (13,3) -- (13.875, 1.625);
\draw[black, thick] (13,3) -- (13.225, 1.375);
\draw[black, thick] (11,1) -- (12.625,0.875);
\draw[black, thick] (11,1) -- (12.375,0.125);

\draw[black,thick] (11,1) -- (10.125, 2.375);
\draw[black,thick] (11,1) -- (10.775, 2.625);
\draw[black,thick] (13,3) -- (11.375, 3.125);
\draw[black,thick] (13,3) -- (11.625, 3.875);

\draw[black,thick] (10,4) -- (10.125, 2.375);
\draw[black,thick] (10,4) -- (10.775, 2.625);
\draw[black,thick] (10,4) -- (11.375, 3.125);
\draw[black,thick] (10,4) -- (11.625, 3.875);

\draw[black, thick] (14,0) -- (13.875, -1.625);
\draw[black, thick] (14,0) -- (13.225, -1.375);
\draw[black, thick] (14,0) -- (12.625, -0.875);
\draw[black, thick] (14,0) -- (12.375, -0.125);

\draw[black, thick] (13,-3) -- (13.875, -1.625);
\draw[black, thick] (13,-3) -- (13.225, -1.375);
\draw[black, thick] (11,-1) -- (12.625, -0.875);
\draw[black, thick] (11,-1) -- (12.375, -0.125);

\draw[black,thick] (11,-1) -- (10.125, - 2.375);
\draw[black,thick] (11,-1) -- (10.775, -2.625);
\draw[black,thick] (13,-3) -- (11.375, -3.125);
\draw[black,thick] (13,-3) -- (11.625, -3.875);

\draw[black,thick] (10,-4) -- (10.125, -2.375);
\draw[black,thick] (10,-4) -- (10.775, -2.625);
\draw[black,thick] (10,-4) -- (11.375, -3.125);
\draw[black,thick] (10,-4) -- (11.625, -3.875);

\filldraw[black] (6,0) circle (2pt) node[left=1.5pt] {$2$};
\filldraw[fill=gray] (6.125, 1.625) circle (2pt);
\filldraw[fill=gray] (6.775, 1.375) circle (2pt);
\filldraw[fill=gray] (7.375,0.875) circle (2pt);
\filldraw[fill=gray] (7.625,0.125) circle (2pt);

\filldraw[fill=gray] (6.125, -1.625) circle (2pt);
\filldraw[fill=gray] (6.775, -1.375) circle (2pt);
\filldraw[fill=gray] (7.375,-0.875) circle (2pt);
\filldraw[fill=gray] (7.625,-0.125) circle (2pt);

\filldraw[fill=gray] (7,-3) circle (2pt);
\filldraw[fill=gray] (7,3) circle (2pt);
\filldraw[fill=gray] (9,-1) circle (2pt);
\filldraw[fill=gray] (9,1) circle (2pt);

\filldraw[fill=gray] (9.875, -2.375) circle (2pt);
\filldraw[fill=gray] (9.225, -2.625) circle (2pt);
\filldraw[fill=gray] (8.635, -3.125) circle (2pt);
\filldraw[fill=gray] (8.375, -3.875) circle (2pt);

\filldraw[fill=gray] (9.875, 2.375) circle (2pt);
\filldraw[fill=gray] (9.225, 2.625) circle (2pt);
\filldraw[fill=gray] (8.635, 3.125) circle (2pt);
\filldraw[fill=gray] (8.375, 3.875) circle (2pt);

\filldraw[black] (10,-4) circle (3pt) node[below=17pt] {$D_3$} node[below] {$\mathbf{2}/1$};
\filldraw[black] (10,4) circle (3pt) node[above] {$\mathbf{1}/1$};

\filldraw[fill=gray] (10.125, -2.375) circle (2pt);
\filldraw[fill=gray] (10.775, -2.625) circle (2pt);
\filldraw[fill=gray] (11.375, -3.125) circle (2pt);
\filldraw[fill=gray] (11.625, -3.875) circle (2pt);

\filldraw[fill=gray] (10.125, 2.375) circle (2pt);
\filldraw[fill=gray] (10.775, 2.625) circle (2pt);
\filldraw[fill=gray] (11.375, 3.125) circle (2pt);
\filldraw[fill=gray] (11.625, 3.875) circle (2pt);

\filldraw[black] (14,0) circle (2pt) node[right=1.5pt] {$2$};
\filldraw[fill=gray] (13,-3) circle (2pt);
\filldraw[fill=gray] (13,3) circle (2pt);
\filldraw[fill=gray] (11,-1) circle (2pt);
\filldraw[fill=gray] (11,1) circle (2pt);

\filldraw[fill=gray] (13.875, 1.625) circle (2pt);
\filldraw[fill=gray] (13.225, 1.375) circle (2pt);
\filldraw[fill=gray] (12.625,0.875) circle (2pt);
\filldraw[fill=gray] (12.375,0.125) circle (2pt);

\filldraw[fill=gray] (13.875, -1.625) circle (2pt);
\filldraw[fill=gray] (13.225, -1.375) circle (2pt);
\filldraw[fill=gray] (12.625,-0.875) circle (2pt);
\filldraw[fill=gray] (12.375,-0.125) circle (2pt);
\end{scope}
\end{tikzpicture}
    \caption{Illustration of the recursion for the diamond hierarchical graphs $(D_{n})_{n=0,1,2,3}$. The thicker black vertices with their bold labels correspond to the marked vertices in each graph, labeled $1,2$. The black vertices all together with their normal font labels correspond to the marked vertices of the copies of the graph from the previous step.}
    \label{fig: diamond}
\end{figure}
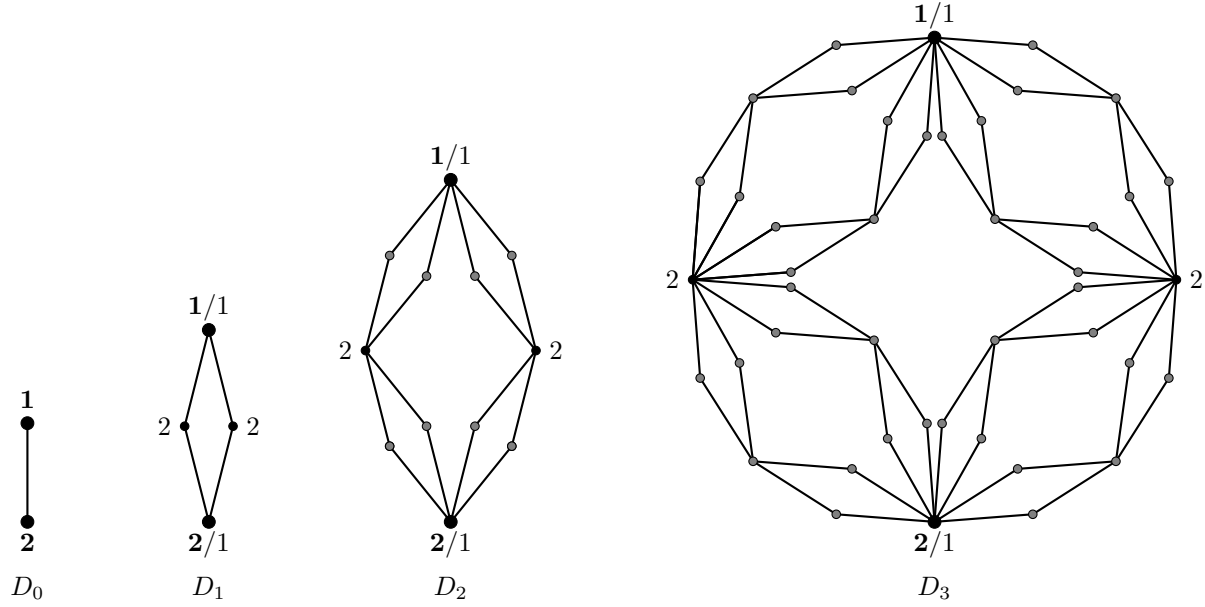

\begin{example}[Diamond hierarchical gluing data]\label{ex: diamon_lattice}
    For the diamond hierarchical graphs $D_n$ (see Figure~\ref{fig: diamond}), the corresponding gluing data $(H,\Phi)$ has parameters $k=2$ and $m=4$. The gluing scheme $H$ is a $4$-cycle, whose edges are alternately labeled $1$ and $2$, and the labeling map $\Phi$ assigns the labels $1$ and $2$ to the two ``opposite'' edges in $H$ labeled $1$. The resulting recursion $\RR_{(H,\Phi)}$ is expanding and degenerate. In fact, the distance between the two marked vertices of $D_n$ is $2^n$, and the vertex degrees of both of these vertices are $2^{n}$ as well.
\end{example}

\begin{example}[More hierarchical lattices]\label{ex: hier_lattices} 
   Further examples of hierarchical lattices were considered in \cite{chio2021chromatic}. Formally, a \emph{hierarchical lattice} is a sequence $(\Gamma_n)_{n\geq 0}$ of graphs in $\GG_2$,  constructed recursively in the following way. Start with a single-edge graph $\Gamma_0$ (with two vertices) and fix a graph $\Gamma_1=\Gamma\in \GG_2$ that is symmetric with respect to its marked vertices, called the \emph{generating graph}. For each $n\in \N$, define the graph $\Gamma_{n}$ by replacing each edge of $\Gamma$ with a copy of $\Gamma_{n-1}$, using the two marked vertices in $\Gamma_{n-1}$ as if they were the endpoints of that edge. 
   We also mark the two vertices of $\Gamma_n$ corresponding to the marked vertices of $\Gamma$, preserving their original labels. Equivalently, we may also construct $\Gamma_n$ by replacing each edge of $\Gamma_{n-1}$ with $\Gamma$. It is easy to see that such sequences $(\Gamma_n)_{n\geq 0}$ fit into our framework of recursive graphs whenever the generating graph $\Gamma$ is bipartite, i.e. its vertex set can be partitioned into two independent subsets. In this case, the corresponding gluing data is non-degenerate if and only if both marked vertices of $\Gamma$ have vertex degree $1$. Furthermore, the gluing data is expanding if and only if there is no edge in $\Gamma$ connecting the marked vertices.  
\end{example}

\begin{figure}
\begin{tikzpicture}[scale=0.8]

\def\y{0.5}
\def\x{0.866}
\def\h{1}

\newcommand{\YtreeUp}[2]{
  \draw[black, thick] (#1,#2) -- (#1,#2-\h);
  \draw[black, thick] (#1,#2) -- (#1+\x,#2+\y);
  \draw[black, thick] (#1,#2) -- (#1-\x,#2+\y);
  \filldraw[fill=gray] (#1,#2) circle (1.6pt);
  \filldraw[fill=gray] (#1,#2-\h) circle (1.6pt);
  \filldraw[fill=gray] (#1+\x,#2+\y) circle (1.6pt);
  \filldraw[fill=gray] (#1-\x,#2+\y) circle (1.6pt);
}

\newcommand{\YtreeDown}[2]{
  \draw[black, thick] (#1,#2) -- (#1,#2+\h);
  \draw[black, thick] (#1,#2) -- (#1+\x,#2-\y);
  \draw[black, thick] (#1,#2) -- (#1-\x,#2-\y);
  \filldraw[fill=gray] (#1,#2) circle (1.6pt);
  \filldraw[fill=gray] (#1,#2+\h) circle (1.6pt);
  \filldraw[fill=gray] (#1+\x,#2-\y) circle (1.6pt);
  \filldraw[fill=gray] (#1-\x,#2-\y) circle (1.6pt);
}

\begin{scope}
  \YtreeUp{0}{0}

  \filldraw[black] (0,0) circle (1.6pt);
  \filldraw[black] (-\x,\y) circle (2.4pt) node[above] {$\mathbf{2}$};
  \filldraw[black] (\x,\y) circle (2.4pt) node[above] {$\mathbf{1}$};
  \filldraw[black] (0,-\h) circle (2.4pt) node[below] {$\mathbf{3}$};
\end{scope}

\begin{scope}[xshift=3.5*\x cm, yshift=-\y cm]
  \YtreeDown{0}{0}
  \YtreeUp{0}{2*\h}

  \filldraw[black] (0,\h) circle (1.6pt) node[right] {$1$}; 
  \filldraw[black] (\x,-\y) circle (1.6pt) node[below] {$3$};
  \filldraw[black] (-\x,-\y) circle (2.4pt) node[below] {$\mathbf{3}/2$};
  \filldraw[black] (-\x,\y+2*\h) circle (2.4pt) node[above] {$\mathbf{2}/3$};
  \filldraw[black] (\x,\y+2*\h) circle (2.4pt) node[above] {$\mathbf{1}/2$};
\end{scope}

\begin{scope}[xshift=8.5*\x cm, yshift=\h cm]
  \YtreeDown{0}{0}
  \YtreeUp{0}{2*\h}
  \YtreeUp{2*\x}{-\h}
  \YtreeDown{-2*\x}{3*\h}

  \filldraw[black] (0,\h) circle (1.6pt) node[right] {$1$}; 
  \filldraw[black] (2*\x,-2*\h) circle (1.6pt) node[below] {$3$};
  \filldraw[black] (-\x,-\y) circle (2.4pt) node[below] {$\mathbf{3}/2$};
  \filldraw[black] (-2*\x,4*\h) circle (2.4pt) node[left] {$\mathbf{2}/3$};
  \filldraw[black] (\x,\y+2*\h) circle (2.4pt) node[above] {$\mathbf{1}/2$};
\end{scope}

\begin{scope}[xshift=17*\x cm, yshift=\h cm]
  \YtreeDown{0}{0}
  \YtreeUp{0}{2*\h}
  \YtreeUp{2*\x}{-\h}
  \YtreeUp{-2*\x}{-\h}
  \YtreeDown{-2*\x}{3*\h}
  \YtreeDown{4*\x}{0}
  \YtreeDown{2*\x}{3*\h}\
  \YtreeUp{-4*\x}{2*\h}

  \filldraw[black] (0,\h) circle (1.6pt) node[right] {$1$}; 
  \filldraw[black] (2*\x,-2*\h) circle (1.6pt) node[below] {$3$};
  \filldraw[black] (-3*\x,-\y) circle (2.4pt) node[above] {$\mathbf{3}/2$};
  \filldraw[black] (-2*\x,4*\h) circle (2.4pt) node[left] {$\mathbf{2}/3$};
  \filldraw[black] (3*\x,\y+2*\h) circle (2.4pt) node[below] {$\mathbf{1}/2$};
\end{scope}

\node at (0,-2) [below] {$T_0$};
\node at (3.5*\x,-2) [below] {$T_1$};
\node at (8.5*\x,-2) [below] {$T_2$};
\node at (17*\x,-2) [below] {$T_3$};

\end{tikzpicture}
    \caption{Illustration of the recursion for the $(z^2+i)$-graphs $(T_n)_{n=0,1,2,3}$. The thicker black vertices with their bold labels correspond to the marked vertices in each graph, labeled $1,2,3$. The black vertices all together with their normal font labels correspond to the marked vertices of the copies of the graph from the previous step.}
    \label{fig: dendrite}
\end{figure}
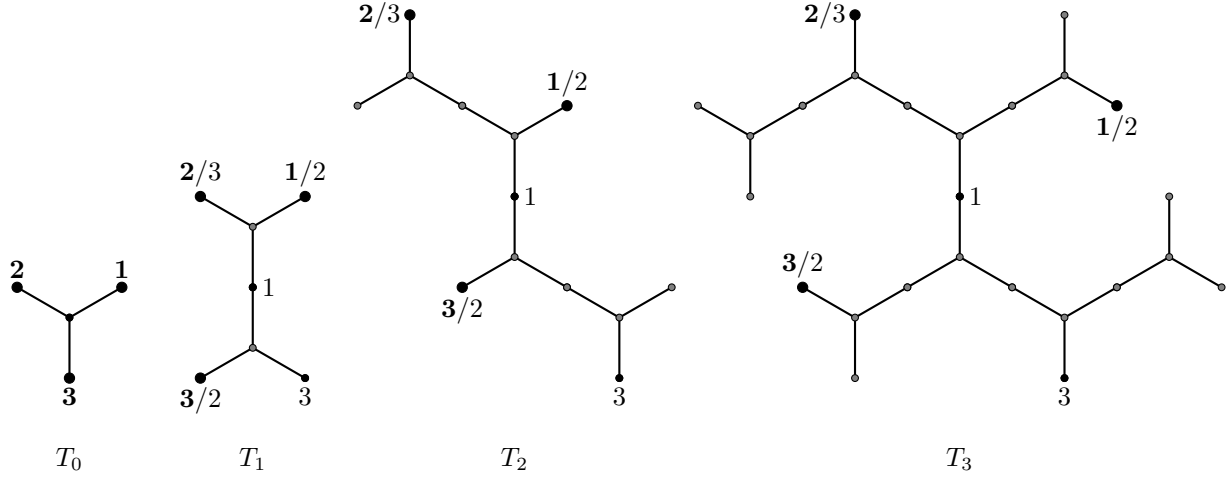

\begin{example}[$(z^2+i)$-gluing data]\label{ex: dendrite}
   Our final example is a graph recursion induced by the dynamics of the quadratic polynomial $z^2+i$ on the complex plane. The corresponding gluing data $(H,\Phi)$ has parameters $k=3$ and $m=2$. The gluing scheme $H$ is a multi-graph with the vertex set $V(H)=\{1,2\}$ and edge multi-set $E(H)$ given by the union of the following three partitions of $V(H)$:
   \[\lbrace\{1,2\}\rbrace, \; \lbrace\{1\},\{2\}\rbrace, \; \text{and }\; \lbrace\{1\},\{2\}\rbrace.\]
   We label the edges in the three partitions above by $1$, $2$, and $3$, respectively. The labeling map $\Phi$ is now defined as follows: we set $\Phi(1)$ to be the edge $\{1\}$ labeled $2$, $\Phi(2)$ to be the edge $\{1\}$ labeled $3$, and $\Phi(3)$ to be the edge $\{2\}$ labeled $2$. Figure~\ref{fig: dendrite} illustrates the iteration of $\RR_{(H,\Phi)}$ when we choose the initial graph $T_0$ to be the tripod graph (i.e., a $3$-star) with labeled 
   leaves. The graph $T_0$ corresponds to the \emph{Hubbard tree} of the polynomial $p(z)=z^2+i$ (see, for example, \cite[Definition~4.1]{DH_Orsay} for the definition), while the trees $T_n$ ($n\geq 1$) correspond to the iterated preimages of the Hubbard tree under $p$. We call the gluing data $(H,\Phi)$ the \emph{$(z^2+i)$-gluing data}. Note that $(H,\Phi)$ is non-degenerate and expanding. We remark that, similarly, we may define a gluing data for every \emph{Misiurewicz polynomial}, that is, for a polynomial whose finite critical points are strictly preperiodic. All such gluing data will be non-degenerate and expanding as well.
\end{example}

\section{Renormalization map and the invariant variety}\label{sec: renormalization map} 

In Section~\ref{subsec: induced dynamics}, we introduce a one-parameter family of rational dynamical systems $F_\lambda: \P^{2^k-1}\dashrightarrow  \P^{2^k-1}$, where $\lambda\in\C^*$, induced by a graph recursion $\RR_{(H,\Phi)}: \GG_k\to \GG_k$. In Section~\ref{subsec: invariant variety}, we define a (canonical) invariant subvariety $\MM\subset \P^{2^k-1}$, and in Sections~\ref{subsec: dynamics on variety} and \ref{subsec: dynamics near variety}, we study the dynamics of $F_\lambda$ in terms of the gluing data $(H,\Phi)$, respectively on and near the variety $\MM$.

\subsection{Renormalization map for the independence polynomial on recursive graphs} \label{subsec: induced dynamics}
Let us fix a gluing
data $(H,\Phi)$ with parameters $k$ and $m$, and consider the associated graph recursion $\RR=\RR_{(H,\Phi)}$. Suppose $G\in \GG_k$, and let $P=P(G)\subset V(G)$ be the corresponding subset of marked vertices in $G$. From now on, to simplify notation, we identify a binary $k$-tuple $\xbf = (x_1, \dots, x_k)\in \{0,1\}^k$ with the vertex assignment $\tau_\xbf: P \to \{0,1\}$ on $P$ that assigns the value $x_j$ to the marked vertex labeled $j\in\{1,\dots,k\}$. In particular, we write 
$Z_G^\xbf(\lambda)=Z^{\tau_\xbf}_G(\lambda)$ for the associated $\xbf$-conditioned independence polynomial (see Section~\ref{subsec: indep-poly}), so that 
\begin{equation}\label{eq: sum of conditioned}
Z_G(\lambda)=\sum_{\xbf\in\{0,1\}^k} Z_G^\xbf(\lambda).
\end{equation}

For each parameter $\lambda\in\C$, let us consider the map $\widehat{\phi}_\lambda: \GG_k \to \C^{2^k}$ defined by 
$$G\mapsto \big(Z^{(0,\ldots,0)}_G(\lambda), \ldots, Z^{(1,\dots,1)}_G(\lambda)\big),$$
where the superscripts run over all possible binary $k$-tuples $\xbf\in\{0,1\}^k$ in lexicographical order. (The specific choice of ordering on $\{0,1\}^k$ is immaterial as long as it is consistent.) We denote by $\phi_\lambda: \GG_k \dashrightarrow\P^{2^k-1}$ the induced map to the respective complex projective space, that is,
$$
\phi_\lambda(G) = [Z_G^{(0, \ldots, 0)}(\lambda): \cdots : Z_G^{(1,\ldots, 1)}(\lambda)].
$$
It is conceivable that for some graph $G$ and some isolated values of $\lambda \in \C$, the coordinates $Z_G^\xbf(\lambda)$ vanish for all $\xbf\in \{0,1\}^k$, in which case $\phi_\lambda(G)$ is not defined as a point of $\P^{2^k-1}$. However, it will turn out that we do not need to be concerned about this issue.

Let us consider the image $\RR(G)\in \GG_k$ of the graph $G$ under the graph recursion $\RR=\RR_{(H,\Phi)}$. It turns out that, while it is typically not possible to express the independence polynomial $Z_{\RR(G)}(\lambda)$ as a function of $Z_G(\lambda)$, there is in fact a rational self-map $F_\lambda$ of $\P^{2^k-1}$ that expresses $\phi_\lambda(\RR(G))$ in terms of $\phi_\lambda(G)$. The values of the orbit $\big(F_\lambda^n(\phi_\lambda(G))\big)_{n\geq 0}$ are in general not sufficient to recover the values $Z_{\RR^n(G)}(\lambda)$ of the corresponding independence polynomials, but by the homogeneity of the sum in equation~\eqref{eq: sum of conditioned} they are sufficient to determine whether these independence polynomials vanish at $\lambda$.

We now describe the rational map $F_\lambda$. Recall that the graph $\RR(G)$ is constructed by first taking $m$ disjoint copies $G(1), \ldots, G(m)$ of $G$, then identifying some of the labeled vertices of these copies according to the gluing scheme $H$, and finally assigning $k$ marked vertices in the resulting graph according to the labeling map $\Phi$; see Definition~\ref{def:graph recursion}. Let $Y\subset V(\RR(G))\setminus P(\RR(G))$ denote the set of vertices of $\RR(G)$ that are induced by the edges in $E(H)\setminus \Phi(\{1,\dots, k\})$. In other words, $Y$ is the set of all vertices of $\RR(G)$ that correspond to labeled vertices in the copies $G(i)$ but are no longer marked in $\RR(G)$. 

We note that for each binary $k$-tuple $\xbf\in \{0,1\}^k$ we can write 
$$
Z_{\RR(G)}^\xbf(\lambda)= \sum_\eta Z_{\RR(G)}^{\xbf \wedge \eta}(\lambda), 
$$
where the sum is taken over all possible vertex assignments $\eta$ on $Y$. For each $i\in \{1,\dots, m\}$, the vertex assignment $\xbf \wedge \eta$ induces a vertex assignment $[\xbf \wedge \eta]_i$ on the set $P(G(i))$ of marked vertices in the copy $G(i)$, so that the vertices that get identified in $\RR(G)$ receive identical values. On a more intuitive level, the assignment $\xbf \wedge \eta$ on $P(\RR(G))\sqcup Y$ ``splits'' into $m$ vertex assignments $[\xbf \wedge \eta]_1, \dots, [\xbf \wedge \eta]_m$ on the marked vertices of the copies $G(1), \dots, G(m)$. In what follows, with a slight abuse of notation, we will also regard each $[\xbf \wedge \eta]_i$ as a vertex assignment on $P\subset V(G)$.

\begin{lemma}[$\xbf$-conditioned independence polynomial of $\RR(G)$]\label{lem:x_prime} Suppose $\lambda\in \C^*$. 
For each $\xbf\in \{0,1\}^k$ and $\eta: Y\to \{0,1\}$, 
let 
$$
\|\xbf \wedge \eta\|_H  := \sum_e (\#e-1) = \sum_{i=1}^m
\#[\xbf\wedge\eta]_i^{-1}(1)- \#(\xbf\wedge \eta)^{-1}(1),
$$
where the first sum is taken over all edges $e\in E(H)$ that correspond to the vertices of $\RR(G)$ receiving value~$1$ from the assignment $\xbf\wedge \eta$.  
Then 
\begin{equation}\label{eq: x_and_y_prime_conditioned}
Z_{\RR(G)}^{\xbf\wedge \eta}(\lambda) = \lambda^{-\|\xbf\wedge \eta\|_H} \prod_{i=1}^m Z_{G(i)}^{[\xbf\wedge \eta]_i}(\lambda),
\end{equation}
and hence
\begin{equation}\label{eq: x_prime_conditioned}
Z_{\RR(G)}^{\xbf}(\lambda) =\sum_{\eta: Y\to \{0,1\}} \lambda^{-\|\xbf\wedge \eta\|_H} \prod_{i=1}^m Z_{G}^{[\xbf\wedge \eta]_i}(\lambda).
\end{equation}
\end{lemma}
\begin{proof}
Intuitively, the exponent $\|\xbf\wedge\eta\|_H$ in \eqref{eq: x_and_y_prime_conditioned} records the overcounting of vertices receiving a $1$ from the assignment $\xbf\wedge\eta$, and which arise from the identifications prescribed by the gluing scheme $H$. To see this formally, let $\xbf\in \{0,1\}^k$ and $\eta: Y\to \{0,1\}$ be arbitrary. We assume below that the assignment $\xbf \wedge \eta$ is  admissible; for otherwise, both sides of \eqref{eq: x_and_y_prime_conditioned} vanish, and thus it trivially holds. Since only the labeled vertices of the disjoint copies $G(i)$ are getting identified to form the graph $\RR(G)$, it follows that removing all vertices in $P(\RR(G))\sqcup Y$ decomposes the graph $\RR(G)$ into $m$ disjoint copies of the induced graph $G[V(G)\setminus P]$. 

For each $i\in\{1,\dots,m\}$, let $G'(i)$ be the induced subgraph obtained from the copy $G(i)$ by removing every labeled vertex that receives a $0$ from the assignment $[\xbf \wedge \eta]_i$, as well as the closed neighborhood of each labeled vertex that receives a $1$ from $[\xbf \wedge \eta]_i$. Note that each $G'(i)$ is a subgraph of the respective copy of $G[V(G)\setminus P]$. Properties \ref{prop:ind-poly-1} and \ref{prop:ind-poly-2} then imply that, up to multiplication by a suitable power of $\lambda$, the $\xbf\wedge\eta$-conditioned independence polynomial $Z_{\RR(G)}^{\xbf\wedge \eta}(\lambda)$ is the product of the independence polynomials of the graphs $G'(i)$. More precisely, we have
\[
Z_{\RR(G)}^{\xbf\wedge \eta}(\lambda) = \lambda^{\#(\xbf\wedge\eta)^{-1}(1)}\prod_{i=1}^m Z_{G'(i)}(\lambda).
\]
At the same time, for each $i\in \{1,\dots,m\}$ property \ref{prop:ind-poly-2} implies that
\[
Z_{G(i)}^{[\xbf\wedge \eta]_i}(\lambda) = \lambda^{\#[\xbf\wedge\eta]_i^{-1}(1)} Z_{G'(i)}(\lambda).
\]
Substituting this expression into the previous identity yields \eqref{eq: x_and_y_prime_conditioned}, which in turn implies \eqref{eq: x_prime_conditioned}. This finishes the proof. 
\end{proof}

Note that, for each fixed $\lambda\in \C^*$, we may interpret the right-hand side of equation~\eqref{eq: x_prime_conditioned} as a homogeneous polynomial of degree $m$ in the variables $Z_G^\ybf(\lambda)$, where $\ybf\in \{0,1\}^k$. Moreover, this polynomial depends only on the gluing data $(H,\Phi)$ and is independent of the specific graph $G\in \GG_k$. We therefore obtain the following corollary.

\begin{corollary}\label{cor: renormalization map}
    Let $\RR=\RR_{(H,\Phi)}$ be the graph recursion operator associated with a gluing data $(H,\Phi)$ with parameters $k,m$. For each fixed parameter $\lambda\in \C^*$, there exists a homogeneous polynomial self-map $\widehat{F}_\lambda:\C^{2^k}\to\C^{2^k}$ of degree $m$ (i.e., every coordinate map of $\widehat{F}_\lambda$ is given by a homogeneous polynomial of degree $m$) such that 
    $$
    \widehat{F}_\lambda(\widehat{\phi}_\lambda(G)) = \widehat{\phi}_\lambda(\RR(G))
    $$
    for all $G\in \GG_k$. Hence, $\widehat{F}_\lambda$ induces a rational self-map $F_\lambda: \P^{2^k-1}\dashrightarrow\P^{2^k-1}$ of degree $m$ such that 
    $$
    F_\lambda(\phi_\lambda(G)) = \phi_\lambda(\RR(G))
    $$
    for all $G\in \GG_k$ for which these values are well-defined. Moreover, the following diagram commutes:
    \begin{center}
\begin{tikzpicture}[->,>={Stealth[round]},shorten >=1pt,auto,semithick]
  \node (G1) {$\GG_k$};
  \node (G2) [right=2cm of G1] {$\GG_k$};
  \node (C1) [below=2cm of G1] {$\C^{2^k}$};
  \node (C2) [below=2cm of G2] {$\C^{2^k}$};
  \draw[->] (G1) to node {$\RR$} (G2);
  \draw[->] (C1) to node {$\widehat{F}_\lambda$} (C2);

  \draw[->] (G1) to node {$\widehat{\phi}_\lambda$} (C1);
  \draw[->] (G2) to node {$\widehat{\phi}_\lambda$} (C2);

   \node (P1) [below=2cm of C1] {$\P^{2^k-1}$};
  \node (P2) [below=2cm of C2] {$\P^{2^k-1}$.};
    \draw[->, dashed] (P1) to node {$F_\lambda$} (P2);
  \draw[->,dashed] (C1) to node {} (P1);
  \draw[->,dashed] (C2) to node {} (P2);

  \node (Z0) [below=1cm of G2] {};
  \node (Z) [right=2cm of Z0] {$\C$};
    \draw[->,-latex] (G2)--(Z) node[midway,above,sloped] () {\small $G\mapsto Z_G(\lambda)$};  
        \draw[->,-latex] (C2)--(Z) node[midway,below,sloped] () {\small $\Sigma$};  
\end{tikzpicture}
\end{center}
    Here, $\Sigma: \C^{2^k}\to \C$ denotes the projection map given by the sum of the coordinates, and $\C^{2^k}\dashrightarrow \P^{2^{k}-1}$ represents the canonical projection map. 
\end{corollary}
We will refer to the rational self-map $F_\lambda$ from Corollary~\ref{cor: renormalization map} as the \emph{renormalization map} (for the independence polynomial) associated with the graph recursion $\RR$.

\begin{remark}
Since we will consistently consider the dynamics on $\P^{2^k-1}$ rather than on $\C^{2^k}$, we have chosen to use the simpler notation $F_\lambda$ (without the hat) for the map on projective space, and similarly use the hatless notation $\phi_\lambda$ for the map from $\GG_k$ into projective space. In our previous paper \cite{HP2024}, the roles of $F_\lambda$ and $\widehat{F}_\lambda$ were reversed.
\end{remark}

We note that the projection map $\Sigma: \C^{2^k}\to \C$
does not descend to the projective space $\P^{2^k-1}$.  However, since this map is homogeneous, its zero set is well-defined in projective space. As we are only interested in the zeros of the independence polynomials, we can therefore work entirely in projective space: the values of the orbit $\big(F_\lambda^n(\phi_\lambda(G))\big)_{n\geq 0}$ are sufficient to determine whether the independence polynomials $Z_{\RR^n(G)}(\lambda)$ vanish at $\lambda$.

\begin{convention}
Given a graph $G\in \GG_k$, a binary $k$-tuple $\xbf=(x_1,\dots, x_k)\in \{0,1\}^k$, and  a fixed parameter $\lambda\in \C$, we will use the shorter notation $(\xbf)=(\!(x_1,\dots, x_k)\!)$ for $Z_G^\xbf(\lambda)$ and $(\xbf)^\prime=(\!(x_1,\dots, x_k)\!)'$ for $Z_{\RR(G)}^\xbf(\lambda)$, so that the respective self-map $F_\lambda$ on the projective space $\P^{2^k-1}$ is given by
$$
[\, (\!(0,\ldots, 0)\!): \cdots : (\!(1,\ldots, 1)\!) \,] \mapsto [\, (\!(0,\ldots, 0)\!)^\prime: \cdots : (\!(1,\ldots, 1)\!)^\prime \,].
$$
In fact, we will also use the notation $(\xbf)=(\!(x_1,\dots, x_k)\!)$ for the respective coordinate in $\C^{2^k}$. Furthermore, we write $\obf := (0, \dots , 0)$ for
the zero vector and $\ebf_j := (0, \dots, 0, 1, 0, \dots, 0)$ for the $j$-th unit vector in $\{0,1\}^k$ for each $j\in \{1,\dots,k\}$.
\end{convention}

\begin{example}
    Let us revisit the recursion operator $\RR$ associated with the Sierpi\'nski gluing data $(H, \Phi)$ with parameters $k=m=3$ from Example~\ref{ex: sierpinksi}. Recall that $V(H)=\{1,2,3\}$ and the graph $H$ has six hyperedges: three $1$-edges $\{1\}, \{2\}, \{3\}$ and three $2$-edges $\{1,2\}, \{1,3\}, \{2,3\}$. Furthermore, we have  $\Phi(j)=\{j\}$ for each $j\in \{1,2,3\}$. Fix $G\in \GG_k$ and let $\xbf = (x_1, x_2, x_3) \in \{0,1\}^3$ be a vertex assignment on the labeled vertices of $\RR(G)$, which correspond to the three $1$-edges in $E(H)$. As before, let $Y$ denote the set of vertices of $\RR(G)$ that correspond to the edges in $E(H)\setminus\Phi(\{1,2,3\})$, that is, to the $2$-edges in $E(H)$. Consider all possible vertex assignments $\eta$ on $Y$, encoded by binary triples $\ybf=(y_1, y_2, y_3)$, so that $y_j$ is the value received by the vertex of $\RR(G)$ corresponding to the $2$-edge $\{1,2,3\}\setminus\{j\}$ for each $j\in \{1,2,3\}$. Following the convention above, we obtain from \eqref{eq: x_prime_conditioned} that the  corresponding homogeneous self-map $\widehat{F}_\lambda: \C^8\to\C^8$ is given by
    $$
    (\!(x_1, x_2, x_3)\!)^\prime = \sum_{\ybf\in \{0,1\}^3} \lambda^{-(y_1+y_2+y_3)} \cdot (\!(x_1, y_2, y_3)\!) (\!(y_1, x_2, y_3)\!) (\!(y_1, y_2, x_3)\!).
    $$
\end{example}

\subsection{Subvariety $\MM$ of uncorrelated marked vertices}\label{subsec: invariant variety} Let $k\in \N$, $G\in \GG_k$, and $\lambda\in \C$ be fixed. As before, we denote by $P\subset V(G)$ the set of marked vertices in $G$. Recall also from Section~\ref{subsec: probability interpretation} that, given a vertex assignment $\tau$ on a subset $X\subset V(G)$, we define
\[
\P_{G,\lambda}[\tau]=\frac{Z^\tau_G(\lambda)}{Z_G(\lambda)}.
\]
In general, $\P_{G,\lambda}[\tau]\in \widehat{\C}=\C\cup\{\infty\}$, as the dependence on $\lambda$ is rational. However, when the parameter $\lambda\in \R_+$, the value $\P_{G,\lambda}[\tau]\in [0,1]$ and it represents the probability that a randomly chosen independent subset of $V(G)$ agrees with the assignment $\tau$ on $X$. 

\begin{definition}[Absence of direct correlations]
    We say that there is \emph{no direct correlation} between the marked vertices of a graph $G\in \GG_k$ at a given parameter $\lambda\in \C$ if for each pair of vertex assignments $\tau$ and $\sigma$ on disjoint subsets of the marked vertices of $G$ we have
    \begin{equation}\label{eq: no direct correlation}
        Z_G(\lambda) \cdot Z_G^{\tau\wedge\sigma}(\lambda) -Z_G^\tau(\lambda) \cdot  Z_G^\sigma(\lambda) = 0.
    \end{equation}
\end{definition}

When $Z_G(\lambda)$ does not vanish, equation~\eqref{eq: no direct correlation} is equivalent to the probabilistic identity
$$
\P_{G,\lambda}[\tau \wedge \sigma] = \P_{G,\lambda}[\tau] \cdot \P_{G,\lambda}[\sigma],
$$
which means that the events $\tau$ and $\sigma$ are independent at the parameter $\lambda$. However, when $Z_G(\lambda)$ vanishes, equation~\eqref{eq: no direct correlation} does not convey sufficient information for our purposes. In this case, we require a more general (algebraic) notion of correlation between marked vertices that corresponds to the probabilistic condition
\begin{equation}\label{eq: no correlation prob}
\P_{G,\lambda}[\tau \wedge \sigma \wedge \eta] \cdot \P_{G,\lambda}[\eta] = \P_{G,\lambda}[\tau \wedge \eta] \cdot \P_{G,\lambda}[\sigma \wedge \eta],
\end{equation}
where $\tau$, $\sigma$, and $\eta$ are assignments on pairwise disjoint subsets of the marked vertices of $G$. Note that, under the assumption $\P_{G,\lambda}[\eta]\neq 0$, the condition above is equivalent to the identity
$$
\P_{G,\lambda}[(\tau \wedge \sigma) : \eta] = \P_{G,\lambda}[\tau : \eta] \cdot \P_{G,\lambda}[\sigma : \eta],
$$
which means that $\tau$ and $\sigma$ are \emph{conditionally independent} given the (auxiliary) assignment $\eta$.

The following definition interprets  probabilistic condition~\eqref{eq: no correlation prob} algebraically using the corresponding conditioned independence polynomials.

\begin{definition}[Absence of correlations]
    We say that there is \emph{no correlation} between the marked vertices of a graph $G\in \GG_k$ at a given parameter $\lambda\in \C$ if for each triple of vertex assignments $\tau$, $\sigma$, and $\eta$ on pairwise disjoint subsets of the marked vertices of $G$ we have
    \begin{equation}\label{eq: no correlation}
        Z_G^{\tau\wedge \sigma \wedge \eta}(\lambda) \cdot Z_G^{\eta}(\lambda) -Z_G^{\tau \wedge\eta}(\lambda) \cdot  Z_G^{\sigma \wedge \eta}(\lambda) = 0.
    \end{equation}
\end{definition}

Since $\eta$ may be chosen to be the empty assignment, it is immediate that absence of correlations implies absence of direct correlations. The reverse implication also often holds.

\begin{lemma}\label{lem: direct vs conditional correlation}
    When $\lambda\in \C$ satisfies $Z_G(\lambda) \neq 0$, absence of direct correlations between the marked vertices of $G$ at $\lambda$ implies absence of correlations between the marked vertices of $G$ at $\lambda$.
\end{lemma}
\begin{proof}
    Suppose $\tau$, $\sigma$, and $\eta$ are arbitrary vertex assignments on pairwise disjoint subsets of the marked vertices of $G$. Using~\eqref{eq: no direct correlation}, we deduce
    \[
    Z_G(\lambda)\cdot Z_G^{\tau\wedge \sigma \wedge \eta}(\lambda)\cdot Z_G^{\eta}(\lambda) = Z_G^{\tau\wedge \eta}(\lambda) \cdot  Z_G^{\sigma}(\lambda)\cdot Z_G^{\eta}(\lambda)
    = Z_G^{\tau\wedge \eta}(\lambda) \cdot Z_G(\lambda)\cdot Z_G^{\sigma \wedge \eta} (\lambda).
    \]
    Therefore, \eqref{eq: no correlation} holds whenever $Z_G(\lambda) \neq 0$.
\end{proof}

We may express $Z_G(\lambda)$ as the sum \eqref{eq: sum of conditioned} of the variables $(\xbf)=Z^\xbf_G(\lambda)$ with $\xbf\in\{0,1\}^k$, and each conditioned independence polynomial in \eqref{eq: no correlation} as a corresponding subsum of \eqref{eq: sum of conditioned}. It follows that equation~\eqref{eq: no correlation} can be interpreted as the vanishing of a homogeneous quadratic polynomial in the variables $(\xbf)$, $\xbf\in\{0,1\}^k$. Consequently, the collection of equations~\eqref{eq: no correlation} over all possible vertex assignments $\tau$, $\sigma$, and $\eta$ determines an algebraic subvariety in the projective space $\P^{2^k-1}$. We may describe this variety in a more concise way as follows. 

\begin{lemma}[Subvariety of uncorrelated marked vertices] \label{lem: inv variety}
    There is no correlation between the marked vertices of $G\in\GG_k$ at $\lambda\in \C$ if and only if the coordinates  $({\bf x}) = Z^{\xbf}_G(\lambda)$, $\xbf\in\{0,1\}^k$, of $\widehat{\phi}_\lambda(G)$ satisfy the homogeneous quadratic equations
    \begin{equation}\label{eq: no correlation variety}
    (\xbf) \cdot (\ybf) = (\zbf) \cdot(\wbf),
    \end{equation}
    where $\xbf,\, \ybf, \zbf, \wbf\in \{0,1\}^{k}$ run over all possible binary $k$-tuples satisfying $\xbf+\ybf=\zbf+\wbf$.
\end{lemma}

In the lemma above, given two binary $k$-tuples  $\xbf=(x_1,\dots,x_k)$ and $\ybf=(y_1,\dots, y_k)$, we write \[\xbf+\ybf:=(x_1+y_1,\dots,x_k+y_k)\] for their component-wise sum. In the proof below, we will also use the following notations:
\begin{align*}
    \min(\xbf, \ybf)
    &:=(\min\{x_1,y_1\},\dots,\min\{x_k,y_k\}),\\
     \max(\xbf, \ybf)
     &:=(\max\{x_1,y_1\},\dots,\max\{x_k,y_k\}).
\end{align*}
Note that $\min(\xbf, \ybf)$ and $\max(\xbf, \ybf)$ are binary $k$-tuples by definition. Moreover, for all $\xbf,\, \ybf, \zbf, \wbf\in \{0,1\}^{k}$ we have 
\begin{equation}\label{eq: inv reduction}
 \xbf+\ybf=\zbf+\wbf \quad \Leftrightarrow \quad  
    \min(\xbf, \ybf) =  \min(\zbf, \wbf) \; \text{ and } \;     \max(\xbf, \ybf) =  \max(\zbf, \wbf).   
\end{equation}

\begin{proof}[Proof of Lemma~\ref{lem: inv variety}]
Let us first show that if equations~\eqref{eq: no correlation variety} are satisfied by the coordinates of $\widehat{\phi}_\lambda(G)$ for some graph $G\in \GG_k$ and $\lambda\in \C$, then there is no correlation between the marked vertices of $G$ at $\lambda$. Suppose we are given three vertex assignments $\tau$, $\sigma$, and $\eta$ on pairwise disjoint subsets $S_\tau$, $S_\sigma$, and $S_\eta$ of the marked set $P=P(G)$, respectively. Let $S_\upsilon:=P\setminus (S_\tau\sqcup S_\sigma\sqcup S_\eta)$ and $\upsilon$ be an arbitrary vertex assignment on $S_\upsilon$. Then $\tau\wedge \sigma\wedge \eta\wedge\upsilon$ is an assignment on the full marked set $P$, and thus it may be naturally viewed as a binary $k$-tuple representing a coordinate in $\C^{2^k}$.  Using~\eqref{eq: no correlation variety}, we then obtain  
\begin{align*}
Z_G^{\tau\wedge \sigma \wedge \eta}(\lambda) \cdot Z_G^{\eta}(\lambda) &= \sum_{\upsilon: S_\upsilon\to\{0,1\} }Z_G^{\tau\wedge \sigma \wedge \eta\wedge \upsilon}(\lambda) \;\; \times \sum_{\substack{{\widetilde{\tau}: S_\tau\to\{0,1\},\;\widetilde{\sigma}: S_\sigma\to\{0,1\},} \\ {\widetilde{\upsilon}: S_\upsilon\to\{0,1\}}}} Z_G^{\widetilde{\tau}\wedge \widetilde{\sigma} \wedge \eta\wedge \widetilde{\upsilon}}(\lambda)\\
&= \sum_{\substack{{\widetilde{\tau}: S_\tau\to\{0,1\},\;\widetilde{\sigma}: S_\sigma\to\{0,1\},} \\ {\widetilde{\upsilon}: S_\upsilon\to\{0,1\}}, \; \upsilon: S_\upsilon\to\{0,1\}}} (\tau\wedge \sigma \wedge \eta\wedge \upsilon)\cdot (\widetilde{\tau}\wedge \widetilde{\sigma} \wedge \eta\wedge \widetilde{\upsilon})\\
&= \sum_{\substack{{\widetilde{\tau}: S_\tau\to\{0,1\},\;\widetilde{\sigma}: S_\sigma\to\{0,1\},} \\ {\widetilde{\upsilon}: S_\upsilon\to\{0,1\}}, \; \upsilon: S_\upsilon\to\{0,1\}}} (\tau\wedge \widetilde{\sigma} \wedge \eta\wedge \upsilon)\cdot (\widetilde{\tau}\wedge \sigma \wedge \eta\wedge \widetilde{\upsilon})\\
&= \sum_{\substack{\widetilde{\sigma}: S_\sigma\to\{0,1\},\\ \upsilon: S_\upsilon\to\{0,1\}}} Z_G^{\tau\wedge \widetilde{\sigma} \wedge \eta\wedge \upsilon}(\lambda)\;\; \times \sum_{\substack{\widetilde{\tau}: S_\tau\to\{0,1\},\\ \widetilde{\upsilon}: S_\upsilon\to\{0,1\}}} Z_G^{\widetilde{\tau}\wedge \sigma \wedge \eta\wedge \widetilde{\upsilon}}(\lambda)\\
&= Z_G^{\tau\wedge \eta}(\lambda) \cdot Z_G^{\sigma\wedge\eta}(\lambda),
\end{align*}
which implies the desired identity~\eqref{eq: no correlation}.

To prove the converse, assume that there is no correlation between the marked vertices of $G$ at some $\lambda\in \C$.
Let $\xbf,\ybf \in \{0,1\}^k$ be two arbitrary binary $k$-tuples. By~\eqref{eq: inv reduction}, it is sufficient to prove that
$$
(\xbf) \cdot  (\ybf) = (\max(\xbf,\ybf)) \cdot (\min(\xbf,\ybf)),
$$
where $({\bf x}) = Z^{\xbf}_G(\lambda)$ and $({\bf y}) = Z^{\ybf}_G(\lambda)$. 

Let $S\subset P$ be the subset of marked vertices receiving identical values from $\xbf$ and $\ybf$. On the remaining marked vertices, exactly one of $\xbf$ and $\ybf$ assigns a $1$ and the other assigns a $0$. Let $Q$ (resp.\ $R$) be the subset of those marked vertices that receive a $1$ (resp.\ a $0$) from $\xbf$ and a $0$ (resp.\ a $1$) from~$\ybf$. Clearly, the sets $Q$, $R$, and $S$ form a partition of $P$. 

Let $\upsilon$ denote the vertex assignment on $S$ that assigns the same value to each vertex as $\xbf$ (and $\ybf$). Following the notation in Section~\ref{subsec: indep-poly}, the vertex assignments $\xbf$ and $\ybf$ coincide with $\ibf_Q \wedge \obf_R \wedge \upsilon$ and $\obf_Q \wedge \ibf_R \wedge \upsilon$ , respectively. Consequently, our goal is to show that 
$$
(\ibf_Q \wedge \obf_R \wedge \upsilon) \cdot  (\obf_Q \wedge \ibf_R \wedge \upsilon) = (\ibf_Q \wedge \ibf_R \wedge \upsilon) \cdot (\obf_Q \wedge \obf_R \wedge \upsilon).
$$

To simplify notation, given any partial vertex assignment $\zeta$ on $P$ we will write $(\zeta)$ for $Z^{\zeta}_G(\lambda)$ below. Using the inclusion--exclusion principle twice, we can then write 
\begin{align*}
(\obf_Q\wedge 
\ibf_R \wedge\upsilon) &= \sum_{{\widetilde Q}\subset Q} (-1)^{\#{\widetilde Q}} \cdot  (\ibf_{\widetilde Q} \wedge \ibf_R\wedge  \upsilon)\\ &=\sum_{{\widetilde Q}\subset Q} (-1)^{\#{\widetilde Q}} \sum_{{\widetilde R}\subset R} (-1)^{\#{\widetilde R}}\cdot (\ibf_{\widetilde Q} \wedge \obf_{\widetilde R}\wedge  \upsilon)\\
&=\sum_{{\widetilde Q}\subset Q,\, {\widetilde R}\subset R} (-1)^{\#{\widetilde Q}+\#{\widetilde R}} \cdot (\ibf_{\widetilde Q} \wedge \obf_{\widetilde R}\wedge  \upsilon)
\end{align*}
and thus
\[
(\ibf_Q \wedge \obf_R \wedge \upsilon) \cdot  (\obf_Q\wedge \ibf_R\wedge \upsilon) = \sum_{{\widetilde Q}\subset Q,\, {\widetilde R}\subset R} (-1)^{\#{\widetilde Q}+\#{\widetilde R}} \cdot (\ibf_Q \wedge \obf_R \wedge \upsilon) \cdot (\ibf_{\widetilde Q} \wedge \obf_{\widetilde R}\wedge  \upsilon).
\]

By the assumption of no correlation between the marked vertices of $G$ at $\lambda$, for all subsets ${\widetilde Q}\subset Q$ and ${\widetilde R}\subset R$ we have
\[
(\ibf_Q \wedge \obf_R \wedge \upsilon) \cdot  (\ibf_{\widetilde Q} \wedge \obf_{\widetilde R} \wedge \upsilon) = (\ibf_Q \wedge \obf_{\widetilde R} \wedge \upsilon) \cdot  (\ibf_{\widetilde Q} \wedge \obf_R \wedge \upsilon),
\]
where we have applied \eqref{eq: no correlation} with $\tau = \ibf_{Q\setminus {\widetilde Q}}$, $\sigma=\obf_{R\setminus {\widetilde R}}$, and $\eta=\ibf_{\widetilde Q}\wedge \obf_{\widetilde R}\wedge\upsilon$. It follows that 
\begin{align*}
(\ibf_Q \wedge \obf_R \wedge \upsilon) \cdot  (\obf_Q\wedge \ibf_R\wedge \upsilon) &= \sum_{{\widetilde Q}\subset Q,\, {\widetilde R}\subset R} (-1)^{\#{\widetilde Q}+\#{\widetilde R}} \cdot (\ibf_Q \wedge \obf_{\widetilde R} \wedge \upsilon) \cdot  (\ibf_{\widetilde Q} \wedge \obf_R \wedge \upsilon)\\
&= \left(\sum_{{\widetilde R}\subset R}(-1)^{\#{\widetilde R}} \cdot (\ibf_Q \wedge \obf_{\widetilde R} \wedge \upsilon)\right)\times \left(\sum_{{\widetilde Q}\subset Q}(-1)^{\#{\widetilde Q}} \cdot (\ibf_{\widetilde{Q}}\wedge \obf_{R} \wedge \upsilon)\right)\\
&=(\ibf_Q \wedge \ibf_R \wedge \upsilon) \cdot (\obf_Q \wedge \obf_R \wedge \upsilon),
\end{align*}
where we have used the inclusion--exclusion principle twice to derive the last equality. Consequently, \[(\xbf) \cdot  (\ybf) = (\max(\xbf,\ybf)) \cdot (\min(\xbf,\ybf)).\]
Since $\xbf$ and $\ybf$ were arbitrary binary $k$-tuples, we conclude that equations~\eqref{eq: no correlation variety} hold, which finishes the proof.
\end{proof}

\begin{definition}[Invariant variety]\label{def: inv variety}
    Let $k\in \N$ and $\MM \subset \P^{2^k-1}$ be the algebraic subvariety defined by the homogeneous quadratic equations 
    \begin{equation*}
    (\xbf) \cdot (\ybf) = (\zbf) \cdot (\wbf),
    \end{equation*}
    where $\xbf,\, \ybf, \zbf, \wbf\in \{0,1\}^{k}$ run over all possible binary $k$-tuples satisfying $\xbf+\ybf=\zbf+\wbf$. We call $\MM$ the \emph{subvariety of uncorrelated marked vertices} or simply the \emph{invariant variety} for $\GG_k$. 
\end{definition}

Following the conventions in Section~\ref{subsec: induced dynamics}, the variables $(\xbf)$ with $\xbf\in \{0,1\}^k$ represent the coordinates in $\C^{2^k}$. We will see below that the algebraic subvariety $\MM$ plays a special role for the renormalization map $F_\lambda$ associated with a graph recursion operator on $\GG_k$; in particular, $F_\lambda$ leaves $\MM$ invariant (Proposition~\ref{prop: invariance}), which justifies this terminology.

\subsubsection{Iterated Segre embeddings}
\label{sss: segre_embeddings}
The invariant variety $\MM$ from Definition~\ref{def: inv variety} turns out to be a well-known construction in projective geometry, allowing us to deduce its geometric properties.

Given $m, n \in \N$, the \emph{Segre embedding} from $\P^m \times \P^n$ to $\P^{(m+1)(n+1)-1}$ is defined by
$$
\left([a_0 : \cdots : a_m], [b_0: \cdots : b_n ]\right) \mapsto [a_0b_0: a_0 b_1: \cdots : a_m b_{n-1}: a_m b_n].
$$
The image of this embedding is called a \emph{Segre variety}; we refer the reader to standard texts on algebraic geometry, e.g., \cite{Hartshorne, Shafarevich}. Using the notation $(\!(i,j)\!)$ for the coordinate in the image corresponding to the product $a_ib_j$, one obtains the following defining equations for the Segre variety:
\[
(\!(i,j)\!) \cdot (\!(s,t)\!) = (\!(s,j)\!) \cdot (\!(i,t)\!), \quad i,s\in \{0,\dots,m\}, \; j,t\in \{0,\dots,n\}.
\]
We note that when $m= n = 1$, these coincide with the defining equations of the invariant variety $\MM\subset\P^3$ obtained for $k=2$ marked vertices. 

The Segre varieties are smooth and irreducible. Smoothness can be seen by considering the image of the Segre embedding on each product of coordinate charts, while irreducibility follows since it is preserved under taking products and under closed embeddings; see also~\cite{Hartshorne}.

Let $\left(\P^1\right)^k$ denote the product $\P^1 \times \cdots \times \P^1$ with $k\in \N$ factors. Using the corresponding Segre embedding at each step of the following chain of embeddings
\begin{align*}
(\P^1)^k=(\P^1\times \P^1)\times(\P^1)^{k-2} \hookrightarrow \P^3\times(\P^1)^{k-2} = (\P^3\times \P^1)\times(\P^1)^{k-3}  \hookrightarrow \dots \hookrightarrow \P^{2^{k-1}-1}\times \P^1\hookrightarrow \P^{2^{k}-1},
\end{align*}
we obtain an embedding of $\left(\P^1\right)^k$
into $\P^{2^k-1}$, defined by
\begin{equation}\label{eq: segre embedding}
\left([a^1_0: a^1_1], \ldots, [a^k_0: a^k_1]\right) \mapsto [\cdots: (a^1_{x_1}\cdots a^k_{x_k}): \cdots],
\end{equation}
where the image contains products over all possible binary $k$-tuples $(x_1, \ldots, x_k) \in \{0,1\}^k$. We can therefore regard the $(\!(x_1, \ldots, x_k)\!)$  as coordinates in $\P^{2^k-1}$. We call the embedding \eqref{eq: segre embedding} the \emph{$k$-fold Segre embedding}, and the respective image of $\left(\P^1\right)^k$ in $\P^{2^k-1}$ the \emph{$k$-fold Segre variety}. Note that when $k=1$, this embedding is the identity and the respective variety is simply $\P^1$.

By construction, it follows immediately that equations~\eqref{eq: no correlation variety} are satisfied on the $k$-fold Segre variety. One can also verify by induction on the dimension $k$ that any point in $\P^{2^k-1}$ satisfying equations~\eqref{eq: no correlation variety} must lie in the image of the $k$-fold Segre embedding. Hence, the invariant variety $\MM$ from Definition~\ref{def: inv variety} coincides with the $k$-fold Segre variety. 

The arguments for smoothness and irreducibility carry over to the $k$-fold Segre variety for each $k\geq 2$. In summary, we obtain the following result. 

\begin{lemma}\label{lem: Segre}
    For each $k\in \N$ the invariant variety $\MM \subset \P^{2^k-1}$ for $\GG_k$ coincides with the $k$-fold Segre variety, which is a smooth and irreducible $k$-dimensional subvariety of $\P^{2^k-1}$.
\end{lemma}

    It is straightforward to verify that, in the affine chart $\{(\obf) \neq 0\}$ of $\P^{2^k-1}$, the invariant variety $\MM$ is given by the system of equations
    \begin{equation}\label{eq: variety in basis coordinates}
    \frac{(\xbf)}{(\obf)} = \prod_{j: \; x_j = 1} \frac{(\ebf_j)}{(\obf)},
    \end{equation}
    for all $\xbf=(x_1,\dots,x_k) \neq \obf$. 
    Consider the embedding $\Psi: (\C^*)^k\to \P^{2^k-1}$ of the standard torus $(\C^*)^k$ into the projective space $\P^{2^k-1}$ given by  
    \[(e_1,\dots, e_k) \mapsto [1: \cdots:\prod_{j: \; x_j = 1} e_j: \cdots ],\]
    meaning that the $(\!(x_1, \ldots, x_k)\!)$-coordinate of the image is given by the monomial $\prod_{j: \; x_j = 1} e_j$, except for the $(\obf)$-coordinate, which is simply $1$. 
    The irreducibility of $\MM$ and \eqref{eq: variety in basis coordinates} together imply that $\MM$ is the Zariski closure of the image of $\Psi$. (In fact, $\MM$ is also the closure of the image of $\Psi$ in the Euclidean topology induced from $\C^{2^k}$.) Hence, $\MM$ is the projective toric variety associated with the finite set \[\mathscr{A}:=\{0,1\}^k\subset \Z^k,\] consisting of the exponent vectors of the coordinate monomials of $\Psi$, and realized in $\P^{2^k-1}$ via the monomial embedding $\Psi$.  Note that the vectors in $\mathscr{A}$ are precisely the $2^k$ vertices of the unit cube in $\R^k$, so that their convex hull satisfies $\operatorname{conv}(\mathscr{A}) = [0,1]^k$. Consequently, $\MM$ is the (projective) toric variety of the polytope $\mathscr{P}=\operatorname{conv}(\mathscr{A})$ (see \cite{Cox} for general background on toric varieties).

\subsection{Action of the renormalization map on the invariant variety}\label{subsec: dynamics on variety}

Let us again fix a gluing
data $(H,\Phi)$ with parameters $k,m$, and consider the associated graph recursion $\RR=\RR_{(H,\Phi)}$ on $\GG_k$.  Suppose $F_\lambda : \P^{2^k-1}\dashrightarrow \P^{2^k-1}$ is the renormalization map associated with $\RR$ for some fixed parameter $\lambda\in \C^*$, and let $\widehat{F}_\lambda : \C^{2^k}\to \C^{2^k}$ be the respective homogeneous polynomial self-map; see Section~\ref{subsec: induced dynamics}. In the following, for each $n\in \N$, we denote by $\IS(F_\lambda^n)$ the \emph{indeterminacy set} of $F^n_\lambda$, that is, the subset of points in $\P^{2^k-1}$ at which the iterate $F^n_\lambda$ is not well-defined. Our first goal is to show the following statement.

\begin{prop}[Invariance of the subvariety $\MM$]\label{prop: invariance}
    The subvariety $\MM$ of uncorrelated marked vertices for $\GG_k$ (see Definition~\ref{def: inv variety}) is invariant under the renormalization map $F_\lambda$ (outside its indeterminacy set). 
\end{prop}

We emphasize that the invariant variety $\MM$ is independent of both the gluing data $(H,\Phi)$ and the parameter $\lambda$, and only depends on  $k$, the number of marked vertices. We also record the following immediate corollary of Proposition~\ref{prop: invariance} and Lemmas~\ref{lem: direct vs conditional correlation} and~\ref{lem: inv variety}.

\begin{corollary}\label{cor: invariance for graphs}
    Let $\lambda\in \R_+$ and $G\in \GG_k$. Assume that there is no correlation between the marked vertices of $G$ at $\lambda$. Then there is no correlation between the marked vertices of $\RR(G)$ at $\lambda$.
\end{corollary}

Before proving Proposition~\ref{prop: invariance}, let us first provide a direct intuitive proof of Corollary~\ref{cor: invariance for graphs}.  

\begin{proof}[Intuitive ``proof'' of Corollary~\ref{cor: invariance for graphs}]
    Since $\lambda\in \R_+$, we have $Z_G(\lambda) \neq 0$, and hence by Lemma~\ref{lem: direct vs conditional correlation} it is sufficient to show absence of direct correlations between the marked vertices of $\RR(G)$ at $\lambda$. To simplify the notation, we set $\P:=\P_{\RR(G),\lambda}$. 

    Suppose $\tau$ and $\sigma$ are arbitrary vertex assignments on two disjoint subsets  $X$ and $\widetilde{X}$ of the marked vertices of $\RR(G)$. We need to check that 
    \[
    \P[\tau \wedge \sigma]= \P[\tau] \cdot\P[\sigma].
    \]
    Since this clearly holds when $\P[\sigma]=0$, we will assume in the following that $\P[\sigma]\neq 0$. Then, it suffices to check that
    \[
    \P[\tau: \sigma]= \P[\tau].
    \]
    
    As in Section~\ref{subsec: induced dynamics}, let $Y\subset V(\RR(G))$ denote the set of vertices of $\RR(G)$ that are induced by the edges in $E(H)\setminus \Phi(\{1,\dots,k\})$.   
     Since there is no correlation between the marked vertices of $G$, the recursive construction of $\RR(G)$ implies that there is a constant $C(\tau)$ such that
     \[
         \P[\tau: \eta \wedge \sigma] = C(\tau) \tag{\textasteriskcentered} \label{eq: main prob fact}
     \]
    for all assignments $\eta: Y \rightarrow \{0,1\}$ with $\P[\eta\wedge \sigma]\neq 0$.
    
    We may now write: 
    \begin{equation}\label{eq: direct conditioned}
    \begin{aligned}
    \P [\tau : \sigma] &= \sum_{\eta: Y \rightarrow \{0,1\}} \P [\tau \wedge \eta : \sigma] = \sum_{\substack{\eta: Y \rightarrow \{0,1\},\\ \P[\eta\wedge \sigma]\neq 0}} \P [\tau \wedge \eta : \sigma] \\
    &=  \sum_{\substack{\eta: Y \rightarrow \{0,1\},\\ \P[\eta\wedge \sigma]\neq 0}} \P [\tau : \eta \wedge \sigma] \cdot \P [\eta : \sigma] =  \sum_{\substack{\eta: Y \rightarrow \{0,1\},\\ \P[\eta\wedge \sigma]\neq 0}} C(\tau) \cdot \P [\eta : \sigma] \\
    &=   C(\tau) \cdot \sum_{\substack{\eta: Y \rightarrow \{0,1\},\\ \P[\eta\wedge \sigma]\neq 0}} \P [\eta : \sigma] =   C(\tau).
    \end{aligned}
    \end{equation}

    Since \eqref{eq: direct conditioned} holds for all assignments $\sigma$ on $\widetilde{X}$ with $\P[\sigma]\neq0$, it follows that 
    \[
   \P[\tau]= \sum_{\substack{\sigma: \widetilde{X} \to \{0,1\},\\ \P[\sigma]\neq 0}} \P[\tau: \sigma] \cdot\P[\sigma] = C(\tau ) \cdot \sum_{\substack{\sigma: \widetilde{X} \to \{0,1\},\\ \P[\sigma]\neq 0}} \P[\sigma] = C(\tau),
    \]
   which finishes the proof.
\end{proof}

The intuitive identity~\eqref{eq: main prob fact} requires a formal justification, which would involve rewriting it in terms of the conditioned independence polynomials for $G$. Since we must work with the corresponding variables $(\xbf)=Z^\xbf_G(\lambda)$, $\xbf\in\{0,1\}^k$, anyway for general complex $\lambda$, we omit the proof of this identity and proceed directly to the proof of Proposition~\ref{prop: invariance}.

\begin{proof}[Proof of Proposition~\ref{prop: invariance}]
     Suppose the coordinates $(\xbf)$, $\xbf\in \{0,1\}^k$, of a point $\widehat{\xi}\in \C^{2^k}$ satisfy equations~\eqref{eq: no correlation variety}. We will show that the coordinates $(\xbf)'$, $\xbf\in \{0,1\}^k$, of the image point $\widehat{\xi}':=\widehat{F}_\lambda(\widehat \xi )\in \C^{2^k}$ satisfy these equations as well. The proposition then immediately follows. 

    Although the coordinates of $\widehat{\xi}$ may not correspond to the values of conditioned independence polynomials for an actual graph $G\in \GG_k$, we temporarily regard them as such in order to recall the definition of the image coordinates $({\bf x})^\prime$, $\xbf\in\{0,1\}^k$. We then view $\xbf$ as a vertex assignment on the set $P(\RR(G))$ of the marked vertices of the graph $\RR(G)$. By the discussion in Section~\ref{subsec: induced dynamics}, we have
    \[
      (\xbf)' =\sum_{\eta: Y\to \{0,1\}} \lambda^{-\|\xbf\wedge \eta\|_H} \prod_{i=1}^m ([\xbf\wedge \eta]_i),
    \]
    where    
    \[
    \|\xbf \wedge \eta\|_H  =\sum_{i=1}^m
    \#[\xbf\wedge\eta]_i^{-1}(1)- \#(\xbf\wedge \eta)^{-1}(1).
    \]    
    Here, $Y$ is the set of vertices of $\RR(G)$ that are induced by the edges in $E(H) \setminus \Phi(\{1,\dots, k\})$, and each $[\xbf\wedge \eta]_i\in \{0,1\}^k$, $i\in\{1,\dots,m\}$, denotes the assignment on the marked vertices of the $i$-th copy $G(i)$ induced by the vertex assignment $\xbf\wedge \eta$ on $P(\RR(G))\sqcup Y$.

    To check that the coordinates of $\widehat{\xi}'$ satisfy  equations~\eqref{eq: no correlation variety}, it suffices to verify that
    \[
        (\xbf)' \cdot (\widetilde{\xbf})'=(\max(\xbf, \widetilde{\xbf}))' \cdot  (\min(\xbf, \widetilde{\xbf}))'
    \]
    for all $\xbf, \widetilde{\xbf}\in \{0,1\}^k$. 
    Note that for all assignments $\eta, \widetilde{\eta}$ on $Y$  we have 
    \[\xbf\wedge \eta + \widetilde{\xbf}\wedge \widetilde{\eta}=\max(\xbf, \widetilde{\xbf}) \wedge \eta + \min(\xbf, \widetilde{\xbf})\wedge \widetilde{\eta},\]
    as well as 
    \[[\xbf\wedge \eta]_i + [\widetilde{\xbf}\wedge \widetilde{\eta}]_i=[\max(\xbf, \widetilde{\xbf}) \wedge \eta]_i + [\min(\xbf, \widetilde{\xbf})\wedge \widetilde{\eta}]_i\]
    for each $i\in \{1,\dots, m\}$, where we view each summand as a binary vector representing the corresponding vertex assignment. In particular, we have
    \[
    \|\xbf\wedge \eta\|_H + \|\widetilde{\xbf}\wedge \widetilde{\eta}\|_H = \|\max(\xbf, \widetilde{\xbf}) \wedge \eta\|_H + \|\min(\xbf, \widetilde{\xbf})\wedge \widetilde{\eta}\|_H,
    \]
    as well as 
    \[([\xbf\wedge \eta]_i) \cdot ([\widetilde{\xbf}\wedge \widetilde{\eta}]_i)= ([\max(\xbf, \widetilde{\xbf}) \wedge \eta]_i) \cdot ([\min(\xbf, \widetilde{\xbf})\wedge \widetilde{\eta}]_i)\]
    for each $i\in \{1,\dots, m\}$.

    Using the relations above, we compute:
     \begin{align*}
        (\xbf)'\cdot (\widetilde{\xbf})'&= \left(\sum_{\eta: Y\to \{0,1\}} \lambda^{-\|\xbf\wedge \eta\|_H} \prod_{i=1}^m ([\xbf\wedge \eta]_i)\right) \cdot
        \left(\sum_{\widetilde{\eta}: Y\to \{0,1\}} \lambda^{-\|\widetilde{\xbf}\wedge \widetilde{\eta}\|_H} \prod_{i=1}^m ([\widetilde{\xbf}\wedge \widetilde{\eta}]_i)\right)\\
        &= \sum_{\eta, \, \widetilde{\eta}: Y\to \{0,1\}} \lambda^{-\|\xbf\wedge \eta\|_H}\lambda^{-\|\widetilde{\xbf}\wedge \widetilde{\eta}\|_H} \prod_{i=1}^m ([\xbf\wedge \eta]_i)([\widetilde{\xbf}\wedge \widetilde{\eta}]_i)\\
         &= \sum_{\eta, \, \widetilde{\eta}: Y\to \{0,1\}} \lambda^{-\|\max(\xbf, \widetilde{\xbf}) \wedge \eta\|_H}\lambda^{-\|\min(\xbf, \widetilde{\xbf}) \wedge \widetilde{\eta}\|_H} \prod_{i=1}^m ([\max(\xbf, \widetilde{\xbf}) \wedge \eta]_i) ([\min(\xbf, \widetilde{\xbf})\wedge \widetilde{\eta}]_i)\\
         &=(\max(\xbf, \widetilde{\xbf}))'\cdot  (\min(\xbf, \widetilde{\xbf}))'.
    \end{align*}
    This finishes the proof. 
\end{proof}

Next, we address the dynamics of $F_\lambda$ on the invariant variety $\MM$. Recall from Definition~\ref{def: label dynamics}, that $\Lambda=\Lambda_{(H,\Phi)}$ denotes the dynamical system on the label set $\{1,\dots, k\}$ induced by the gluing data $(H,\Phi)$. In particular, for a graph $G\in \GG_k$, the vertex labeled $j\in \{1,\dots, k\}$ in the graph $\RR(G)$ is obtained by identifying the vertices labeled $\Lambda(j)$ in $\#\Phi(j)$ copies $G(i)$, where $i\in \Phi(j)$, of the graph $G$.

\begin{lemma}\label{lem: graph on 0-chart}
    On the chart $\{(\obf) \neq 0\}$, the action of $F_\lambda$ on the invariant variety $\MM$ is governed by the equations
    \[
    (\xbf)'= (\obf)'\cdot \prod_{j:\; x_j=1} \lambda^{1-\#\Phi(j)}  \left(\frac{(\ebf_{\Lambda(j)})}{(\obf)}\right)^{\#\Phi(j)}
    \]
    for all $\xbf=(x_1,\dots, x_k)\in \{0,1\}^k\setminus \{\obf\}$.
\end{lemma}

Here, similarly to the proof of Proposition~\ref{prop: invariance}, we assume that $(\xbf)$, $\xbf\in \{0,1\}^k$, represent the coordinates of a lift $\widehat{\xi}\in \C^{2^k}$ of a point $\xi\in \MM$, while  $(\xbf)'$, $\xbf\in \{0,1\}^k$, are the coordinates of the image $\widehat{\xi}':=\widehat{F}_\lambda(\widehat \xi )\in \C^{2^k}$. Lemma~\ref{lem: graph on 0-chart} implies that, if $\xi$ lies outside the indeterminacy set of $F_\lambda$ and belongs to the chart $\{(\obf) \neq 0$\}, then its image $\xi':=F_\lambda(\xi)$ lies in the same chart. Moreover, we have \[\left(\MM\cap \{(\obf) \neq 0\}\right) \cap \IS(F_\lambda) = \left(\MM\cap \{(\obf) \neq 0\}\right)  \cap \{(\obf)' = 0\}.\] 
Note that the image of the point $\xi=[1: 0: \cdots : 0]\in \MM\cap \{(\obf) \neq 0\}$ satisfies $(\obf)' = (\obf)^m=1$ by \eqref{eq: x_prime_conditioned}, and thus $\xi\notin\IS(F_\lambda)$. It follows that the set $\left(\MM\cap \{(\obf) \neq 0\}\right) \setminus \IS(F_\lambda)$ is open and dense in $\MM$.

\begin{proof}[Proof of Lemma~\ref{lem: graph on 0-chart}]
Fix an arbitrary $\xbf=(x_1,\dots, x_k)\in \{0,1\}^k\setminus \{\obf\}$. Following the notation in the proof of Proposition~\ref{prop: invariance}, for every vertex assignment $\eta$ on the set $Y$, we have
\begin{align*}
    \|\xbf \wedge \eta\|_H  &=\sum_{i=1}^m
    \#[\xbf\wedge\eta]_i^{-1}(1)- \#(\xbf\wedge \eta)^{-1}(1)\\
    &=\sum_{i=1}^m
    \#[\obf\wedge\eta]_i^{-1}(1)- \#(\obf\wedge \eta)^{-1}(1) + \sum_{j:\; x_j=1}(\#\Phi(j) - 1)\\
    &= \|\obf\wedge \eta\|_H +  \sum_{j:\; x_j=1}(\#\Phi(j) - 1).
\end{align*}
At the same time, using~\eqref{eq: variety in basis coordinates}, for each $i\in \{1,\dots, m\}$ we have
\[ 
([\xbf\wedge \eta]_i) = ([\obf\wedge \eta]_i) \cdot \prod_{j:\; i\in \Phi(j) \text{ and } x_j=1} \frac{(\ebf_{\Lambda(j)})}{(\obf)}, 
\]
and thus 
\[
    \prod_{i=1}^m ([\xbf\wedge \eta]_i) = \prod_{i=1}^m ([\obf\wedge \eta]_i) \cdot  \prod_{j:\; x_j=1} \left(\frac{(\ebf_{\Lambda(j)})}{(\obf)}\right)^{\#\Phi(j)}.
\]
It follows that 
\begin{align*}
      (\xbf)' &=\sum_{\eta: Y\to \{0,1\}} \lambda^{-\|\xbf\wedge \eta\|_H} \prod_{i=1}^m ([\xbf\wedge \eta]_i) \\
      &=  \prod_{j:\; x_j=1} \lambda^{1-\#\Phi(j)}\left(\frac{(\ebf_{\Lambda(j)})}{(\obf)}\right)^{\#\Phi(j)}\cdot \sum_{\eta: Y\to \{0,1\}} \lambda^{-\|\obf\wedge \eta\|_H} \prod_{i=1}^m ([\obf\wedge \eta]_i) \\
      &= (\obf)'\cdot \prod_{j:\; x_j=1} \lambda^{1-\#\Phi(j)} \left(\frac{(\ebf_{\Lambda(j)})}{(\obf)}\right)^{\#\Phi(j)},
\end{align*}
which completes the proof.
\end{proof}

Lemma~\ref{lem: graph on 0-chart} implies that, for each fixed non-zero parameter $\lambda$, the corresponding renormalization map $F_\lambda$ sends $\MM\setminus\IS(F_\lambda)$---i.e., the \emph{regular part} of $\MM$---to a subvariety of dimension $\#\Lambda(\{1,\dots,k\})$. In the case when the defining gluing data $(H, \Phi)$ is non-degenerate, we can say even more: up to passing to an iterate, the restriction $F_\lambda|\MM$ is a retraction, that is,  $F_\lambda^2 = F_\lambda$ on $\MM\setminus \IS(F_\lambda^2)$. 

\begin{definition}[Pre-fixed gluing data]
A gluing data $(H,\Phi)$ with parameters $k,m$ (and associated graph recursion $\RR = \RR_{(H,\Phi)}$) is called \emph{pre-fixed} if the induced dynamical system $\Lambda=\Lambda_{(H,\Phi)}$ is pre-fixed, that is, $\Lambda^2=\Lambda$ on $\{1,\dots, k\}$.
\end{definition}

Clearly, for each gluing data $(H,\Phi)$ its $N$-th iterate $(H_N,\Phi_N)$ is pre-fixed for some sufficiently large $N\in \N$. For this iterate $N$, each periodic label $j\in \{1,\dots, k\}$ is fixed under $\Lambda^N:=\Lambda^N_{(H,\Phi)}$, while for each non-periodic label $j$ the respective image $\Lambda^N(j)$ is fixed under $\Lambda^N$. Furthermore, if $(H,\Phi)$ is also non-degenerate, then each periodic label $j\in \{1,\dots, k\}$ satisfies $\#\Phi_N(j)=1$. We also note that, if a gluing data is pre-fixed, then each of its iterates is pre-fixed as well. 

\begin{prop}[Pre-fixed dynamics]\label{prop: pre fixed dynamics}
    Suppose the gluing data $(H, \Phi)$ is pre-fixed, and let $1 \leq k_0\leq k$ be the number of periodic labels, that is, $k_0=\#\Lambda(\{1,\dots, k\})$. Then, for each $\lambda\in \C^*$, $F_\lambda$ maps (the regular part of) the invariant variety $\MM$ to an irreducible $k_0$-dimensional subvariety $\MM_0= \MM_0(\lambda) \subset \MM$, which depends continuously on $\lambda\in \C^*$ in the Hausdorff metric. The regular part $\MM_0\setminus\IS(F_\lambda)$ of $\MM_0$ is invariant under the renormalization map $F_\lambda$.

    Moreover, if $(H, \Phi)$ is also non-degenerate, then $\MM_0$ is smooth and $F_\lambda$ restricts to the identity on the regular part of $\MM_0$:
      \begin{center}
\begin{tikzpicture}[->,>={Stealth[round]},shorten >=1pt,auto,semithick]
  \node (M) {$\MM$};
  \node (M0) [right=2cm of M] {$\MM_0$};
  \draw[->, dashed] (M) to node {$F_\lambda$} (M0);
  \draw[->, dashed] (M0) edge[loop right] node[right] {\small $F_\lambda=\mathrm{id}_{\MM_0}$} ();
\end{tikzpicture}
\end{center}
Furthermore, the subvarieties $\MM_0(\lambda)$ form a holomorphic family in $\lambda$: \[\mathcal{Y}:=\{(\lambda, \xi): \lambda\in \C^*, \xi\in \MM_0(\lambda)\}\] is a complex submanifold of $\C^*\times \P^{2^k-1}$.
\end{prop}

We remark that, while the invariant variety $\MM\subset \P^{2^{k}-1}$ is universal, in the sense that it depends only on the number $k$ of marked vertices, its subvariety $\MM_0=\MM_0(\lambda)$ from the proposition above depends more subtly both on the gluing data $(H,\Phi)$ and on the parameter $\lambda$.

\begin{proof}
    Up to renaming, we may assume that the labels $1,\dots, k_0$ are periodic. Since $(H,\Phi)$ is pre-fixed, we then have $\Lambda(j)=j$ for each label $j\leq k_0$, and $\Lambda(j)\in \{1,\dots, k_0\}$ for all $j>k_0$.

    Consider the monomial map $\Psi_{0, \lambda}: (\C^*)^{k_0}\to \P^{2^k-1}$ defined by  
    \[(e_1,\dots, e_{k_0}) \mapsto [1: \cdots:\prod_{j:\; x_j=1} \lambda^{1-\#\Phi(j)}  \,(e_{\Lambda(j)})^{\#\Phi(j)}: \cdots ],\]
    meaning that the $(\!(x_1, \ldots, x_k)\!)$-coordinate of the image is given by the monomial 
    \begin{equation}\label{eq: x-coordinate of Psi_0}
    \prod_{j:\; x_j=1} \lambda^{1-\#\Phi(j)} \, (e_{\Lambda(j)})^{\#\Phi(j)},
    \end{equation}
    except for the $(\obf)$-coordinate, which equals $1$. In particular, for each periodic label $j=1,\dots, k_0$, the $(\ebf_j)$-coordinate of the image equals \[\lambda^{1-\#\Phi(j)} \, (e_{j})^{\#\Phi(j)}.\]

    Let $Y_{\lambda}$ be the image of $\Psi_{0, \lambda}$ and $\MM_0=\MM_0(\lambda)$ be its Zariski closure. Since $\Psi_{0,\lambda}$ is a regular map with finite fibers for each $\lambda\in \C^*$, irreducibility and dimension are preserved, that is, $\MM_0$ is an irreducible $k_0$-dimensional variety in $\P^{2^k-1}$. By \eqref{eq: variety in basis coordinates}, the set $Y_{\lambda}$, and thus its Zariski closure $\MM_0$, is contained in the invariant variety $\MM$. Moreover, by Lemma~\ref{lem: graph on 0-chart}, $F_\lambda$ maps the regular part of $\MM\cap \{(\obf) \neq 0\}$ into $Y_{\lambda}\subset \MM_0$. Since $\left(\MM\cap \{(\obf) \neq 0\}\right) \setminus \IS(F_\lambda)$ is open and dense in $\MM$, it follows that $\MM_0$ contains the image $F_\lambda(\MM \setminus \IS(F_\lambda))$. In particular, the regular part $\MM_0 \setminus \IS(F_\lambda)$ is invariant under $F_\lambda$. Finally, since $\Psi_{0,\lambda}$ is given by the monomials \eqref{eq: x-coordinate of Psi_0}, the subvariety $\MM_0=\MM_0(\lambda)$ is the closure of the image $Y_\lambda=\Psi_{0,\lambda}((\C^*)^{k_0})$ also in the Euclidean topology induced from $\C^{2^k}$, and thus varies continuously with $\lambda$. 
    
    Now suppose that $(H, \Phi)$ is non-degenerate. Then for each periodic label $j=1,\dots, k_0$, we have $\#\Phi(j)=1$, and thus the $(\ebf_j)$-coordinate of $\Psi_{0,\lambda}\big((e_1,\dots, e_{k_0})\big)$ equals $e_j$. (In particular, $\Psi_{0,\lambda}$ is an embedding in this case.) By Lemma~\ref{lem: graph on 0-chart}, $F_\lambda$ fixes each point in the open dense subset $Y_{\lambda}\setminus \IS(F_\lambda)$ of  $\MM_0$. We conclude that $F_\lambda$ fixes all points of $\MM_0\setminus \IS(F_\lambda)$, i.e., it restricts to the identity on the regular part of $\MM_0$. It remains to show that $\MM_0$ is smooth. 
    
    The definition of $\Psi_{0,\lambda}$ implies that it is sufficient to prove smoothness of $\MM_0=\MM_0(\lambda)$ for $\lambda = 1$. In this case, $\MM_0$ is the projective toric variety associated with the finite set $\mathscr{A}_0\subset \Z^{k_0}$, consisting of the exponent vectors of the coordinate monomials of $\Psi_{0,1}$, and realized in $\P^{2^k-1}$ via the monomial embedding $\Psi_{0,1}$. More explicitly, for each coordinate $(\xbf)$, $\xbf=(x_1,\dots, x_k)\in \{0,1\}^k$, the $i$-th entry of the corresponding exponent vector $v_{(\xbf)}\in \mathscr{A}_0\subset \Z^{k_0}$ is given by 
    \begin{equation}\label{eq: exponents}
    \sum_{j:\; \Lambda(j)=i \text{ and } x_j=1} \#\Phi(j),
    \end{equation}
    where the sum is understood to be zero if empty. It is straightforward to verify that the convex hull $\mathscr{P}_0:=\operatorname{conv}(\mathscr{A}_0)$ is the $k_0$-dimensional hyperbox in $\R^{k_0}$ with opposite vertices \[v_{(\obf)}=(0,0, \dots, 0) \text{\; and  \;} 
    v_{(\!(1,\dots,1)\!)}=\left(\sum_{j:\; \Lambda(j)=1} \#\Phi(j), \sum_{j:\; \Lambda(j)=2} \#\Phi(j),  \dots, \sum_{j:\; \Lambda(j)=k_0} \#\Phi(j)\right).\]
    Note that, since $\mathscr{P}_0$ may contain $\Z^{k_0}$-lattice points that are not in $\mathscr{A}_0$,  the variety $\MM_0$ is not necessarily the projective toric variety associated with the polytope $\mathscr{P}_0=\operatorname{conv}(\mathscr{A}_0)$ itself. Nevertheless, we may still utilize the idea of the smoothness criterion for projective toric varieties from \cite[Theorem~2.4.3]{Cox}.

    Namely, by \cite[Theorem~2.1.9]{Cox}, the vertices of the polytope $\mathscr{P}_0$ determine an affine covering 
    \[\MM_0=\bigcup_{\xbf\in \Xbf_0} \left(\MM_0\cap \{(\xbf) \neq 0\}\right) ,\]
    of the variety $\MM_0$, where 
    \[\Xbf_0:=\{\xbf\in\{0,1\}^k: \,v_{(\xbf)}\text{ is a vertex of $\mathscr{P}_0$}\}.\]
    Now, for each vertex $v_{(\xbf)}$ of $\mathscr{P}_0$ with $\xbf=(x_1,\dots,x_k)\in \Xbf_0\subset \{0,1\}^k$ and each standard basis vector $\xi_i$ of $\R^{k_0}$ with $i\in \{1,\dots, k_0\}$, one of the two $\Z^{k_0}$-lattice points 
    $v_{(\xbf)}\pm\xi_i$ must be in $\mathscr{A}_0$. Indeed, by~\eqref{eq: exponents}, we have
    \[v_{(\xbf+\ebf_i)}=v_{(\xbf)}+\xi_i \quad \text{if $x_i=0$}\]
    and 
    \[v_{(\xbf-\ebf_i)}=v_{(\xbf)}-\xi_i \quad \text{if $x_i=1$}.\]
    Combinatorially, this means that the corresponding point $w_{(\xbf),i}=v_{(\xbf)}\pm\xi_i$ is the first lattice point of $\mathscr{A}_0$ different from $v_{(\xbf)}$ encountered as one traverses the edge of $\mathscr{P}_0$ parallel to $\xi_i$ starting at $v_{(\xbf)}$. Hence, the vectors $w_{(\xbf),i}-v_{(\xbf)}$, $i\in\{1,\dots, k_0\}$, form a basis of the $\Z^{k_0}$-lattice. It follows that the Jacobian 
     of the affine parametrization of $\MM_0\cap \{(\xbf) \neq 0\}$ induced by $\Psi_{0,1}$ has full rank at every point of this chart, and thus $\MM_0\cap \{(\xbf) \neq 0\}$ is smooth; compare the smoothness criterion from \cite[Theorem~2.4.3]{Cox}. Since these charts cover $\MM_0$, the variety $\MM_0$ is smooth. 
     
     Finally, the set $\mathcal{Y}=\{(\lambda, \xi): \lambda\in \C^*, \xi\in \MM_0(\lambda)\}\subset \C^*\times \P^{2^k-1}$ is the image of $\C^*\times \MM_0(1)$ under the biholomorphism $(\lambda,\xi)\mapsto(\lambda, D_\lambda(\xi))$, where $D_\lambda$ is the diagonal automorphism of $\P^{2^k-1}$ for which the $(\xbf)$-coordinate is multiplied by $\prod_{j:\; x_j=1} \lambda^{1-\#\Phi(j)}$. Hence $\mathcal{Y}$ is a complex submanifold of $\C^*\times \P^{2^k-1}$. This finishes the proof of the proposition. 
\end{proof}

\begin{example}\label{ex: M0}\mbox{}
    \begin{enumerate}[label=(\roman*)]
    \item\label{item: ex_M0_i} The Sierpi\'{n}ski gluing data $(H,\Phi)$ with parameters $k = m = 3$ from Example~\ref{ex: sierpinksi} is non-degenerate and pre-fixed. In fact, the induced dynamical system $\Lambda_{(H,\Phi)}$ fixes each label. It follows that the subvariety $\MM_0=\MM_0(\lambda)$ from Proposition~\ref{prop: pre fixed dynamics} is simply the corresponding invariant variety $\MM\subset \P^{2^k-1}$ for $\GG_k$ itself. In particular, $\MM_0$ is smooth and $F_\lambda$ acts as the identity on the regular part of $\MM$ for each $\lambda\in \C^*$.
    
    \item\label{item: ex_M0_ii} The diamond hierarchical gluing data $(H,\Phi)$ with parameters $k=2$ and $m=4$ from Example~\ref{ex: diamon_lattice} is degenerate and pre-fixed. 
    The induced dynamical system $\Lambda_{(H,\Phi)}$ sends both labels to $1$, so that $k_0 = 1$ and the subvariety $\MM_0=\MM_0(\lambda)$ from Proposition~\ref{prop: pre fixed dynamics} is a curve in the corresponding invariant variety $\MM\subset \P^{2^k-1}$ for $\GG_k$. In fact, on the regular part of the chart $\MM\cap \{(\obf)\neq 0\}$ the renormalization map $F_\lambda$ is given by:
    \[[1:s:t:st]\mapsto [1:\frac{1}{\lambda}\,s^2:\frac{1}{\lambda}\,s^2:\frac{1}{\lambda^2}\,s^4],\]
    and thus the Zariski closure $\MM_0$ of the image is a smooth curve.

    \item\label{item: ex_M0_iii} The ($z^2+i$)-gluing data $(H,\Phi)$ with parameters $k=3$ and $m=2$ from Example~\ref{ex: dendrite} is non-degenerate but not pre-fixed. The induced dynamical system $\Lambda_{(H,\Phi)}$ interchanges the labels $2$ and $3$, while the label $1$ is pre-periodic. Thus the regular part of the $3$-dimensional variety $\MM \subset \P^7$ is mapped onto the smooth $2$-dimensional subvariety $\MM_0$, and regular points of $\MM_0$ are fixed under the action of the second iterate $F_\lambda^2$.
    \end{enumerate}
\end{example}

\begin{remark} For degenerate gluing data, the subvariety $\MM_0$ from Proposition~\ref{prop: pre fixed dynamics} may or may not be smooth. We consider two examples, both for $k=2$ marked vertices. We will only provide the map $\Lambda=\Lambda_{(H,\Phi)}$ on labels and their local degrees, leaving it to the reader to supply the corresponding gluing data $(H,\Phi)$. In both cases, we take $\Lambda(1) = \Lambda(2) = 1$, so that $k_0 = 1$ and $\MM_0$ is a curve in $\MM\subset \P^3$, and we let $\# \Phi(1) = 2$, so that the gluing data $(H,\Phi)$ is degenerate.  

When $\# \Phi(2) = 1$, the renormalization map $F_\lambda$ defined on the regular part of the chart $\MM\cap \{(\obf) \neq 0\}$ is given by 
\[
[1:s:t:st] \mapsto [1:\frac{1}{\lambda}s^2 : s: \frac{1}{\lambda}s^3], 
\]
and hence the Zariski closure $\MM_0$ of the image is a smooth curve. (See also Example~\ref{ex: M0}\ref{item: ex_M0_ii} for the case when $\#\Phi(2)=2$, in which case $\MM_0$ is also a smooth curve.) However, when $\# \Phi(2) = 3$, the respective map $F_\lambda$ on the regular
part of the same chart of $\MM$ is given by
\[
[1:s:t:st] \mapsto [1 : \frac{1}{\lambda}s^2 : \frac{1}{\lambda^2}s^3 : \frac{1}{\lambda^3}s^5].
\]
In this case, the Zariski closure $\MM_0$ of the image is a curve with cusp singularities at $[1:0:0:0]$ and $[0:0:0:1]$.
\end{remark}

\subsection{Dynamics near the invariant variety}\label{subsec: dynamics near variety}

Our goal in this subsection is to discuss the behavior of the renormalization map $F_\lambda$ near the invariant variety $\MM\supset \MM_0=\MM_0(\lambda)$. First, given a self-map on a manifold, we define what it means for a submanifold to be \emph{transversally superattracting}. We introduce this concept in the smooth setting and afterwards discuss equivalent characterizations in terms of the dynamical behavior of nearby orbits, both in the general case and in the case $\MM_0 \subset \MM\subset \P^{2^k-1}$ of our interest.

\begin{definition}[Transversal superattraction]
    Let $\mathcal{X}$ be a smooth manifold and $\mathcal{Y} \subset \mathcal{X}$ be a smooth submanifold. Suppose $F:U(\mathcal{Y})\to \mathcal{X}$ is a smooth map defined on a neighborhood $U(\mathcal{Y})$ of $\mathcal{Y}$ that fixes the submanifold $\mathcal{Y}$ pointwise. We say that $F$ is \emph{transversally superattracting} on $\mathcal{Y}$ (or simply that $\mathcal{Y}$ is transversally superattracting, when $F$ is understood) if at every point $p \in \mathcal{Y}$, the characteristic polynomial $\det(z\cdot \mathrm{Id}-DF(p))$ of the differential $DF(p)$ has the form $(z-1)^{\dim \mathcal{Y}}z^{\dim\mathcal{X} - \dim\mathcal{Y}}$. 
\end{definition}

    Note that, since $F$ fixes $\mathcal{Y}$ pointwise, $DF(p)$ restricts to the identity on the tangent space $T_p(\mathcal{Y})$, and the condition above simply means that all eigenvalues of $DF(p)$ corresponding to (generalized) eigenvectors that are not tangent to $\mathcal{Y}$ are zero. 
    We record the following standard lemma providing a characterization of transversal superattraction in terms of superexponentially fast convergence of orbits starting close to a point of $\mathcal{Y}$. 

\begin{lemma}\label{lem: trans_superattr}
     Let $\mathcal{X}$ be a smooth manifold, $\mathcal{Y}\subset \mathcal{X}$ be a smooth submanifold, and let $F:U(\mathcal{Y})\to \mathcal{X}$ be a smooth map defined on a neighborhood $U(\mathcal{Y})$ of $\mathcal{Y}$ that fixes the submanifold $\mathcal{Y}$ pointwise. 
     
     Suppose $F$ is transversally superattracting on $\mathcal{Y}$, and consider an arbitrary point $p_0 \in \mathcal{Y}$ and a coordinate neighborhood $U(p_0)$. Then there exists a sub-neighborhood $U'(p_0)\subset U(p_0)$ such that for every starting point $p\in U'(p_0)$, the orbit $(F^n(p))_{n\geq 0}$ is well-defined, remains in $U(p_0)$, and converges to a point $p_\infty=p_\infty(p)\in\mathcal{Y}$ superexponentially fast with respect to the Euclidean metric on $U(p_0)$. 

     Conversely, the map $F$ is transversally superattracting on $\mathcal{Y}$ if the following holds for each point $p_0\in \mathcal{Y}$: there exists some coordinate neighborhood $U(p_0)$ of $p_0$ and a sub-neighborhood $U'(p_0)\subset U(p_0)$ such that the orbits of points $p\in U'(p_0)$ are well-defined, 
     remain in $U(p_0)$, and converge to $\mathcal{Y}$ superexponentially fast with respect to the Euclidean metric on $U(p_0)$. 
\end{lemma}

In the lemma above, ``superexponentially fast'' means faster than any exponential rate. More formally, for all $\rho\in (0,1)$, there exists a constant $C_\rho>0$ such that $\dist(F^n(p),p_\infty)<C_\rho \rho^n$ for all $n\geq 0$ in the forward direction and $\dist(F^n(p),\mathcal{Y})<C_\rho \rho^n$ for all $n\geq 0$ in the converse.

\medskip

Let us now fix a non-degenerate and pre-fixed recursion operator $\RR$ on $\GG_k$ with $k\geq1$, and consider the associated renormalization map $F_\lambda: \P^{2^k-1}\dashrightarrow \P^{2^k-1}$ for some fixed $\lambda\in \C^*$. We recall from Proposition~\ref{prop: pre fixed dynamics} that in this setting, away from the indeterminacy set $\IS(F_\lambda)$, the map $F_\lambda$ is a holomorphic retraction from $\MM$ to the smooth submanifold $\MM_0=\MM_0(\lambda)$. We then say that $F_\lambda$ is \emph{transversally superattracting on $\MM_0$} if this is the case away from $\MM_0\cap \IS(F_\lambda)$. If 
the recursion operator $\RR$ is not pre-fixed, we say that $F_\lambda$ is transversally superattracting on $\MM_0$ if some iterate $F_\lambda^N$ has this property, where $\RR^N$ is pre-fixed and $\MM_0$ is the corresponding subvariety of $\MM$ from Proposition~\ref{prop: pre fixed dynamics}.

Our goal is to prove the following result. 

\begin{theorem}\label{thm: superattraction}
    Suppose that a recursion operator $\RR$ on $\GG_k$ with $k\geq 1$ is non-degenerate, pre-fixed, and expanding, and fix $\lambda\in \C^*$.  Then the renormalization map $F_\lambda: \P^{2^k-1}\dashrightarrow \P^{2^k-1}$ is transversally superattracting on $\MM_0=\MM_0(\lambda)$.
\end{theorem}

Note that for $k=1$, we have $\MM=\MM_0(\lambda)$ and the statement is vacuous, so we will assume that $k\geq 2$ in the following. To prove the theorem, we will establish superexponential convergence of orbits in some coordinate neighborhood $U(p_0)$ of each point $p_0\in \MM_0$. We note that superexponential convergence does not depend on the choice of metric in $U(p_0)$, as long as the metrics are bi-Lipschitz equivalent near $p_0$; this is the case, for instance, for the Euclidean metrics of any two coordinate charts containing $p_0$. Alternatively, we may also use the Fubini--Study metric on $\P^{2^k-1}$. However, in our setting, it is natural to consider instead superexponential convergence to zero of the expression
\[
\max_{\substack{\xbf, \ybf, \zbf, \wbf\in \{0,1\}^k:\\ \xbf + \ybf = \zbf + \wbf}}|({\xbf})({\ybf}) - ({\zbf})({\wbf})|,
\]
in suitably chosen affine coordinates on a neighborhood $U(p_0)$ of $p_0$. We note that, since this expression arises from the defining quadratic equations of the smooth manifold $\MM$ (see Definition~\ref{def: inv variety}), in $U(p_0)$ this maximum is comparable to the Euclidean distance to $\MM$. Moreover, since $F_\lambda$ is a holomorphic retraction from $M$ to $\MM_0$, along orbits staying in the neighborhood $U(p_0)$, superexponentially fast convergence to $\MM$, measured by the defining equations, is equivalent to superexponentially fast convergence to $\MM_0$ in terms of the Euclidean metric.

By the discussion above, Theorem~\ref{thm: superattraction} is a straightforward consequence of Proposition~\ref{prop: euclidean attracting} below, whose proof is perhaps the most technical of those presented in this paper. At the heart of it however lies a clear idea that can perhaps be better understood in the language of graphs and probabilities. For this reason, we first present a more
intuitive proof for superexponentially fast convergence to $\MM_0$ in the case of $\lambda\in \R_+$ and coordinates that correspond to a graph $G\in \GG_k$. In this setting, it is natural to work in the chart $\sum_{\xbf\in \{0,1\}^k} (\xbf)=1$, and consider 
the following expression 
\[
\theta_\lambda(G):=\max_{\tau, \sigma} |\P_{G,\lambda}[\tau:\sigma] - \P_{G,\lambda}[\tau]|,
\]
where the maximum is taken over all assignments $\tau$ and $\sigma$ on disjoint subsets of $P(G)$. Note that $\theta_\lambda(G)$ is well-defined as long as the marked vertices in $G$ are all non-adjacent. Since we work with expanding operators $\RR$, this is the case for all sufficiently large $n$ in any recursive sequence $G_n=\RR^n(G_0)$ with $G_0\in \GG_k$. Moreover, as $\RR$ is also non-degenerate, the vertex degrees in the graphs $G_n$ are uniformly bounded (see Lemma~\ref{lem: degeneration and expansion}). By Lemma~\ref{lem: probability lower bound}, this implies that, for all sufficiently large $n$, the probabilities $\P_{G_n,\lambda}[\tau]$ are uniformly bounded away from $0$ for all assignments $\tau$ on a subset of $P(G_n)$. Since, by Lemmas~\ref{lem: direct vs conditional correlation} and~\ref{lem: inv variety}, the quantity $\theta_\lambda(G)$ vanishes precisely when $\phi_\lambda(G)\in \MM$, it follows that in order to prove superexponentially fast convergence of $\phi_\lambda(G_n)$ to $\MM_0$ it is sufficient to show that $\theta_\lambda(G_n)$ converges to $0$ as $n\to \infty$ superexponentially fast. The latter immediately follows from the statement below.

\begin{prop}\label{prop: attracting for graphs}
     Let $(H,\Phi)$ be a gluing data with parameters $k\geq 2, m \geq 2$ that satisfies the following condition: 
    \begin{enumerate}[label=($\star$),font=\normalfont]
        \item\label{cond: superattraction} for each pair $j,j'\in \{1,\dots, k\}$ of distinct labels the corresponding edges $\Phi(j)$ and $\Phi(j')$ of   the gluing scheme $H$ form disjoint subsets of $V(H)$.
    \end{enumerate}
    Let $\RR=\RR_{(H,\Phi)}$, and fix $\lambda\in \R_+$ and $\Delta\in \N$. Then there are constants $\epsilon_0 >0$ and $C>0$ such that for any graph $G\in \GG_k$ whose marked vertices are pairwise non-adjacent and have degrees bounded by $\Delta$ the following implication holds: if $\theta_\lambda(G)\leq \epsilon$ for some $0<\epsilon<\epsilon_0$, then $\theta_\lambda(\RR(G))\leq C \epsilon^2$.
\end{prop}

Note that condition~\ref{cond: superattraction} is satisfied for an iterate of every expanding gluing data; see Remark~\ref{rem: expansion}. It also implies that if the marked vertices of $G$ are pairwise non-adjacent, then so are the marked vertices of $\RR(G)$; in particular, the quantity $\theta_\lambda(G)$ is well-defined.

To prove the proposition, we need the following auxiliary statement.

\begin{lemma}\label{lemma: effect of gluing}
    Fix $\lambda\in\R_+$ and $\Delta, k, m\in \N$. Then there are constants $\epsilon_0 > 0$ and $C_0>0$ such that the following holds for any recursion operator $\RR=\RR_{(H,\Phi)}$ with parameters $k$ and $m'\leq m$ and a graph $G\in \GG_k$ whose marked vertices are non-adjacent and have degrees bounded by $\Delta$:

    If $\theta_\lambda(G)\leq\epsilon$ for some $0<\epsilon<\epsilon_0$, then 
    \[
    \theta^S_\lambda(\RR(G))=\max_{\tau, \sigma} |\P_{\RR(G),\lambda}[\tau:\sigma] - \P_{\RR(G),\lambda}[\tau]| \leq C_0 \epsilon, 
    \]
    where the maximum is taken over all assignments $\tau$ and $\sigma$ on disjoint subsets of the set $S$ of vertices of $\RR(G)$ induced by $E(H)$. 
\end{lemma}
\begin{proof}[Proof outline] The statement of the lemma may be proven directly using explicit computations of conditional probabilities. However, we provide instead a more elegant argument which relies on a natural extension of the material discussed in Sections~\ref{subsec: induced dynamics}--\ref{subsec: dynamics on variety}.   

For each $\lambda\in \C^*$, we may partition the independence polynomial $Z_{\RR(G)}(\lambda)$ according to the vertex assignments $\eta$ on the set $S$. Arguing as in Section~\ref{subsec: induced dynamics}, the recursion operator $\RR$ then induces a rational map $F^S_\lambda: \P^{2^k-1}\dashrightarrow\P^{2^s-1}$, with $s=\#S$, where we use the coordinates on $\P^{2^s-1}$ corresponding to the $\eta$-conditioned independence polynomials $Z_{\RR(G)}^\eta(\lambda)$. Similarly to Section~\ref{subsec: invariant variety}, we may introduce a subvariety $\MM^S$ of $\P^{2^s-1}$ defined by the absence of correlations between the vertices in $S$ for $\RR(G)$. Proposition~\ref{prop: invariance} then naturally generalizes to the statement $F^S_\lambda(\MM)\subset \MM^S$. 

Now let $\MM(\R_{\ge 0})$ (resp.\  $\MM^S(\R_{\ge 0})$) denote the compact subset of $\MM$ (resp.\ of $\MM^S$), where all coordinates are non-negative and real. Note that the map $F^S_\lambda$ is well-defined and holomorphic in a neighborhood $U$ of $\MM(\R_{\ge 0})$ for each $\lambda\in \R_+$, and moreover, $F^S_\lambda\big(\MM(\R_{\ge 0})\big)\subset \MM^S(\R_{\ge 0})$. Since $\MM(\R_{\ge 0})$ is compact, $F_\lambda^S$ is Lipschitz in a neighbourhood of $\MM(\R_{\ge 0})$. Recall that in each sufficiently small coordinate neighborhood the quantity $\theta_\lambda(G)$ is comparable to the distance to $\MM$, with constants depending only on $\lambda$ and $\Delta$. The analogous statement also holds for $\theta_\lambda(\RR(G))$ and $\MM^S$. Choosing $\varepsilon_0>0$ small enough so that $\varphi_\lambda(G) \in U$, we get the desired conclusion of the lemma.
\end{proof}

\begin{proof}[Proof of Proposition~\ref{prop: attracting for graphs}] Recall that condition~\ref{cond: superattraction} implies that the marked vertices of $\RR(G)$ are pairwise non-adjacent. Let $v$ be a  marked vertex of $\RR(G)$ with a label $j\in \{1,\dots, k\}$, and let $\tau$ be a vertex assignment on $X:=\{v\}$. We will prove that for every vertex assignment $\sigma$ on a non-empty subset $W$ of $P(\RR(G))\setminus \{v\}$ we have
\[
\left| \P_{\RR(G)}[\tau : \sigma] - \P_{\RR(G)}[\tau] \right| \le C \cdot \epsilon^2,
\]
for some constant $C> 0$, which is sufficient in order to deduce the more general statement in the proposition; compare the final step in the proof of Proposition~\ref{prop: decay of correlation}.

Let $e=\Phi(j)\subset \{1,\dots, m\}$ be the edge of the gluing scheme $H$ corresponding to the vertex $v \in V(\RR(G))$, so that $v$ arises from identifying the vertices labeled $\ell(e)$ in the copies $G(i)$ with  $i \in e$; see Definition~\ref{def:graph recursion}.  We denote by $\RR(G)_v$ the induced subgraph of $\RR(G)$ containing only the vertices arising from the copies $G(i)$ with $i\in e$, and by $Y$ the set of vertices of $\RR(G)_v$ corresponding to the marked vertices of these copies, not including the vertex $v$. Note that, by our assumptions, we then have $X\cap Y = W \cap Y = \emptyset$. Moreover, no edge in $\RR(G)$ connects a
vertex in $Y$ with a vertex in $X$ or $W$. 

Arguing as in Lemma~\ref{lem: conditioning wrt Y}, we can derive that 
\[
 \P_{\RR(G)}[\tau : \eta\wedge \sigma]=\P_{\RR(G)_v}[\tau : \eta] 
\]
for every vertex assignment $\eta$ on $Y$. Consequently, we obtain
$$
\begin{aligned}
\P_{\RR(G)}[\tau : \sigma] = & \sum_{\eta} \P_{\RR(G)}[\tau \wedge \eta: \sigma] =  \sum_{\eta} \P_{\RR(G)}[\tau : \eta \wedge \sigma] \cdot \P_{\RR(G)}[\eta : \sigma]\\
= & \sum_{\eta} \P_{\RR(G)_v}[\tau : \eta] \cdot \P_{\RR(G)}[\eta : \sigma],
\end{aligned}
$$
and similarly
$$
\begin{aligned}
\P_{\RR(G)}[\tau] = & \sum_{\eta} \P_{\RR(G)}[\tau \wedge \eta] =  \sum_{\eta} \P_{\RR(G)}[\tau : \eta] \cdot \P_{\RR(G)}[\eta]\\
= & \sum_{\eta} \P_{\RR(G)_v}[\tau : \eta] \cdot \P_{\RR(G)}[\eta],
\end{aligned}
$$
where all the sums here and below are taken over all possible assignments $\eta$ on $Y$. It follows that 
\begin{equation*}
\begin{aligned}
\P_{\RR(G)}[\tau : \sigma] - \P_{\RR(G)}[\tau] = & \sum_{\eta} \P_{\RR(G)_v}[\tau : \eta] \cdot \left(\P_{\RR(G)}[\eta : \sigma] - \P_{\RR(G)}[\eta]\right)\\
= & \sum_{\eta} \P_{\RR(G)_v}[\tau] \cdot \left(\P_{\RR(G)}[\eta : \sigma] - \P_{\RR(G)}[\eta]\right) \\
+ &
\sum_{\eta} \left( \P_{\RR(G)_v}[\tau : \eta] - \P_{\RR(G)_v}[\tau] \right) \cdot \left(\P_{\RR(G)}[\eta : \sigma] - \P_{\RR(G)}[\eta]\right).
\end{aligned}
\end{equation*}
Since 
\[\sum_{\eta}\P_{\RR(G)}[\eta : \sigma]=\sum_{\eta}\P_{\RR(G)}[\eta]=1,\]
the first sum on the right-hand side of the equation above vanishes, so that 
\begin{equation}\label{eq: punch line}
|\P_{\RR(G)}[\tau : \sigma] - \P_{\RR(G)}[\tau]| \leq 
\sum_{\eta} | \P_{\RR(G)_v}[\tau : \eta] - \P_{\RR(G)_v}[\tau] | \cdot |\P_{\RR(G)}[\eta : \sigma] - \P_{\RR(G)}[\eta]|.
\end{equation}
On the other hand, by Lemma~\ref{lemma: effect of gluing}, there exists an $\epsilon_0>0$ such that each factor in the sum~\eqref{eq: punch line} is of order $O(\epsilon)$ as long as $\theta_\lambda(G)<\epsilon<\epsilon_0$. (Here, for the first term, we realize $\RR(G)_v$ using a gluing scheme with parameters $k, m'=\#e\leq m$.) It follows that the quantity $\theta_\lambda(\RR(G))$ is of order $O(\epsilon^2)$, with the constant depending only on $\lambda, m, k, \Delta$. This finishes the proof.
\end{proof}

We now establish the main technical statement of this subsection. 

\begin{prop}\label{prop: euclidean attracting}
    Let $(H,\Phi)$ be a gluing data with parameters $k\geq 2,m \geq 2$ that satisfies condition~{\normalfont \ref{cond: superattraction}}. Suppose further that $\lambda\in \C^*$ is fixed and $\widehat{F}_\lambda: \C^{2^k}\to \C^{2^k}$ is the homogeneous polynomial self-map that induces the renormalization map $F_\lambda$ associated with the graph recursion $\RR_{(H,\Phi)}$. 
    
    Let $\widehat{\xi}=\big( (\xbf): \xbf\in \{0,1\}^k\big)$ be a point in $\C^{2^k}$ and   $\widehat{\xi}':=\widehat{F}_\lambda(\widehat{\xi})=\big( (\xbf)': \xbf\in \{0,1\}^k\big)$ be its image under $\widehat{F}_\lambda$. Then there exists a constant $C(\widehat\xi, \lambda)$ such that the following is true for all $\epsilon\in \R_+$: 

    If the coordinates of  $\widehat{\xi}$ satisfy
    \begin{equation}\label{eq: no correlation inequality}
    \big|(\xbf) \cdot (\ybf) - (\zbf) \cdot  (\wbf)\big| \leq \epsilon \quad \text{for all \; $\xbf,\, \ybf, \zbf, \wbf\in \{0,1\}^{k}$ with $\xbf+\ybf=\zbf+\wbf$,}
    \end{equation}
    then the coordinates of $\widehat{\xi}'$ satisfy
    \begin{equation}\label{eq: no correlation image inequality}
    \big|(\xbf)' \cdot (\ybf)' - (\zbf)' \cdot (\wbf)'\big| \leq C(\widehat{\xi},\lambda)\cdot \epsilon^2 \quad \text{for all \; $\xbf,\, \ybf, \zbf, \wbf\in \{0,1\}^{k}$ with $\xbf+\ybf=\zbf+\wbf$}.
    \end{equation}
    Moreover, the constant $C(\widehat{\xi},\lambda)$ depends polynomially on $\abs{\frac{1}{\lambda}}$ and on the absolute values of the coordinates of~$\widehat{\xi}$.
\end{prop}

Before we prove the proposition above, let us first deduce Theorem~\ref{thm: superattraction}.

\begin{proof}[Proof of Theorem~\ref{thm: superattraction}]
Since the recursion operator $\RR$ is assumed to be non-degenerate and pre-fixed, the map $F_\lambda$ is a holomorphic retraction from $\MM$ to $\MM_0$ for each $\lambda\in \C^*$; see Proposition~\ref{prop: pre fixed dynamics}. By definition, it is then sufficient to show that some iterate of $F_\lambda$ is transversally superattracting on $\MM_0$. Hence, by replacing $F_\lambda$ with a sufficiently large iterate, we may assume that condition~\ref{cond: superattraction} is satisfied. 
Given an arbitrary point $\xi \in \MM_0$ outside of the indeterminacy set, we can now choose a lift $\widehat{F}_\lambda$ so that $\widehat{F}_\lambda(\widehat{\xi}) = \widehat{\xi}$. The estimates obtained in Proposition~\ref{prop: euclidean attracting} then imply superexponentially fast convergence of nearby orbits to the invariant variety, and thus to $\MM_0\subset \MM$, because $F_\lambda$ maps $\MM$ into $\MM_0$ and is Lipschitz near $\xi$. The theorem now follows from Lemma~\ref{lem: trans_superattr}.
\end{proof}

\begin{proof}[Proof of Proposition~\ref{prop: euclidean attracting}]
    It is sufficient to verify that
    \begin{equation}\label{eq: no correlation image inequality reduced}
     \big|(\xbf)' \cdot (\widetilde{\xbf})' - (\max(\xbf, \widetilde{\xbf}))' \cdot (\min(\xbf, \widetilde{\xbf}))'\big| \leq C_1(\widehat{\xi},\lambda)\cdot \epsilon^2 \quad \text{for all \; $\xbf,\, \widetilde{\xbf}\in\{0,1\}^k$},
    \end{equation}
    for some constant $C_1(\widehat{\xi},\lambda)$ depending polynomially on $\abs{\frac{1}{\lambda}}$ and on the absolute values of the coordinates of $\widehat{\xi}$. If the binary tuples $\xbf$ and $\widetilde{\xbf}$ differ in at most one entry, then the left-hand side of \eqref{eq: no correlation image inequality reduced} trivially vanishes. Hence, it suffices to show the inequality when these tuples differ in at least two entries.

    To simplify the discussion, it will again be convenient to regard the binary $k$-tuples enumerating the coordinates of $\widehat{\xi}$ and $\widehat{\xi}'$ as assignments on the marked vertices of $G$ and $\RR_{(H,\Phi)}(G)$, respectively, for some graph $G\in \GG_k$. We also introduce the following notation for graphs in $\GG_k$. Given $s\in \{0,1\}$ and $j\in\{1,\dots,k\}$, let $s_j$ denote the assignment on the vertex labeled $j$ that assigns the value $s$ to this vertex. Now suppose $j, \widetilde{j}\in \{1,\dots, k\}$ are two distinct labels, and let $\tau$ be an assignment on the marked vertices with labels different from $j$ and $\widetilde{j}$. Given $s, \widetilde{s} \in \{0,1\}$, we define 
    \[
    \tau_{s\widetilde{s}}:=\tau \wedge s_j \wedge \widetilde{s}_{\widetilde{j}}.
    \]
    Since $\tau_{s\widetilde{s}}$ is an assignment on the set of all marked vertices, we may naturally view it as a binary $k$-tuple, so that $(\tau_{s\widetilde{s}})'$ denotes the corresponding coordinate of $\widehat{\xi}'$.

    \begin{claim}
        Let $j, \widetilde{j} \in \{1, \dots, k\}$ be two distinct labels, and let $\tau,\widetilde{\tau}$ be two assignments on the marked vertices with labels different from $j$ and $\widetilde{j}$. Then
        \[
            \big|(\tau_{10})' \cdot (\widetilde{\tau}_{01})' - (\tau_{11})' \cdot (\widetilde{\tau}_{00})'\big| \leq C_2(\widehat{\xi},\lambda)\cdot \epsilon^2,
        \]
       where $C_2(\widehat{\xi},\lambda)$ is a polynomial in $\abs{\frac{1}{\lambda}}$ and in the absolute values of the coordinates of $\widehat{\xi}$.
    \end{claim}

    Before we prove the claim, let us note that it implies the desired inequality~\eqref{eq: no correlation image inequality reduced}. Indeed, we can pass from the pair $(\xbf, \widetilde{\xbf})$ to $(\max(\xbf, \widetilde{\xbf}), \min(\xbf, \widetilde{\xbf}))$ by ``flipping'' the differing entries one at a time.  Using the claim at each step and the triangle inequality, we conclude that \eqref{eq: no correlation image inequality reduced} holds with $C_1(\widehat{\xi}, \lambda)=k\cdot C_2(\widehat{\xi}, \lambda)$.

    \smallskip

    It remains to prove the claim. Following the notation in Section~\ref{subsec: induced dynamics}, we have
    \begin{align*}
      (\tau_{10})'\cdot (\widetilde{\tau}_{01})' &=\left(\sum_{\eta: Y\to \{0,1\}} \lambda^{-\|\tau_{10}\wedge \eta\|_H} \prod_{i=1}^m ([\tau_{10}\wedge \eta]_i)\right)\left(\sum_{\widetilde{\eta}: Y\to \{0,1\}} \lambda^{-\|\widetilde{\tau}_{01}\wedge \widetilde{\eta}\|_H} \prod_{i=1}^m ([\widetilde{\tau}_{01}\wedge \widetilde{\eta}]_i)\right)\\
      &=\sum_{\eta,\widetilde{\eta}: Y\to \{0,1\}} \lambda^{-\|\tau_{10}\wedge \eta\|_H-\|\widetilde{\tau}_{01}\wedge \widetilde{\eta}\|_H} \prod_{i=1}^m ([\tau_{10}\wedge \eta]_i)([\widetilde{\tau}_{01}\wedge \widetilde{\eta}]_i)\\
      &=\sum_{\eta,\widetilde{\eta}: Y\to \{0,1\}} \lambda^{-\|\tau_{10}\wedge \widetilde{\eta}\|_H-\|\widetilde{\tau}_{01}\wedge \eta\|_H} \prod_{i=1}^m ([\tau_{10}\wedge \widetilde{\eta}]_i)([\widetilde{\tau}_{01}\wedge \eta]_i),\\
    \intertext{where the last equality follows by symmetry upon interchanging the roles of $\eta$ and $\widetilde{\eta}$. Similarly, we get}
      (\tau_{11})'\cdot (\widetilde{\tau}_{00})' 
      &=\sum_{\eta,\widetilde{\eta}: Y\to \{0,1\}} \lambda^{-\|\tau_{11}\wedge \eta\|_H-\|\widetilde{\tau}_{00}\wedge \widetilde{\eta}\|_H} \prod_{i=1}^m ([\tau_{11}\wedge \eta]_i)([\widetilde{\tau}_{00}\wedge \widetilde{\eta}]_i)\\
       &=\sum_{\eta,\widetilde{\eta}: Y\to \{0,1\}} \lambda^{-\|\tau_{11}\wedge \widetilde{\eta}\|_H-\|\widetilde{\tau}_{00}\wedge \eta\|_H} \prod_{i=1}^m ([\tau_{11}\wedge \widetilde{\eta}]_i)([\widetilde{\tau}_{00}\wedge \eta]_i).
    \end{align*}
    Since for all assignments $\eta, \widetilde{\eta}$ on $Y$ we have 
    \begin{align*}
    &\|\tau_{10}\wedge \eta\|_H + \|\widetilde{\tau}_{01}\wedge \widetilde{\eta}\|_H = \|\tau_{10}\wedge \widetilde{\eta}\|_H + \|\widetilde{\tau}_{01}\wedge \eta\|_H=\\
       &\|\tau_{11}\wedge \eta\|_H + \|\widetilde{\tau}_{00}\wedge \widetilde{\eta}\|_H = \|\tau_{11}\wedge \widetilde{\eta}\|_H + \|\widetilde{\tau}_{00}\wedge \eta\|_H \geq 0,
   \end{align*} 
    we may write
    \begin{equation}\label{eq: double difference}
        \begin{aligned}
      &2\big((\tau_{10})'\cdot(\widetilde{\tau}_{01})'-(\tau_{11})'\cdot(\widetilde{\tau}_{00})'\big)=\\ 
      &\sum_{\eta,\widetilde{\eta}: Y\to \{0,1\}} \lambda^{-\|\tau_{10}\wedge \eta\|_H-\|\widetilde{\tau}_{01}\wedge \widetilde{\eta}\|_H} \left(\prod_{i=1}^m ([\tau_{10}\wedge \eta]_i)([\widetilde{\tau}_{01}\wedge \widetilde{\eta}]_i)-\prod_{i=1}^m ([\tau_{11}\wedge \eta]_i)([\widetilde{\tau}_{00}\wedge \widetilde{\eta}]_i)\right)+\\
      &\sum_{\eta,\widetilde{\eta}: Y\to \{0,1\}} \lambda^{-\|\tau_{10}\wedge \eta\|_H-\|\widetilde{\tau}_{01}\wedge \widetilde{\eta}\|_H} \left(\prod_{i=1}^m ([\tau_{10}\wedge \widetilde{\eta}]_i)([\widetilde{\tau}_{01}\wedge \eta]_i)-\prod_{i=1}^m ([\tau_{11}\wedge \widetilde{\eta}]_i)([\widetilde{\tau}_{00}\wedge \eta]_i)\right).\\
    \end{aligned}
    \end{equation}

    We now separately consider each product in the above identity. Since the gluing data $(H,\Phi)$ satisfies condition~\ref{cond: superattraction}, we may write 
    \begin{equation}\label{eq: product in copies}
    \begin{aligned}
      \prod_{i=1}^m ([\tau_{10}\wedge \eta]_i)&= \prod_{i\in \Phi(j)}([\tau_{10}\wedge \eta]_i)\prod_{i\in \Phi(\widetilde{j})}([\tau_{10}\wedge \eta]_i)\prod_{i\notin\Phi(j)\sqcup\Phi(\widetilde{j})}([\tau_{10}\wedge \eta]_i)\\
      &=\prod_{i\in \Phi(j)}(1_{\Lambda(j)}\wedge\eta(i))\prod_{i\in \Phi(\widetilde{j})}(0_{\Lambda(\widetilde{j})}\wedge\eta(i))\prod_{i\notin\Phi(j)\sqcup\Phi(\widetilde{j})}([\tau_{10}\wedge \eta]_i).\\
    \end{aligned}
    \end{equation}
    Here, $\Lambda=\Lambda_{(H,\Phi)}: \{1,\dots, k\}\to \{1,\dots, k\}$ denotes the dynamics on labels induced by the gluing data $(H,\Phi)$ (see Definition~\ref{def: label dynamics}), and $\eta(i)$ denotes the partial assignment on the marked vertices of the $i$-th copy of the graph $G$ induced by the assignment $\eta: Y\to \{0,1\}$.  

    Note that 
    \[[\tau_{10}\wedge \eta]_i= [\tau_{11}\wedge \eta]_i \quad \text{ and } \quad [\widetilde{\tau}_{01}\wedge \widetilde{\eta}]_i= [\widetilde{\tau}_{00}\wedge \widetilde{\eta}]_i\]
    for each $i\notin\Phi(j)\sqcup \Phi(\widetilde{j})$, and thus 
    \[\prod_{i\notin\Phi(j)\sqcup\Phi(\widetilde{j})}([\tau_{10}\wedge \eta]_i)([\widetilde{\tau}_{01}\wedge \widetilde{\eta}]_i) =\prod_{i\notin\Phi(j)\sqcup\Phi(\widetilde{j})}([\tau_{11}\wedge \eta]_i)([\widetilde{\tau}_{00}\wedge \widetilde{\eta}]_i).\]
    Using \eqref{eq: product in copies} and analogous identities for the other three involved products, we obtain 
    \begin{align*}
        &\prod_{i=1}^m ([\tau_{10}\wedge \eta]_i)([\widetilde{\tau}_{01}\wedge \widetilde{\eta}]_i)-\prod_{i=1}^m ([\tau_{11}\wedge \eta]_i)([\widetilde{\tau}_{00}\wedge \widetilde{\eta}]_i)=\\
        &\prod_{i\in \Phi(j)}(1_{\Lambda(j)}\wedge\eta(i))(0_{\Lambda(j)}\wedge\widetilde{\eta}(i))\; \times \prod_{i\notin\Phi(j)\sqcup\Phi(\widetilde{j})}([\tau_{11}\wedge \eta]_i)([\widetilde{\tau}_{00}\wedge \widetilde{\eta}]_i) \\ 
        & \quad \times \left( \prod_{i\in \Phi(\widetilde{j})}(0_{\Lambda(\widetilde{j})}\wedge\eta(i))(1_{\Lambda(\widetilde{j})}\wedge\widetilde{\eta}(i)) - \prod_{i\in \Phi(\widetilde{j})}(1_{\Lambda(\widetilde{j})}\wedge\eta(i))(0_{\Lambda(\widetilde{j})}\wedge\widetilde{\eta}(i))\right).
    \end{align*}
    Interchanging the roles of $\eta$ and $\widetilde{\eta}$, we get
    \begin{align*}
        &\prod_{i=1}^m ([\tau_{10}\wedge \widetilde{\eta}]_i)([\widetilde{\tau}_{01}\wedge \eta]_i)-\prod_{i=1}^m ([\tau_{11}\wedge \widetilde{\eta}]_i)([\widetilde{\tau}_{00}\wedge \eta]_i)=\\
        & \prod_{i\in \Phi(j)}(1_{\Lambda(j)}\wedge\widetilde{\eta}(i))(0_{\Lambda(j)}\wedge\eta(i))\; \times \prod_{i\notin\Phi(j)\sqcup\Phi(\widetilde{j})}([\tau_{11}\wedge \widetilde{\eta}]_i)([\widetilde{\tau}_{00}\wedge \eta]_i) \\ 
        & \quad \times \left( \prod_{i\in \Phi(\widetilde{j})}(0_{\Lambda(\widetilde{j})}\wedge\widetilde{\eta}(i))(1_{\Lambda(\widetilde{j})}\wedge\eta(i)) - \prod_{i\in \Phi(\widetilde{j})}(1_{\Lambda(\widetilde{j})}\wedge\widetilde{\eta}(i))(0_{\Lambda(\widetilde{j})}\wedge\eta(i))\right).
    \end{align*}
    Observing that the bracketed factors in the two identities above are additive inverses, we substitute these two identities into~\eqref{eq: double difference} and obtain:

    \begin{align*}
      &2\big((\tau_{10})'\cdot(\widetilde{\tau}_{01})'-(\tau_{11})'\cdot(\widetilde{\tau}_{00})'\big)=\\ 
      &\sum_{\eta,\widetilde{\eta}: Y\to \{0,1\}} \lambda^{-\|\tau_{10}\wedge \eta\|_H-\|\widetilde{\tau}_{01}\wedge \widetilde{\eta}\|_H} \left( \prod_{i\in \Phi(\widetilde{j})}(0_{\Lambda(\widetilde{j})}\wedge\eta(i))(1_{\Lambda(\widetilde{j})}\wedge\widetilde{\eta}(i)) - \prod_{i\in \Phi(\widetilde{j})}(1_{\Lambda(\widetilde{j})}\wedge\eta(i))(0_{\Lambda(\widetilde{j})}\wedge\widetilde{\eta}(i))\right)\\
      &\times \quad \left(\; \prod_{i\in \Phi(j)}(1_{\Lambda(j)}\wedge\eta(i))(0_{\Lambda(j)}\wedge\widetilde{\eta}(i)) \; \times \prod_{i\notin\Phi(j)\sqcup\Phi(\widetilde{j})}([\tau_{11}\wedge \eta]_i)([\widetilde{\tau}_{00}\wedge \widetilde{\eta}]_i) - \right. \\&\left.  \quad \quad\quad \prod_{i\in \Phi(j)}(1_{\Lambda(j)}\wedge\widetilde{\eta}(i))(0_{\Lambda(j)}\wedge\eta(i)) \; \times \prod_{i\notin\Phi(j)\sqcup\Phi(\widetilde{j})}([\tau_{11}\wedge \widetilde{\eta}]_i)([\widetilde{\tau}_{00}\wedge \eta]_i)\; \right).
        \end{align*}

    We now require the following elementary fact. 

    \begin{fact}
        Suppose $a_1,\dots, a_n, b_1, \dots, b_n\in \C$ satisfy $|a_i-b_i|\leq \epsilon$ for each $i\in\{1,\dots,n\}$. Then 
        \[|a_1\cdots a_n-b_1\cdots b_n|\leq \epsilon \cdot Q(|a_1|,\dots, |a_n|, |b_1|, \dots, |b_n|), \]
        where $Q$ is a polynomial of degree $n-1$. 
    \end{fact}

    Note that for each $i\in \Phi(j)$, using~\eqref{eq: no correlation inequality},  we have:
    \[\big|(1_{\Lambda(j)}\wedge\eta(i))(0_{\Lambda(j)}\wedge\widetilde{\eta}(i)) - (1_{\Lambda(j)}\wedge\widetilde{\eta}(i))(0_{\Lambda(j)}\wedge\eta(i))\big|\leq\epsilon. \]
    Similarly, for each $i\in \Phi(\widetilde{j})$, we have
    \[\big|(0_{\Lambda(\widetilde{j})}\wedge\eta(i))(1_{\Lambda(\widetilde{j})}\wedge\widetilde{\eta}(i)) - (1_{\Lambda(\widetilde{j})}\wedge\eta(i))(0_{\Lambda(\widetilde{j})}\wedge\widetilde{\eta}(i))\big|\leq\epsilon,\]
    as well as 
    \[\big|([\tau_{11}\wedge \eta]_i)([\widetilde{\tau}_{00}\wedge \widetilde{\eta}]_i)-
    ([\tau_{11}\wedge \widetilde{\eta}]_i)([\widetilde{\tau}_{00}\wedge \eta]_i) \big| \leq \epsilon\]
    for each $i\notin\Phi(j)\sqcup\Phi(\widetilde{j})$.

    The claim now follows by applying the triangle inequality to the resulting expression for the difference $2\big((\tau_{10})'\cdot (\widetilde{\tau}_{01})'-(\tau_{11})'\cdot (\widetilde{\tau}_{00})'\big)$ and then applying the above fact to each difference of products appearing in that expression.
\end{proof}

\begin{remark}
    We note that when the gluing data is expanding but degenerate, the estimates obtained in Proposition~\ref{prop: euclidean attracting} do not imply that orbits starting sufficiently close to $\MM$ converge to $\MM$, even when the orbits are well-defined. Let us clarify the reason for this.

    Since the gluing data is assumed to be expanding, we may consider a sufficiently large iterate of $\RR$ for which the conditions in Proposition~\ref{prop: euclidean attracting} are satisfied. Suppose the orbit of a point $\xi \in \P^{2^k-1}$ avoids the indeterminacy set. We can consider a lift $\widehat{\xi} \in \C^{2^k}$ of norm $1$, and likewise a lift $\widehat{F(\xi)}$ of norm $1$. However, if $\xi$ is close to the indeterminacy set of $F$, then the norm of $\widehat{\xi}^\prime$ can be arbitrarily close to zero, using the notation in the proof of Proposition~\ref{prop: euclidean attracting}. Thus 
    \[
    \widehat{F(\xi)} = z \cdot \widehat{\xi}^\prime
    \]
    for a complex number $z$ with arbitrarily large absolute value, and the upper estimates on defining equations for the point $\widehat{\xi}^\prime$ do not give any information on the defining equations for the point $\widehat{F(\xi)}$, which in turn correspond to the distance of $F(\xi)$ to $\MM$ in the Fubini--Study metric.

    We note that if one can somehow conclude that the orbit of points remains bounded away from the indeterminacy set, for example when $\lambda>0$ and one only considers points in $\P^{2^k-1}$ with non-negative coordinates, then indeed there exists a neighborhood of $\MM$ in which orbits converge to $\MM_0$ superexponentially fast. The non-degenerate setting, where points sufficiently close to $\MM\setminus I(F)$ converge superexponentially fast to fixed points in $\MM_0$, is another setting in which the estimates from  Proposition~\ref{prop: euclidean attracting} are sufficient.
\end{remark}

\section{Decay of correlation}\label{sec: decay of correlation}

In this section we study \emph{correlation properties} of the probability measure $\P_{G,\lambda}$ (see \eqref{eq: prob_of_state}) for positive parameters $\lambda$ and sequences $(G_n)_{n\geq0}$ of graphs with $k\geq 2$ marked vertices. More precisely, we show that the maximal correlation between the marked vertices of $G_n$ decays to $0$ as $n\to\infty$, provided that the sequence $(G_n)_{n\geq0}$ has \emph{finite order of ramification} between the marked vertices (see Definition~\ref{def: finite ram order} and Proposition~\ref{prop: decay of correlation}). In particular, this condition is satisfied when the graph sequence is recursively generated by a non-degenerate and expanding gluing data (Theorem~\ref{thm: decay of correlation}). This allows us to derive important consequences for the dynamics of the associated renormalization map $F_\lambda$ (see Corollary~\ref{cor: conv_for_positive}), which will be used in the next section to establish the absence of zeros of the independence polynomials of $(G_n)_{n\geq 0}$ near  the positive real axis (see Theorem~\ref{thm: zero-free}).

We begin by formalizing the notion of decay of correlation between the marked vertices in a graph sequence in $\GG_k$.

\begin{definition}\label{def: decay of correlation}
    Let $(G_n)_{n\geq 0}$ be an arbitrary sequence of graphs, each with $k \ge 2$ marked vertices. We say that \emph{decay of correlation occurs between the marked vertices of the graphs $G_n$} if, for each $\lambda\in \R_+$, 
    \[ \theta_\lambda(G_n)=\max_{\tau,\sigma}\big|\P_{G_n,\lambda}[\tau: \sigma]- \P_{G_n,\lambda}[\tau]\big| \to 0 \text{ as } n\to \infty, 
    \]    
    where the maximum is taken over all assignments $\tau$ and $\sigma$ on disjoint subsets of the marked vertices of~$G_n$. 
\end{definition}

In the above definition, we implicitly assume that, for all sufficiently large $n$,  $\P_{G_n,\lambda}[\sigma]>0$ for every partial assignment $\sigma$ on the marked vertices of $G_n$, so that the conditional probability $\P_{G_n,\lambda}[\tau: \sigma]$ is well-defined. Equivalently, this means that no two marked vertices of $G_n$ are adjacent for all sufficiently large $n$.

We record the following straightforward consequence of decay of correlation, valid for arbitrary sequences in $\GG_k$ and requiring no bound on the vertex degrees.

\begin{lemma}\label{lem: conv_to_M}
    Let $\lambda\in \R_+$, and let $(G_n)_{n\geq 0}$ be a sequence of graphs in $\GG_k$ such that decay of correlation occurs between the marked vertices of the graphs $G_n$. Then the sequence of points $\phi_\lambda(G_n)$, $n\geq0$, converges to the invariant variety $\MM$ for $\GG_k$.
\end{lemma}

\begin{proof} Since $\P^{2^k-1}$ is compact, it is sufficient to prove that any limit point of the sequence $(\phi_\lambda(G_n))_{n\geq0}$ lies in the invariant variety $\MM$. 

Let $\xi$ be a limit point of a subsequence $\xi_r=\phi_\lambda(H_r)$ with $H_r=G_{n_r}$, $r\geq 0$. We will work in the chart 
\[\Sigma_1:=\big\lbrace\sum_{\xbf\in \{0,1\}^k} (\xbf)=1\big\rbrace\]
on $\P^{2^k-1}$. Note that since $\lambda>0$, each $\xi_r$ has a lift $\widehat{\xi}_r\in \Sigma_1$ with non-negative coordinates. In fact, $\widehat{\xi}_r$ is the probability vector with entries $\P_{H_r,\lambda}[\xbf]$, $\xbf\in \{0,1\}^k$. It follows that the limit point $\xi$ has a lift $\widehat{\xi}\in \Sigma_1$, which is a probability vector as well. 

Let $\tau$ and $\sigma$ be arbitrary assignments on disjoint subsets of the marked vertices of~$H_r$. Then for all sufficiently large $r$ we have $\P_{H_r,\lambda}[\sigma]>0$, and thus 
\[
\big|\P_{H_r,\lambda}[\tau\wedge \sigma]- \P_{H_r,\lambda}[\tau]\cdot\P_{H_r,\lambda}[\sigma]\big|=\P_{H_r,\lambda}[\sigma]\cdot \big|\P_{H_r,\lambda}[\tau: \sigma]- \P_{H_r,\lambda}[\tau]\big| \leq  \big|\P_{H_r,\lambda}[\tau: \sigma]- \P_{H_r,\lambda}[\tau]\big|.
\]
By hypothesis, decay of correlation occurs between the marked vertices of the graphs $G_n$, and thus
 \[ \max_{\tau,\sigma}\big|\P_{H_r,\lambda}[\tau\wedge \sigma]- \P_{H_r,\lambda}[\tau]\cdot\P_{H_r,\lambda}[\sigma]\big| \to 0 \text{ as } r\to \infty,
    \]    
 where the maximum is taken over all assignments $\tau$ and $\sigma$ on disjoint subsets of the marked vertices of~$H_r$. It follows that the coordinates of $\widehat{\xi}$ satisfy equations~\eqref{eq: no direct correlation}, where we interpret the conditioned independence polynomials as the sum of corresponding coordinates of $\widehat{\xi}$. By the proof of Lemma~\ref{lem: direct vs conditional correlation}, the coordinates of $\widehat{\xi}$ then also satisfy equations~\eqref{eq: no correlation}. Finally, the proof of Lemma~\ref{lem: inv variety} implies that the defining equations of $\MM$ vanish at $\widehat{\xi}$, which finishes the proof of the lemma.\end{proof}

The main result of this section is the following.

\begin{theorem}\label{thm: decay of correlation}
     Let $(H,\Phi)$ be a non-degenerate and expanding gluing data with parameters $k\geq 2, m\geq 2$, and let $\RR=\RR_{(H,\Phi)}$ be the associated graph recursion on $\GG_k$. Suppose $(G_n)_{n\geq 0}$, with $G_{n+1}=\RR(G_n)$, is a recursive graph sequence with a starting graph $G_0\in \GG_k$. Then decay of correlation occurs between the marked vertices of the graphs $G_n$.
\end{theorem}

\begin{remark}
    We note that decay of correlation is not a straightforward consequence of the fact that the graph distance between the marked vertices in $G_n$ increases with $n$. For example, fix $d\geq 2$ and consider the sequence of discrete $d$-dimensional tori $T_n=(\Z/2n\Z)^d$, which have $(2n)^d$ vertices. Note that the graphs $T_n$ are bipartite: the vertex set $V(T_n)$ is naturally subdivided into \emph{even} and \emph{odd sublattices} consisting of vertices with even and odd coordinate sums, respectively. For each $n$, we mark a pair $(u_n, v_n)$ of vertices in the even sublattice of $T_n$ so the graph distance between them diverges to $\infty$ as $n\to \infty$.

    It is known that the hard-core model on $T_n$ exhibits \emph{long-range order} for all sufficiently large $\lambda>0$. Namely, for every $\epsilon>0$, at sufficiently large fixed $\lambda>0$, there are disjoint \emph{even-phase} and \emph{odd-phase} events $\Omega^{even}_{n}$ and $\Omega^{odd}_{n}$ such that, for a random independent subset $I\subset V(T_n)$ and any fixed vertex $w$ in the even sublattice,
    \begin{equation}\label{eq: torus_1}
        \P_{T_n,\lambda}[\Omega^{even}_{n}]=\P_{T_n,\lambda}[\Omega^{odd}_{n}]=1/2+o(1)
    \end{equation}
    and such that 
    \begin{equation}\label{eq: torus_2}
        \P_{T_n,\lambda}[w\notin I: \Omega^{even}_{n}]\leq \epsilon,
    \end{equation}
    where $o(1)$ refers to $n\to\infty$ with $\lambda$ fixed.
    These estimates follow from Pirogov--Sinai theory combined with cluster expansions, by rewriting the independence polynomial $Z_{T_n}(\lambda)$ in terms of the \emph{contour partition function} on $T_n$ with parameter $z=\frac{1}{\lambda}$; see  \cite{BorgsImbrie} and \cite[Section~6]{HelmuthPerkinsRegts}. Informally speaking, the event $\Omega^{even}_{n}$ (resp.\ $\Omega^{odd}_{n}$) consists of those independent sets of $T_n$ that deviate from the even (resp.\ odd) sublattice on a small set of vertices.

    Let us now fix some $\epsilon\in (0,\frac{1}{4})$. By vertex transitivity, we have that $\P_{T_n,\lambda}[w\in I]$ does not depend on $w\in V(T_n)$, and since two adjacent vertices cannot be in $I$ at the same time, we actually have $\P_{T_n,\lambda}[w\in I]\leq \frac{1}{2}$. Using~\eqref{eq: torus_1} and~\eqref{eq: torus_2}, we obtain  
    \begin{align*}
        \P_{T_n,\lambda}[u_n,v_n\in I]&\geq \P_{T_n,\lambda}[u_n,v_n\in I: \Omega_n^{even}] \cdot \P_{T_n,\lambda}[\Omega_n^{even}] \\&\geq (1-\P_{T_n,\lambda}[u_n\notin I: \Omega_n^{even}]-\P_{T_n,\lambda}[v_n\notin I: \Omega_n^{even}])\cdot \P_{T_n,\lambda}[\Omega_n^{even}] \\&\geq \frac{1}{2} - \epsilon + o(1),
    \end{align*}   
    and thus 
    \[
       \P_{T_n,\lambda}[u_n\in I : v_n\in I] -\P_{T_n,\lambda}[u_n\in I] = \frac{\P_{T_n,\lambda}[u_n,v_n\in I]}{\P_{T_n,\lambda}[v_n\in I]} - \P_{T_n,\lambda}[u_n\in I] \geq \frac{1}{2}-2\epsilon +o(1).
    \] 
    It follows that decay of correlation does not occur, even though the graph distance between $u_n$ and $v_n$ diverges.
\end{remark}

We now record an immediate consequence of Theorem~\ref{thm: decay of correlation} and Lemma~\ref{lem: conv_to_M}.

\begin{corollary}\label{cor: conv_for_positive}
    Suppose we are in the setting of Theorem~\ref{thm: decay of correlation}, and let $F_\lambda : \P^{2^k-1}\dashrightarrow \P^{2^k-1}$ be the renormalization map associated with the graph recursion $\RR$ for some fixed parameter $\lambda\in \R_+$.  Then the sequence of points $F_\lambda^n(\phi_\lambda(G_0))=\phi_\lambda(G_n)$, $n\geq0$, converges to the invariant variety $\MM$ for $\GG_k$.
\end{corollary}

 \begin{rem}\label{rem: conv_to_M0}
    If, in addition to the assumptions of Corollary~\ref{cor: conv_for_positive}, we assume that the graph recursion $\RR$ is pre-fixed, then the sequence $(\phi_\lambda(G_n))_{n\geq 0}$ converges to a fixed point in the subvariety $\MM_0(\lambda)\subset \MM$ from Proposition~\ref{prop: pre fixed dynamics}. Indeed, if $\xi$ is a limit point of the sequence, then $\xi\in \MM$ by the corollary, and thus the fixed point $F_\lambda(\xi)\in \MM_0(\lambda)$ is also a limit point. Theorem~\ref{thm: superattraction} and Lemma~\ref{lem: trans_superattr} now imply that $(\phi_\lambda(G_n))_{n\geq 0}$ converges to $F_\lambda(\xi)$. 
 \end{rem}

Before proving Theorem~\ref{thm: decay of correlation}, we first provide some auxiliary definitions and facts.

\begin{definition}[Finite order of ramification]\label{def: finite ram order} 
Let $k\geq 2$, and let $(G_n)_{n\geq 0} \subset \GG_k$ be a 
sequence of marked graphs.  We say that $(G_n)_{n\geq 0}$ has \emph{finite order of ramification} if there exists a constant $M \in \N$ such that the following condition holds:

For each $n\geq 0$ and each marked vertex $v\in P(G_n)$, there exist pairwise disjoint subsets 
\begin{equation}\label{eq: sep_seq}
    U_0=\{v\}, U_1, \dots, U_T, U_{T+1}=P(G_n)\setminus \{v\} \subset V(G_n),
\end{equation}    
with $T=T(v,n)\geq 0$, such that:
\begin{enumerate}[label=(\roman*)]
\item\label{item: FOR-i} every subset $U_t$ is non-empty and contains at most $M$ vertices;
\item\label{item: FOR-ii} for every $t\in\{1,\dots, T\}$, the set $U_t$ \emph{separates} $U_0$ from $U_{t+1}$ in $G_n$, meaning that each path in $G_n$ from $v$ to a vertex in $U_{t+1}$ passes through some vertex in $U_t$;
\item\label{item: FOR-iii} $\min_{v\in P(G_n)}{T(v,n)} \to \infty$ as $n\to \infty$.
\end{enumerate}

We will refer to a sequence $U_0, \dots, U_{T+1}$ as in \eqref{eq: sep_seq} that satisfies condition~
\ref{item: FOR-ii} above as a \emph{separating sequence} for the pair $(G_n, v)$, and to its elements as \emph{separating sets}. 
\end{definition}

This notion was inspired by Definition~\ref{def: finite order of ramification} for a single infinite graph. The differences are clear: here we consider a sequence of finite graphs $G_n$, and only paths between the marked vertices are relevant.

\begin{remark}
    If $(G_n)_{n\geq 0}\subset \GG_k$ has finite order of ramification, then the graph distance between the marked vertices in $G_n$ necessarily diverges as $n\to \infty$. Consequently, for each $\lambda\in \R_+$ and each partial vertex assignment $\sigma$ on $P(G_n)$, the probability $\P_{G_n,\lambda}[\sigma]$ is positive for all sufficiently large $n$.
\end{remark}

The next lemma shows that the recursive graph sequences generated by a non-degenerate and expanding gluing data have finite order of ramification. 

\begin{lemma}\label{lem: recursion implies finite ramification}
    Let $(H,\Phi)$ be a non-degenerate and expanding gluing data with parameters $k\geq 2, m\geq 2$. Suppose $(G_n)_{n\geq 0}\subset \GG_k$ is a recursive sequence of marked graphs constructed using the graph recursion operator $\RR=\RR_{(H,\Phi)}$. Then $(G_n)_{n\geq 0}$ has finite order of ramification.
\end{lemma}
\begin{proof}
    The idea of the proof is simple: for a graph $G_n$ and a marked vertex $v\in P(G_n)$ the separating sets $U_t$ will correspond to the marked vertices of the copies $G_t(i)$ that contain the vertex $v$. The expansion property of $\RR$ guarantees that a path from $v$ to another marked vertex of $G_n$ must pass through an increasing number of these separating sets as $n \rightarrow \infty$, while the non-degeneracy of $\RR$ guarantees that the number of copies $G_t(i)$ containing $v$ stays bounded, which implies that the cardinality of each separating set $U_t$ is uniformly bounded. We now provide a careful realization of this idea.
    
    Given $n\in \N$, let us write $(H_n,\Phi_n)$ for the gluing data with parameters $k,m^n$ induced by the $n$-th iterate of the recursion operator $\RR=\RR_{(H,\Phi)}$. Since the given gluing data $(H,\Phi)$ is expanding, there is an iterate $N$ such that the edges $\Phi_{N}(j)$ and $\Phi_N(j')$ are disjoint for each pair of distinct labels $j,j'\in \{1,\dots,k\}$. Moreover, we may assume that the induced dynamical system $\Lambda^N=\Lambda^N_{(H,\Phi)}$ on labels is pre-fixed: $\Lambda^N(j)=j$ for each periodic label $j\in \{1,\dots, k\}$, while $\Lambda^{2N}(j)=\Lambda^N(j)$ for each non-periodic label $j$. Since $(H,\Phi)$ is non-degenerate, we then have that each periodic label $j$ is not critical for $\Lambda^N$, that is, $\#\Phi_N(j)=1$.    

    First, we construct inductively the required separating sequence $U_0,\dots, U_{T+1}$ for each pair $(G_n,v)$ where the marked vertex $v\in P(G_n)$ has some periodic label $j\in \{1,\dots,k\}$. When $n\in \{0,\dots, N-1\}$, we simply set $T=0$, $U_0=\{v\}$, and $U_{T+1}= P(G_n)\setminus \{v\}$. When $n\geq N$, we can write $G_n=\RR^N(G_{n-N})$, that is, consider $G_n$ as the result of applying the graph recursion $\RR^N$ to $G_{n-N}$. Since $\#\Phi_N(j)=1$ and $\Lambda^N(j)=j$, the marked vertex $v$ of $G_n$ is induced by a single copy $G_{n-N}(i)$ of $G_{n-N}$ during the corresponding identification process, and it is also labeled $j$ in this copy. Denoting by $G'_{n-N}\in \GG_k$ the corresponding (marked) subgraph of $G_n$, we may consider the respective separating sequence \[U'_0=\{v\},\, U'_1,\, \dots,\, U'_{T'+1}=P(G'_{n-N})\setminus\{v\} \subset V(G_n)\] for the pair $(G'_{n-N}, v)$. We then set $T=T(n,v):=T'+1$ and 
    \[U_0=U'_0=\{v\},\, \dots, \, U_T=U'_{T}, \,U_{T+1}=P(G_n)\setminus\{v\}.\] 
    By the choice of the iterate $N$, the vertices in $U_{T+1}$ are induced by the copies of $G_{n-N}$ that are distinct from $G_{n-N}(i)$, and thus the set $U_T$ separates $U_0$ from $U_{T+1}$ in $G_n$. Moreover, by the inductive construction, it follows that $U_t$ separates $U_0$ from $U_{t+1}$ in $G_n$ for each $t\in \{1,\dots, T\}$. Thus the constructed sets $U_0,\dots,U_{T+1}$ form a separating sequence for $(G_n,v)$. It also follows that $T=\lfloor{n/N}\rfloor$ and $\#U_1=\dots=\#U_{T+1}=k-1$.

    When the label $j$ of a marked vertex $v\in P(G_n)$ is non-periodic, we similarly set $T=0$, $U_0=\{v\}$, and $U_{T+1}= P(G_n)\setminus \{v\}$ when $n\in \{0,\dots, N-1\}$. Otherwise, when $n\geq N$, we can again write $G_n=\RR^N(G_{n-N})$. For each $i\in \Phi_N(j)$, let $G'_{n-N}(i)\in \GG_k$ be the (marked) subgraph of $G_n$ corresponding to the copy $G_{n-N}(i)$ of $G_{n-N}$. Then $v$ is a marked vertex in each such subgraph and has the periodic label $\Lambda^N(j)$. 
    Taking now the respective separating sets in each subgraph $G'_{n-N}(i)$ and adding $U_{T+1}=P(G_n)\setminus\{v\}$, we obtain the desired separating sequence for $(G_n,v)$. More precisely, if \[U'_0(i)=\{v\},\, U'_1(i), \dots, U'_{T'+1}(i)=P(G'_{n-N}(i))\setminus\{v\} \subset V(G_n)\] denotes the respective separating sequence for $(G'_{n-N}(i),v)$ for $i\in \Phi_N(j)$, we set $T=T(n,v):=T'+1$, $U_0=\{v\}$, $U_{T+1}= P(G_n)\setminus \{v\}$, and 
    \[U_t=\bigcup_{i\in \Phi_N(j)}U'_t(i), \quad \text{for each $t\in \{1,\dots, T\}$}.\] 
    By the inductive construction, the sets $U_0,\dots,U_{T+1}$ form a separating sequence for $(G_n,v)$. It also follows that $T=\lfloor{n/N}\rfloor$ and \[(k-1)\leq \#U_t \leq \#\Phi_N(j)\cdot (k-1) \leq m^N(k-1)\] for each $t\in \{1,\dots, T\}$.
    
    By construction, for all $n\geq 0$ and $v\in P(G_n)$, the size of each separating set for $(G_n,v)$ is at most $M:=(k-1)m^N$. Moreover, \[\min_{v\in P(G_n)}{T(v,n)}= 
\left\lfloor{n/N}
\right\rfloor \to \infty \quad\text{as $n\to \infty$}.\] This finishes the proof of the lemma.
\end{proof}

By Lemma~\ref{lem: recursion implies finite ramification}, Theorem~\ref{thm: decay of correlation} is an immediate consequence of the following more general result.

\begin{prop}\label{prop: decay of correlation}
    Let $k \ge 2$, and let $(G_n)_{n\geq 0}\subset \GG_k$ be a sequence of marked graphs having finite order of ramification. Then decay of correlation occurs between the marked vertices of the graphs $G_n$.

\end{prop}

\begin{proof}
    We will first prove the desired convergence statement from Definition~\ref{def: decay of correlation} for the case when $\tau$ is an assignment on a single marked vertex $v$ of $G_n$; in this case, $\sigma$ is an assignment on a subset of $P(G_n)\setminus \{v\}$. 

    Fix a sufficiently large $n\in\N$ such that no two marked vertices of $G_n$ are connected by an edge, and let 
    \[
        U_0=\{v\}, U_1, \dots, U_T, U_{T+1}=P(G_n)\setminus \{v\} 
    \]
    be a separating sequence for $(G_n, v)$ as in Definition~\ref{def: finite ram order}, where $v\in P(G_n)$. For each $t\in\{1,\dots, T\}$, 
    let us consider the subset $U'_t\subset U_t$ consisting of those vertices $u\in U_t$ that may be connected to $v$ by a path in $G_n$ that does not pass through any vertex in $U_t\setminus\{u\}$. It is immediate that $U'_0=U_0, U'_1, \dots, U'_{T+1}=U_{T+1}$ is a separating sequence for $(G_n, v)$. Moreover, by construction, there are no edges in $G_n$ connecting a vertex in $U'_t$ with a vertex in $U'_{t+2}$ for all $0\leq t\leq T-1$. It follows that, up to redefining the separating sequence $U_0, \dots, U_{T+1}$ for $(G_n, v)$, we may assume that there are no edges in $G_n$ connecting a vertex in $U_t$ with a vertex in $U_{t+1}$ for each $0 \le t \le T$. For simplicity, we set $\P:=\P_{G_n,\lambda}$ in the following, where $\lambda$ is a fixed positive parameter. Moreover, we redefine $U_{T+1}$ to be an arbitrary non-empty subset of $P(G_n)\setminus \{v\}$.

    Now suppose that $\eta_t$ with $t\in \{0,\dots, T+1\}$ is an assignment on $U_t$. Recall that $\eta_t$ is called admissible if $\P[\eta_t] > 0$. We remark that the assignments $\eta_0$ and $\eta_{T+1}$ are always admissible. Moreover, for each $t\in\{0,\dots, T\}$ and  admissible $\eta_t$, the assignment $\eta_{t+1}$ is admissible if and only if  $\eta_t\wedge \eta_{t+1}$ is admissible.

    Using~\eqref{prop:ind-poly-3} and Lemma~\ref{lem: conditioning wrt Y}, for each admissible assignment $\eta_t$ with $t\in \{0,\dots, T\}$, we may write
    \begin{equation}\label{eq: conditioning wrt Ut}
    \begin{aligned}
    \P[\eta_t: \eta_{T+1}] 
    &= \sum_{\text{adm.\ $\eta_{t+1}$}}  \P[\eta_t: \eta_{t+1}\wedge \eta_{T+1}]\cdot \P[\eta_{t+1}: \eta_{T+1}]\\
    &= \sum_{\text{adm.\ $\eta_{t+1}$}} \P[\eta_t: \eta_{t+1}]\cdot \P[\eta_{t+1}: \eta_{T+1}],
    \end{aligned}
    \end{equation}
    where the sums range over all admissible assignments $\eta_{t+1}$ on $U_{t+1}$. For each $t\in \{0,\dots, T\}$, let $A_{t,n}$ be the matrix whose entries are given by the conditional  probabilities $\P[\eta_t: \eta_{t+1}]$, with rows corresponding to fixed admissible assignments $\eta_t$ on $U_t$ and columns corresponding to fixed admissible assignments $\eta_{t+1}$ on $U_{t+1}$. Clearly, every $A_{t,n}$ is a stochastic matrix.
    Moreover, by Lemma~\ref{lem: total var bound}, the Dobrushin contraction coefficient of $A_{t,n}$ satisfies $\delta(A_{t,n})\leq \nu$ for a constant $\nu=\nu(\lambda, k,M)\in (0,1)$, depending only on the fixed parameter $\lambda\in \R_+$, the cardinality $k$ of the marked set $P(G_n)$, and the maximal cardinality $M$ of the separating sets~$U_t$.
    
    Up to an appropriate ordering of the rows and columns in the matrices $A_{t,n}$, equation~\eqref{eq: conditioning wrt Ut} implies that the entries of the product 
    \[B_n:=A_{0,n}\cdots A_{T,n}\]
    represent the conditional probabilities $\P[\eta_0: \eta_{T+1}]$, where the rows and columns of $B_n$ correspond to fixed assignments $\eta_0$ and $\eta_{T+1}$ on $U_0$ and $U_{T+1}$, respectively. By Lemma~\ref{lem: Dobrushin contraction}, we then have $\delta_n=\delta(B_n) \xrightarrow[n \to \infty]{} 0$ at least at the rate $\nu^{T(v,n)}$.

    Now, let us fix an assignment $\tau$ on $U_0=\{v\}$ and an assignment $\sigma$ on $U_{T+1}\subset P(G_n)\setminus \{v\}$. Then
    \begin{align*}
        \big|\P[\tau: \sigma]-\P[\tau]\big|&=\big|\P[\tau: \sigma]\cdot \sum_{\eta_{T+1}}\P[\eta_{T+1}]-
        \sum_{\eta_{T+1}}\P[\tau:\eta_{T+1}]\cdot \P[\eta_{T+1}]\big|\\
        &\leq\sum_{\eta_{T+1}}\P[\eta_{T+1}] \cdot \big|\P[\tau: \sigma]-\P[\tau:\eta_{T+1}]\big| \leq \sum_{\eta_{T+1}}\P[\eta_{T+1}] \cdot (2\delta_n) = 2\delta_n,
    \end{align*} 
    where the sums range over all assignments $\eta_{T+1}$ on $U_{T+1}$. Since $\min_{v\in P(G_n)}{T(v,n)} \to \infty$ as $n\to \infty$, we conclude that  
    \[\Theta_n:=\max_{\tau,\sigma}\big|\P[\tau: \sigma]- \P[\tau]\big| \to 0 \text{ as } n\to \infty, \] where the maximum is taken over all assignments $\tau$ on a single marked vertex of $G_n$ and all assignments $\sigma$ on a non-empty subset of the remaining marked vertices of $G_n$. 

    Suppose now that $\tau$ and $\sigma$ are arbitrary assignments on disjoint subsets of the marked vertices of $G_n$. We may assume that these subsets are non-empty, since $\P[\tau: \sigma]- \P[\tau]$ is identically $0$ otherwise. Let $\{v_1,\dots,v_s\}\subset P(G_n)$ be the domain of $\tau$, and set $\tau_i:=\tau|\{v_i\}$ for each $i=1,\dots,s$, so that $\tau=\tau_1\wedge \dots \wedge \tau_s$. We then have
    \begin{align*}
        \P[\tau:\sigma]-\P[\tau]=\P[\tau_1:\sigma]\cdot\prod_{i=2}^{s}\P[\tau_i:\tau_{i-1}\wedge\dots\wedge \tau_1\wedge \sigma] - \P[\tau_1]\cdot\prod_{i=2}^{s}\P[\tau_i:\tau_{i-1}\wedge\dots\wedge \tau_1]. 
    \end{align*}
    Since $\big|\P[\tau_1:\sigma]-\P[\tau_1]\big|\leq \Theta_n$ and 
    \[
     \big|\P[\tau_i:\tau_{i-1}\wedge\dots\wedge \tau_1\wedge \sigma] - \P[\tau_i:\tau_{i-1}\wedge\dots\wedge \tau_1]\big|\leq 2\Theta_n \; \text{ for each $i=2,\dots, s$},
    \]
    we obtain that
    \[\big|\P[\tau:\sigma]-\P[\tau]\big|\leq (2s-1) \cdot \Theta_n < 2k \cdot \Theta_n.\]
    This implies the desired convergence statement from Definition~\ref{def: decay of correlation} and finishes the proof.
\end{proof}

\begin{rem}\label{rem: correlation_decay}
The proofs of Lemma~\ref{lem: recursion implies finite ramification} and Proposition~\ref{prop: decay of correlation} show that, in the setting of Theorem~\ref{thm: decay of correlation}, the maximal correlation between the marked vertices of $G_n$ decays to $0$ exponentially fast in $n$.
\end{rem}

\section{Absence of zeros} \label{section: zeros}

With the preparation above, we are finally ready to use our dynamical framework to restrict location of zeros of the independence polynomials for recursive graph sequences. We begin with the following key result. 

\begin{theorem}\label{thm: zero-free}
  Let $(H,\Phi)$ be a non-degenerate and expanding gluing data with parameters $k\geq 2, m\geq 2$, and let  
  $\RR=\RR_{(H,\Phi)}$ be the associated graph recursion on $\GG_k$. Fix an arbitrary starting graph $G_0\in \GG_k$, and consider the induced recursive graph sequence $(G_n)_{n\geq 0}$ with $G_{n+1}=\RR(G_n)$. Then the zeros of the independence polynomials  $Z_{G_n}$ avoid a neighborhood of $\R_{\geq 0}$. 
\end{theorem}

\begin{proof}
    First note that, since $(H,\Phi)$ is non-degenerate, the vertex degrees of the graphs $G_n$ are uniformly bounded by Lemma~\ref{lem: degeneration and expansion}. It then follows from a result of Shearer \cite[Theorem~2]{Shearer}, see also \cite[Corollaries~5.3 and 5.7]{ScottSokal05}, that the zeros of the independence polynomials $Z_{G_n}$ avoid a fixed neighborhood of the origin for all $n\geq 0$. Hence it is sufficient to prove that for a given $\lambda_0\in \R_+$ there exists a zero-free neighborhood of $\lambda_0$.

   Let $N\in\N$ be chosen so that the $N$-th iterate of the gluing data $(H,\Phi)$ is pre-fixed. Then we may decompose the sequence $(G_n)_{n\geq 0}$ into $N$ recursive subsequences $(\RR^{nN}(G_0))_{n\geq 0},\dots, (\RR^{nN}(G_{N-1}))_{n\geq 0}$ induced by the graph recursion $\RR^N$ and the starting graphs $G_0,\dots, G_{N-1}$, respectively. Since $\RR^N$ is also non-degenerate and expanding, it follows that it is sufficient to prove the existence of a zero-free neighborhood of $\lambda_0$ when the given gluing data $(H,\Phi)$ is pre-fixed (i.e., $N=1$), which we will assume in the following.  

     In line with notation from Section~\ref{subsec: induced dynamics}, let us write $\xi_0(\lambda) = \phi_{\lambda}(G_0) \in \P^{2^k-1}$, and consider the respective orbit
    \[
    \xi_n(\lambda) = \phi_{\lambda}(G_n) = F_{\lambda}^n(\xi_0(\lambda))
    \]
    under the renormalization map $F_{\lambda} : \P^{2^{k}-1}\dashrightarrow \P^{2^{k}-1}$; see Corollary~\ref{cor: renormalization map}. We denote by $\Sigma_0$ the zero set in $\P^{2^{k}-1}$ of the projection map $\Sigma: \C^{2^k}\to \C$ given by the sum of the coordinates.

    Since $\lambda_0\in \R_+$ and the gluing data $(H,\Phi)$ is assumed to be non-degenerate, expanding, and pre-fixed, Corollary~\ref{cor: conv_for_positive} combined with Remark~\ref{rem: conv_to_M0} implies that the orbit $(\xi_n(\lambda_0))_{n\geq0}$ must converge to a fixed point $p(\lambda_0)\in \MM_0(\lambda_0)$. Note also that this orbit stays within the compact subset of $\P^{2^k-1}$ consisting of points with non-negative coordinates, and therefore its limit $p(\lambda_0)\notin\Sigma_0$. In particular, we may find a neighborhood $U$ of $p(\lambda_0)$ that avoids $\Sigma_0$. 

    Since \[Z^{\xbf}_{G_n}(\lambda_0) \leq \max\{1,\lambda_0^k\}\cdot Z^{\obf}_{G_n}(\lambda_0)\] for all $\xbf\in \{0,1\}^k$ and $n\geq0$, we get that the $(\obf)$-coordinate of the limit point $p(\lambda_0)\in \MM_0(\lambda_0)\subset\MM$ does not vanish. It follows that, by shrinking the neighborhood $U$ of $p(\lambda_0)$, we may assume that $U$ avoids the indeterminacy set $\IS(F_\lambda)$ for all $\lambda$ in a complex neighborhood $D(\lambda_0)$ of $\lambda_0$.  

     \begin{claim}
        There exist a neighborhood $U'\subset U$ of $p(\lambda_0)$ and a neighborhood $D'(\lambda_0)\subset D(\lambda_0)$ of $\lambda_0$ such  that the orbits of points in $U'$ under $F_\lambda$ stay within $U$ for all $\lambda\in D'(\lambda_0)$.
    \end{claim}

    Indeed, let us consider the map $\widetilde{F}(\lambda,\xi)= (\lambda, F_\lambda(\xi))$ on $\mathcal{X}=\C^*\times \P^{2^k-1}$ and the complex submanifold \[\mathcal{Y}:=\{(\lambda, \xi): \xi\in \MM_0(\lambda)\}\] from Proposition~\ref{prop: pre fixed dynamics}. Set \[\mathcal{Y}_0:=(D(\lambda_0)\times U)\cap \mathcal{Y} \quad \text{and} \quad  U(\mathcal{Y}_0):=D(\lambda_0)\times U\subset\mathcal{X}.\] By above, $\widetilde{F}: U(\mathcal{Y}_0)\to \mathcal{X}$ is well-defined and holomorphic on the neighborhood $U(\mathcal{Y}_0)$ of the smooth submanifold $\mathcal{Y}_0\subset \mathcal{X}$. Since $\widetilde{F}$ fixes $\mathcal{Y}_0$ pointwise and the $\lambda$-direction is tangent to $\mathcal{Y}_0$, Theorem~\ref{thm: superattraction} implies that $\widetilde{F}$ is transversally superattracting on $\mathcal{Y}_0$. The statement of the claim now immediately follows from Lemma~\ref{lem: trans_superattr} applied to the point $(\lambda_0, p(\lambda_0))\in \mathcal{Y}_0$ and its coordinate neighborhood $D(\lambda_0)\times U$. 

    \medskip

    Now let $r\in \N$ be an iterate such that $\xi_r(\lambda_0)\in U'$. We may now shrink the neighborhood $D'(\lambda_0)\subset D(\lambda_0)$ further so that $\xi_r(\lambda)\in U'$ for all $\lambda\in D'(\lambda_0)$, while $\xi_0(\lambda), \dots, \xi_{r-1}(\lambda)$ avoid $\Sigma_0$. The claim above implies that the orbit $(\xi_n(\lambda))_{n\ge 0}$ avoids $\Sigma_0$ for all $\lambda\in D'(\lambda_0)$. It follows that the independence polynomials $Z_{G_n}$ have no zeros in the neighborhood $D'(\lambda_0)$. This completes the proof. 
\end{proof}

We now derive the less technical formulation of the main result as stated in the introduction.

\begin{proof}[Proof of Theorem~\ref{main result zeros}.]
Suppose $\RR=\RR_{(H,\Phi)}$ is a recursion operator on $\GG_k$, and let $(G_n)_{n\geq0}$, with $G_{n+1}=\RR(G_n)$, be a recursive graph sequence with a starting graph $G_0\in\GG_k$. We assume that the sequence $(G_n)_{n\geq0}$ satisfies the following two conditions:

\ref{item: non-deg intro} the maximal vertex degrees of the graphs $G_n$ are uniformly bounded, and

\ref{item: exp intro} the minimal distance between pairs of distinct labeled vertices in $G_n$ diverges as $n \rightarrow \infty$.

\smallskip

When $k=1$, condition~\ref{item: non-deg intro} implies that either the gluing data $(H,\Phi)$ is non-degenerate (i.e., $\#\Phi(1)=1$) or  the unique marked vertex of $G_0$ is isolated (see Lemma~\ref{lem: degeneration and expansion}). In either case, by definition of the graph recursion $\RR$, each connected component of $G_n$ with $n\geq 2$ appears as a connected component of $G_1$. It follows that the zero set of $Z_{G_n}$ is the same for all $n\geq 1$, and thus the assertion of the theorem trivially holds. Hence, we will assume that $k\geq 2$ in the following. 

If the starting graph $G_0$ is connected, the statement follows immediately from Theorem~\ref{thm: zero-free}, since, by Lemma~\ref{lem: degeneration and expansion}, assumptions~\ref{item: non-deg intro} and~\ref{item: exp intro} imply that the gluing data $(H,\Phi)$ is non-degenerate and expanding, respectively. However, when labeled vertices lie in different connected components of $G_0$, these assumptions do not guarantee that the gluing data is either expanding or non-degenerate. Nevertheless, we can essentially reduce the analysis in the general case to the setting of connected graphs as follows.

\smallskip

First,  without loss of generality, we may assume that the gluing data $(H,\Phi)$ is pre-fixed. Let $S\subset \{1,\dots, k\}$ denote the set of all periodic labels with respect to $\Lambda=\Lambda_{(H,\Phi)}$. For each $n\geq 0$, we consider an equivalence relation $\sim_n$ on $\{1,2,\dots, k\}$ defined by $j\sim_n j'$ if the vertices labeled $j$ and $j'$ belong to the same connected component of $G_n$. We denote by $\approx_n$ the restriction of $\sim_n$ to $S$. By definition of the recursion $\RR$, if $j\not\approx_n j'$ for some  $n\geq 0$ and $j,j'\in S$, then $j\not\approx_{n+1}j'$. It follows that the sequence $(\approx_n)_{n\geq0} \subset S\times S$ is decreasing, and thus it must eventually stabilize, that is, $\approx_{N}=\approx_{N+1}=\approx_{N+2}=\dots$ for some sufficiently large $N\in \N$. Moreover, if $\Lambda(j)\not\approx_N \Lambda(j')$ for some $j,j'\in \{1,\dots, k\}$, then $j\not\sim_{n}j'$ for all $n\ge N+1$. 

Since it is harmless to skip finitely many iterates, we will assume that the starting graph $G_0$ equals $G_{N+1}$ below, so that the following property holds: 
\begin{equation}\label{eq: conn_relation}
    \text{if $\Lambda(j)\not\approx_0 \Lambda(j')$ for some $j,j'\in \{1,\dots, k\}$, then $j\not\sim_{n}j'$ for all $n\ge 0$.} 
\end{equation}

For each equivalence class $\alpha=[j]\subset S$ of $\approx_0$, we now consider a new gluing data $(H^\alpha,\Phi^\alpha)$ obtained from $(H,\Phi)$ by ignoring the labels outside of $\Lambda^{-1}(\alpha)$. More precisely, the hypergraph $H^\alpha$ has the same vertex set $V(H^\alpha)=V(H)$, while the edge multi-set $E(H^\alpha)$ consists only of those edges $e\in E(H)$ whose labels $\ell(e)$ are contained in $\Lambda^{-1}(\alpha)$. The labeling map $\Phi^\alpha$ is defined to be the restriction of $\Phi$ to $\Lambda^{-1}(\alpha)$. The corresponding recursion operator $\RR^\alpha:=\RR_{(H^\alpha, \Phi^\alpha)}$ then acts on graphs with $\#\Lambda^{-1}(\alpha)$ labeled vertices with labels in $\Lambda^{-1}(\alpha)$. We define the starting graph $G^\alpha_0$ to be the graph obtained from $G_0$ by forgetting the labels outside of $\Lambda^{-1}(\alpha)$, and consider the induced recursive graph sequence $(G^\alpha_n)_{n\geq 0}$ with $G^\alpha_{n+1}=\RR^\alpha (G^\alpha_n)$. (We remark a slight abuse of notation compared to Definition~\ref{def:graph recursion}, where the labels are assumed to be in an integer interval $\{1,\dots, k\}$. Clearly, our terminology and discussion in Section~\ref{sec: graph recursion} can be naturally adapted to an arbitrary finite set of labels to cover the setting here.)  

By definition of the graph recursion $\RR$, when constructing $G_{n+1}=\RR(G_n)$ from $G_n$ we can only identify labeled vertices (with identical labels) in the $\#V(H)$ copies of the graph $G_n$. It now follows from \eqref{eq: conn_relation} that every connected component that occurs in $G_n$ must also occur in the graph $G_n^\alpha$ for some  equivalence class $\alpha$ of $\approx_0$. Since the independence polynomial of a disconnected graph equals the product of the independence polynomials of its connected components, to conclude the desired statements it is sufficient to check that the zeros
of the independence polynomials $Z_{G^\alpha_n}$ avoid a uniform neighborhood of $\R_{\geq 0}$ for each equivalence class $\alpha$. 

When $\#\alpha\geq 2$, the gluing data $(H^\alpha, \Phi^\alpha)$ is non-degenerate and expanding due to the standing assumptions~\ref{item: non-deg intro} and~\ref{item: exp intro} and Lemma~\ref{lem: degeneration and expansion} (even though the starting graph $G^\alpha_0$ may be disconnected). Hence Theorem~\ref{thm: zero-free} applies to the sequence $(G^\alpha_n)_{n\geq 0}$. When $\#\alpha=1$, the same is still true, unless the connected component of $G_0^\alpha$ containing the unique vertex with a periodic label is a singleton. In the latter case, the zeros of the independence polynomials $Z_{G^\alpha_n}$ for $n>0$ coincide with those of $Z_{G^\alpha_1}$, and thus obviously avoid a uniform neighborhood of $\R_{\geq 0}$. This finishes the proof of the theorem. 
\end{proof}

We will now switch our attention towards proving that the zeros of the independence polynomials are uniformly bounded for recursive graph sequences with a suitable starting graph.

\begin{definition}[Maximally independent graphs] \label{def: maximally}
    A graph $G\in \GG_k$ with $k\ge 2$ marked vertices is called \emph{maximally independent} if the following two conditions hold: 
    \begin{enumerate}[label=(\roman*)]
    \item\label{item: max_ind_1} for each vertex assignment $\xbf\in \{0,1\}^k$ on the marked set $P(G)$ there exists a unique largest independent subset $I_\xbf$ of $G$ that agrees with $\xbf$ on $P(G)$; 
    \item\label{item: max_ind_2} $\# I_{(1,\dots,1)} - \# I_{(0,\dots,0)} = k$. 
    \end{enumerate} 
\end{definition}

We observe that in this definition, the difference between the two cardinalities in condition~\ref{item: max_ind_2} is assumed to be \emph{maximal}: since any independent subset that agrees with the assignment $(1, \ldots, 1)$ on the marked vertices remains independent when all the marked vertices are taken out from it, we always have \[\# I_{(0,\dots,0)}\geq \# I_{(1,\dots,1)} - k.\] By the same argument applied to an arbitrary assignment $\xbf =(x_1,\dots, x_k)\in \{0,1\}^k$, conditions~\ref{item: max_ind_1} and~\ref{item: max_ind_2} imply that \[\#I_{(x_1,\dots,x_k)} - \# I_{(0,\dots,0)} = x_1+\dots+ x_k.\]

The simplest example of a maximally independent graph is a star:

\begin{example}[Maximally independent stars]
Recall that a $d$-star is a tree with a single non-leaf vertex of degree $d\geq 2$. It is straightforward to verify that, given any $d\geq k \geq  2$, a $d$-star with $k$ of its $d$ leaves labeled $1,\dots, k$ is maximally independent if and only if $d\geq k+2$. In particular, condition~\ref{item: max_ind_2} fails for $d=k$, while  condition~\ref{item: max_ind_1} fails for $d=k+1$.
\end{example}

The rest of the paper is devoted to proving Theorem~\ref{main bounded zeros} from the introduction: for non-degenerate and expanding graph recursion operators $\RR=\RR_{(H,\Phi)}$ on $\GG_k$ with $k\geq 2$ and for maximally independent starting graphs $G_0\in \GG_k$, the zeros of the independence polynomials of $G_n=\RR^n(G_0)$ are uniformly bounded (Theorem~\ref{thm: bounded}) and avoid a cone around $\R_{\geq0}$ (Corollary~\ref{cor: zero-free-cone}). 

Given $\lambda\in \C^*$, it will be convenient to replace the coordinates $(\xbf)$, $\xbf=(x_1,\dots,x_k)\in\{0,1\}^k$, on $\C^{2^{k}}$ and $\P^{2^{k}-1}$ from Section~\ref{subsec: induced dynamics} by
\begin{equation}\label{eq: new_coord}
    \llparenthesis \xbf \rrparenthesis =\llparenthesis x_1, \ldots, x_k \rrparenthesis := \lambda^{-(x_1+ \cdots + x_k)} \cdot (\!(x_1, \ldots, x_k)\!).
\end{equation}
To emphasize that we consider the renormalization map $F_\lambda$ associated with the graph recursion $\RR$ in these new coordinates, we will write $\F_\lambda$ (resp.\ $\widehat{\F}_\lambda$) instead of $F_\lambda$ (resp.\ $\widehat{F}_\lambda)$ below. Using notation from Section~\ref{subsec: induced dynamics}, the map $\widehat{\F}_\lambda$ is then given by
\begin{equation}\label{eq: x_prime_new_coord}
\llparenthesis \xbf \rrparenthesis'=\sum_{\eta: Y\to \{0,1\}} \lambda^{\#\eta^{-1}(1)} \prod_{i=1}^m  \llparenthesis [\xbf \wedge\eta]_i\rrparenthesis, \quad \xbf\in \{0,1\}^k, 
\end{equation}
where we follow the convention that $\left( \llparenthesis \xbf \rrparenthesis': \xbf\in \{0,1\}^k\right)$ denotes the $\llparenthesis\cdot\rrparenthesis$-coordinates of the image $\widehat{\F}_\lambda(\widehat\xi)$ of the point $\widehat\xi= \left( \llparenthesis \xbf \rrparenthesis: \xbf\in \{0,1\}^k\right)$.

Recall that for a graph $G\in \GG_k$, the coordinates $(\xbf)$ of the point $\widehat{\phi}_\lambda(G)\in \C^{2^k}$ record the values $Z_G^\xbf(\lambda)$ of the respective $\xbf$-conditioned independence polynomials. In other words, these coordinates encode the partition~\eqref{eq: sum of conditioned} of the sum defining the independence polynomial $Z_G(\lambda)$ according to the assignments $\xbf$ on the marked vertices. In the new coordinates~\eqref{eq: new_coord}, we obtain the same partition, except that the marked vertices receiving the value $1$ are no longer counted in $\llparenthesis \xbf \rrparenthesis$ and therefore contribute an explicit factor, so that
\[
Z_G(\lambda) = \sum_{\xbf=(x_1,\dots,x_k) \in \{0,1\}^k} \lambda^{x_1+ \cdots + x_k}\llparenthesis \xbf \rrparenthesis.
\]

We note that the invariant variety $\MM$ for $\GG_k$ is preserved by the coordinate change~\eqref{eq: new_coord}, i.e., it is defined by the same homogeneous quadratic equations
\begin{equation*}
    \llparenthesis \xbf\rrparenthesis  \cdot \llparenthesis \ybf\rrparenthesis  = \llparenthesis \zbf\rrparenthesis  \cdot \llparenthesis \wbf\rrparenthesis ,
    \end{equation*}
    where $\xbf,\, \ybf, \zbf, \wbf\in \{0,1\}^{k}$ run over all possible binary $k$-tuples satisfying $\xbf+\ybf=\zbf+\wbf$.
    Since for all $\lambda\in \C^*$ the chart $\{(\obf) \neq 0\} \subset \P^{2^k-1}$ is also invariant under the coordinate change, the intersection $\MM\cap \{\llparenthesis\obf\rrparenthesis\neq 0\}$ is still given by the equations
\[
\frac{\llparenthesis\xbf\rrparenthesis}{\llparenthesis\obf\rrparenthesis} = \prod_{j:\; x_j=1} \frac{\llparenthesis\ebf_j\rrparenthesis}{\llparenthesis\obf\rrparenthesis},
\]
for all $\xbf=(x_1,\dots,x_k)\in\{0,1\}^k\setminus \{\obf\}$. Furthermore, it follows from Lemma~\ref{lem: graph on 0-chart} that the restriction of $\F_\lambda$ to $\MM\cap \{\llparenthesis\obf\rrparenthesis\neq 0\}$ is governed by the equations
\begin{equation}\label{eq: image_in_new_coord}
    \llparenthesis\xbf\rrparenthesis'= \llparenthesis\obf\rrparenthesis'\cdot \prod_{j:\; x_j=1}  \left(\frac{\llparenthesis\ebf_{\Lambda(j)}\rrparenthesis}{\llparenthesis\obf\rrparenthesis}\right)^{\#\Phi(j)},
\end{equation}
for all $\xbf=(x_1,\dots, x_k)\in \{0,1\}^k\setminus \{\obf\}$. In particular, the image of the regular part of $\MM\cap \{\llparenthesis\obf\rrparenthesis\neq 0\}$ under $\F_\lambda$ does not depend on $\lambda$. Consequently, when the gluing data $(H,\Phi)$ is pre-fixed, the subvariety $\MM_0=\MM_0(\lambda)\subset \MM$ from Proposition~\ref{prop: pre fixed dynamics} does not depend on $\lambda$ in the $\llparenthesis \cdot \rrparenthesis$-coordinates. In contrast to $\MM$, the subvariety $\MM_0$ is therefore typically not preserved by our change of coordinates; nevertheless, we keep the notation $\MM_0$ in the new coordinates.

\begin{remark}
    The proof of Theorem~\ref{thm: zero-free} could equally well have been carried out in the $\llparenthesis \cdot \rrparenthesis$-coordinates. There is, however, no significant gain: while the subvariety $\MM_0(\lambda)$ becomes independent of $\lambda$ in these coordinates, the relevant hyperplane $\Sigma_0$, on which the independence polynomial $Z_G(\lambda)$ vanishes, now does depend on $\lambda$. Namely, the hyperplane $\Sigma_0=\Sigma_0(\lambda)$ is cut out by the equation
    \[\sum_{\xbf=(x_1,\dots,x_k) \in \{0,1\}^k} \lambda^{x_1+ \cdots + x_k}\llparenthesis \xbf \rrparenthesis=0.\]
\end{remark}

\begin{theorem}\label{thm: bounded}
      Let $(H,\Phi)$ be a non-degenerate and expanding gluing data with parameters $k\geq 2, m\geq 2$, and let  $\RR=\RR_{(H,\Phi)}$ be the associated graph recursion on $\GG_k$. Fix an arbitrary maximally independent starting graph $G_0\in \GG_k$, and consider the induced recursive graph sequence $(G_n)_{n\geq 0}$ with $G_{n+1}=\RR(G_n)$. Then the zeros of the independence polynomials $Z_{G_n}$ are uniformly bounded.
\end{theorem}
\begin{proof}
    First note that all graphs in the sequence $(G_n)_{n\geq 0}$ are maximally independent; this easily follows from the recursive construction by induction on $n$. It is therefore sufficient to prove the statement when the gluing data $(H,\Phi)$ is pre-fixed, which we will assume in the following.

    We work in the $\llparenthesis \cdot \rrparenthesis$-coordinates~\eqref{eq: new_coord}. Let us first consider what happens to the point $\phi_\lambda(G_0) \in \P^{2^k-1}$ as $|\lambda|\to \infty$. 
    
    \begin{claim}
    As $|\lambda| \rightarrow \infty$, the point $\phi_\lambda(G_0)$ converges to the point $[1: \cdots : 1]\in \MM_0$ in the $\llparenthesis\cdot \rrparenthesis$-coordinates. 
    \end{claim}

    Indeed, the $\llparenthesis \xbf\rrparenthesis$-coordinates of $\widehat{\phi}_\lambda(G_0)\in \C^{2^k}$ are given by $\lambda^{-(x_1+\dots+x_k)} Z^\xbf_{G_0}(\lambda)$, where $\xbf=(x_1,\dots, x_k)\in \{0,1\}^k$, and thus they are all monic polynomials in $\lambda$ of equal degrees, because $G_0$ is assumed to be maximally independent. Hence $\phi_\lambda(G_0)$ converges to the point $[1: \cdots : 1]\in \P^{2^k-1}$ as $|\lambda| \rightarrow \infty$. It remains to check that $[1: \cdots : 1]\in \MM_0$. Note that $[1: \cdots : 1]\in \MM$ and $[1: \cdots : 1]$ is fixed under $\F_\lambda$ for all $\lambda\in \C^*$ such that $[1: \cdots : 1]\notin \IS(\F_\lambda)$; see equations~\eqref{eq: image_in_new_coord}. We conclude that $[1: \cdots : 1]\in \F_\lambda(\MM\setminus\IS(\F_\lambda)) \subset \MM_0$ for all such $\lambda$, and thus, since $\MM_0$ is independent of $\lambda$, we have $[1: \cdots : 1]\in \MM_0$ for all $\lambda\in \C^*$. The claim follows.\\

    Let us now consider what happens to the action of $\F_\lambda$ as $|\lambda| \rightarrow \infty$. By~\eqref{eq: x_prime_new_coord}, every coordinate function of $\F_\lambda$ is a polynomial in $\lambda$, with the leading term corresponding to the single assignment $\eta=\ibf_Y$. (Note that $\#Y\geq 1$ as $(H,\Phi)$ is non-degenerate.) Therefore, as $|\lambda|\to \infty$ the maps $\frac{1}{\lambda^{\#Y}}{\widehat{\F}}_\lambda$ converge to the monomial map $\widehat{\F}_\infty$ on $\C^{2^k}$ given by 
     \[\llparenthesis \xbf\rrparenthesis' = \prod_{i=1}^m\llparenthesis[\xbf\wedge \ibf_Y]_i\rrparenthesis, \quad \xbf\in \{0,1\}^k.\]
     Consequently, the rational maps $\F_\lambda: \P^{2^k-1} \dashrightarrow \P^{2^k-1}$ converge to the respective rational self-map $\F_\infty: \P^{2^k-1} \dashrightarrow \P^{2^k-1}$ as $|\lambda| \rightarrow \infty$, where convergence is uniform on compact subsets of the complement of the indeterminacy set $\IS(\F_\infty)$. 
     
     Since each $F_\lambda$ holomorphically retracts $\MM$ to $\MM_0$ and is transversally superattracting there (see Proposition~\ref{prop: pre fixed dynamics} and Theorem~\ref{thm: superattraction}), the same is true for $\F_\lambda$, where now both $\MM$ and $\MM_0$ are independent of $\lambda$ in the $\llparenthesis \cdot\rrparenthesis$-coordinates. Passing to the limit as $|\lambda|\to \infty$, we obtain that (away from the indeterminacy set) $\F_\infty$ is also a holomorphic retraction of $\MM$ to $\MM_0$, and is transversally superattracting on $\MM_0$.

    Let us now fix a small neighborhood $U$ of $[1: \cdots : 1]\in \MM_0\subset  \P^{2^k-1}$ whose closure avoids $\IS(\F_\infty)$.  We may assume that, as long as $|\lambda|$ is sufficiently large, the value
    \[
   \sum_{\xbf=(x_1,\dots,x_k)  \in \{0,1\}^k} \lambda^{x_1+ \cdots + x_k}\llparenthesis \xbf \rrparenthesis \neq 0
    \]
    for all $\xi\in U$ with coordinates $\llparenthesis \xbf \rrparenthesis$, $\xbf\in \{0,1\}^k$. Indeed, the term $\lambda^k\llparenthesis 1,\dots, 1 \rrparenthesis$ dominates the others on $U$ by a factor of order $|\lambda|$.

    By above, the maps $\F_\lambda$ converge uniformly to $\F_\infty$ on $U$ as $|\lambda|\to \infty$.  Moreover, by passing to a sub-neighborhood $U'\subset U$ of $[1:\cdots : 1]$, we may assume that, as long as $\lambda$ lies in a sufficiently small neighborhood $D'(\infty)$ of $\infty\in \widehat{\C}$, the orbits of points in $U'$ under $\F_\lambda$ remain in~$U$; compare the proof of Theorem~\ref{thm: zero-free}. 

    Using now the claim, we have that $\xi_0(\lambda)=\phi_\lambda(G_0)\in U'$ for all $\lambda$ with sufficiently large $|\lambda|$, and thus the orbit 
    \[\xi_n(\lambda):=\phi_\lambda(G_n)=\F_\lambda^n(\xi_0(\lambda))\]
    stays within $U$. In particular, the independence polynomials $Z_{G_n}(\lambda)$ do not vanish as long as $|\lambda|$ is sufficiently large. This finishes the proof.
\end{proof}

By combining the boundedness of the zeros (Theorem~\ref{thm: bounded}) with the fact that the zeros avoid a neighborhood of $\R_{\geq0}$ (Theorem~\ref{thm: zero-free}), we obtain the following immediate corollary.

\begin{corollary}\label{cor: zero-free-cone}
    For non-degenerate and expanding graph recursion operators $\RR$ on $\GG_k$ with $k\geq 2$ and a maximally independent starting graph $G_0\in 
    \GG_k$, the zeros of the independence polynomials of the graphs $G_n=\RR^n(G_0)$ avoid a uniform cone around $\R_{\geq0}$.
\end{corollary}

To complete the discussion, we provide an example that demonstrates that whether the zeros of the independence polynomials for a recursive sequence $(G_n)_{n\geq 0}$ of graphs are uniformly bounded depends in general on the starting graph $G_0$.

\begin{example}[Chebyshev gluing data]\label{ex: chebyshev}
Consider the gluing data $(H, \Phi)$ with parameters $k=2$ and $m=2$ defined as follows. The gluing scheme $H$ with $V(H)=\{1,2\}$ has exactly three edges: edges $\{1\}$ and $\{2\}$ labeled $1$, and an edge $\{1,2\}$ labeled $2$. The corresponding labeling map $\Phi$ is given by $\Phi(1)=\{1\}$ and $\Phi(2)=\{2\}$. In other words, given $G\in \GG_2$, the graph $\RR_{(H,\Phi)}(G)$ is constructed by taking two copies $G(1)$ and $G(2)$ of $G$, identifying in them the two marked vertices labeled $2$, and relabeling the two vertices with old label $1$. Clearly, this recursion operator $\RR=\RR_{(H,\Phi)}$ is both expanding and non-degenerate.

When $G_0\in \GG_2$ is a single edge, the graphs $G_n=\RR^n(G_0)$ are paths with $2^n$ edges. The zeros of the independence polynomial of these graphs all lie on the negative real axis, but are not uniformly bounded. It follows that the recursion operator being expanding and non-degenerate is not sufficient for the zeros to remain bounded.

Now suppose that $G_0\in \GG_2$ is a $3$-star where two of the three leaves are labeled. This graph is not maximally independent, but one easily checks that the graph $G_1 = \RR(G_0)$ is however maximally independent. Thus by Theorem~\ref{thm: bounded} the zeros of the independence polynomial for the sequence $G_n=\RR^n(G_0)$ are uniformly bounded.
\end{example}

\bibliographystyle{alpha}
\bibliography{article}

\end{document}